\documentclass[11pt]{amsart}
\usepackage{graphicx,color}
\usepackage{euscript}
\usepackage{amsmath}
\usepackage{amsfonts}
\usepackage{amssymb}
\usepackage{amsthm}
\usepackage{hyperref}
\usepackage{epsfig}
\usepackage{epsf}
\usepackage{epstopdf}
\usepackage{caption}
\usepackage{url}
\usepackage{array}
\usepackage{amssymb}\usepackage{amscd}
\usepackage{epic}
\usepackage{eufrak}
\usepackage{tikz,underscore,pgfplots}
\usepackage{tikz-cd}
\usepackage{bm}
\usepackage{float}
\usepackage{anysize}
\usepackage{amsfonts}
\usepackage{pinlabel} 

\newtheorem{theorem}{Theorem}[section] 
\newtheorem{corollary}[theorem]{Corollary} 
\newtheorem{lemma}[theorem]{Lemma} 
\newtheorem{proposition}[theorem]{Proposition} 
\newtheorem{example}[theorem]{Example}

\newtheorem{problem}[theorem]{Problem}  
\newtheorem{remark}[theorem]{Remark}  
\newtheorem{definition}[theorem]{Definition}

\newcommand{\ch}{\mathrm{char}}

\newcommand{\F}{\mathbb{F}}
\newcommand{\C}{\mathbb{C}}
\newcommand{\CC}{\mathcal{C}}
\newcommand{\cCFL}{\mathcal{C\!F\!L}}

\newcommand{\cHFL}{\mathcal{H\!F\! L}}
\newcommand{\CFL}{\mathit{CFL}}
\newcommand{\HFL}{\mathit{HFL}}
\newcommand{\HFK}{\mathit{HFK}}

\newcommand{\Cone}{\mathrm{Cone}}

\newcommand{\s}{\mathfrak{s}}
\newcommand{\ee}{\mathbf{e}}
\newcommand{\ff}{\mathbf{f}}

\newcommand{\sgn}{\mathrm{sgn}}
\newcommand{\lk}{\mathrm{lk}}
\newcommand{\Split}{\mathrm{split}}
\newcommand{\Spin}{\text{Spin}^{c}}
\newcommand{\alphas}{\boldsymbol{\alpha}}
\newcommand{\betas}{\boldsymbol{\beta}}
\newcommand{\gammas}{\boldsymbol{\gamma}}
\newcommand{\deltas}{\boldsymbol{\delta}}
\newcommand{\ovl}{\overline}

\newcommand{\cbHFL}{\boldsymbol{{\mathcal{H\!F\!L}}}}
\newcommand{\cbCFL}{\boldsymbol{\mathcal{C\!F\!L}}}
\newcommand{\bJ}{\boldsymbol{\J}}

\newcommand{\Imm}{\mathrm{Im}}
\newcommand{\cK}{\mathcal{K}}

\newcommand{\Alt}{\mathrm{Alt}}
\newcommand{\HY}{\mathrm{HY}}

\def\L{\mathcal{L}}
\def\gr{\textup{gr}}
\def\w{{\bf{w}}}
\def\z{{\bf{z}}}
\def\d{\mathbf{d}}
\def\x{\mathbf{x}}
\def\y{\mathbf{y}}

\def\k{\mathbf{k}}
\def\m{\mathbf{m}}
\def\bF{\mathcal{F}}

\newcommand{\HHH}{\mathrm{HHH}}

\newcommand{\HH}{\mathbb{H}}
\newcommand{\Z}{\mathbb{Z}}
\newcommand{\R}{\mathbb{R}}
\newcommand{\UU}{\mathbf{U}}
\newcommand{\A}{\mathfrak{A}}

\newcommand{\J}{\mathcal{J}}

\newtheorem{convention}[equation]{Convention}

\author{Akram Alishahi}
\address{A. A. :Department of Mathematics, University of Georgia\\ Athens, GA 30602 }
\email{akram.alishahi@uga.edu}

\author{Eugene Gorsky}
\address{E. G.: Department of Mathematics, University of California Davis, One Shields Avenue, Davis CA 94702, USA}
\email{egorskiy@math.ucdavis.edu}

\author{Beibei Liu}
\address{B.L.: Department of Mathematics, The Ohio State University, 100 Math Tower, 231 West 18th Avenue, Columbus, OH, 43210, USA}
\email{bbliumath@gmail.com}

\title[Splitting maps and integer points]{Splitting maps in link Floer homology and integer points in permutahedra}

\begin{document}

\begin{abstract}
In this paper, we study the skein exact sequence for links via the exact surgery triangle of link Floer homology and compare it with other skein exact sequences given by Ozsv\'ath and Szab\'o. As an application, we use the skein exact sequence to study the splitting number and splitting maps for links. In particular, we associate the splitting maps for the torus link $T(n, n)$ to integer points in the $(n-1)$-dimensional permutahedron, and obtain the link Floer homology of an $n$-component homology nontrivial unlink in $S^{1}\times S^{2}$. 

\end{abstract}

\maketitle

\tableofcontents

\section{Introduction}

In this paper, we study various maps in Heegaard Floer  homology associated to crossing changes in link diagrams. Given such a diagram  with a chosen crossing, we can consider three links $L_+$, $L_-$ and $L_0$ in the three-sphere corresponding to the positive crossing, negative crossing and oriented resolution (see Figure \ref{crossings}). We will always assume that the crossing is between different components of $L_{\pm}$, so that $L_+$ and $L_-$ both have one more component than $L_0$.

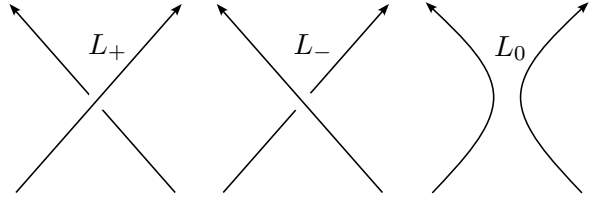
\begin{figure}[H]
\centering
\begin{tikzpicture}
    \node[anchor=south west,inner sep=0] at (0,0) {\includegraphics[width=3.0in]{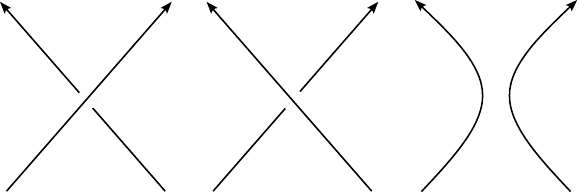}};
    \node[label=above:{$L_+$}] at (1.3,1.5){};
    \node[label=above:{$L_-$}] at (4,1.5){};
    \node[label=above:{$L_0$}] at (6.6,1.5){};
\end{tikzpicture}
\caption{From left to right: positive crossing, negative crossing and oriented resolution.} \label{crossings}
\end{figure}

For various technical reasons, we work with the ``full" version $\cHFL$ of the Heegaard Floer homology with two marked points on each link component over $\F=\Z/2\Z$, developed in \cite{Zemke,Zemke2}. In particular, for a link $L$ with $n$ components  in $S^{3}$, $\cHFL(L)$ is a $\Z\oplus \Z^n$-graded module over $\F[U_1,\ldots,U_n,V_1,\ldots,V_n]$ where all products $U_iV_i$ act by the same operator which we will denote by $\UU$. Sometimes we will need to work with the completion $\cbHFL(L)$ which is a module over the power series ring $\F[[U_1,\ldots,U_n,V_1,\ldots,V_n]]$.

Our first result describes the crossing change maps in this version of Heegaard Floer homology generalizing the maps in \cite{OS2,OS3,OS4} for $\widehat{\HFK}$ and $\HFK^-$.

\begin{theorem}
\label{thm: one crossing intro}
Given a crossing between the components $L_i$ and $L_j$ of an oriented link in the three-sphere, for all $k\in \mathbb{Z}$ corresponding to $\Spin$-structures in certain surgery cobordism shown in Figure \ref{blowup},   there are maps 
\begin{equation}
\label{eq: skein intro}
\psi_k:\cHFL(L_+)\to \cHFL(L_-)\ \mathrm{and}\ \phi_k:\cHFL(L_-)\to \cHFL(L_+)
\end{equation}
satisfying the following equations:
\begin{itemize}
\item[(a)] The maps $\psi_k$ are determined by $\psi_0$ and $\psi_{-1}$:
$$\psi_k=(V_iU_j)^k\UU^{\frac{k(k-1)}{2}}\psi_0\quad \mathrm{for}\ k\ge 0,\quad \psi_k=(V_jU_i)^{-1-k}\UU^{\frac{(k+1)(k+2)}{2}}\psi_{-1}\quad \mathrm{for}\ k\le -1.$$

\item[(b)] We have $V_j\psi_0=V_i\psi_{-1}$ and $U_i\psi_0=U_j\psi_{-1}$.

\item[(c)] The maps $\phi_k$ are determined by $\phi_0$ and $\phi_1$:
$$\phi_k=(U_iU_j)^{k-1}\UU^{\frac{(k-1)(k-2)}{2}}\phi_1\quad \mathrm{for}\ k\ge 1,\quad \phi_k=(V_iV_j)^{-k}\UU^{\frac{k(k+1)}{2}}\phi_{0}\quad \mathrm{for}\ k\le 0.$$

\item[(d)]The maps $\psi_k$ and $\phi_k$ compose as follows:
$$
\phi_0\psi_0=V_i,\ \phi_0\psi_{-1}=V_j, \ \phi_1\psi_0=U_j, \ \phi_1\psi_{-1}=U_i
$$
$$
\psi_0\phi_0=V_i,\ \psi_{-1}\phi_0=V_j,\ \psi_0\phi_1=U_j,\ \psi_{-1}\phi_1=U_i,
$$
The rest of compositions are determined by these.
\end{itemize}
\end{theorem}

See Section \ref{sec: surgery} for more details and the gradings for all these maps.
In \cite{OS4} Ozsv\'ath and Szab\'o proved a skein exact triangle for $\HFK^-$
$$
\rightarrow \HFK^-(L_+)\rightarrow\HFK^-(L_-)\rightarrow \HFK^-(L_0)\otimes W\rightarrow\ldots
$$
where $W$ is some given bigraded module.

We generalize this as follows.
\begin{theorem}
\label{thm: intro skein}
Given a crossing between the components $L_i$ and $L_j$, there is an exact triangle
\begin{equation}
\label{eq: intro skein}
\rightarrow \cbHFL(L_+)\xrightarrow{\Psi_{ij}}\cbHFL(L_-)\rightarrow H_*(\cbCFL(L_0)\otimes \boldsymbol{\mathcal{K}})\rightarrow\ldots
\end{equation}
where the map $\Psi_{ij}: \cbHFL(L_+)\to \cbHFL(L_-)$ is given by $\Psi_{ij}=\sum_{k\in \Z}(-1)^{k}\psi_k$, and $\boldsymbol{\mathcal{K}}$ is the completion of the module $\mathcal{K}$ defined in \eqref{def:K}. 
\end{theorem}

Theorem \ref{thm: one crossing intro} implies the following:
\begin{corollary}
\label{cor: tau intro}
We have $\Psi_{ij}=\tau(\psi_0-\psi_{-1})$ where $\tau=1+\ldots$ is an explicit invertible power series in $\F[[U_1,\cdots,U_n,V_1,\cdots,V_n]]$ defined in Lemma \ref{lem: tau}. In particular, the cones of $\Psi_{ij}$ and of $\Psi_{ij}^{0}=\psi_0-\psi_{-1}$ are homotopy equivalent.
\end{corollary}

Note that the map $\Psi_{ij}^{0}$ has homological degree $0$ and can be defined without completion.

\begin{remark}
Since we work over the field $\F=\Z/2\Z$, the signs here and below are purely for esthetic reasons. However, we expect all the maps to exist for theories with integer coefficients (similar to \cite{AE-minus}), and conjecture that (up to an overall normalization) the signs would match. See also Section \ref{sec: J cousins} on comparison of the signs with triply graded Khovanov-Rozansky homology.
\end{remark}

Next, we study the compositions of crossing change maps. Since we have two essentially different maps $\psi_0$ and $\psi_{-1}$ (resp. $\phi_0$ and $\phi_1$) for a single crossing change, for a sequence of $r$ crossing changes we have $2^r$ possible associated maps in Heegaard Floer homology of various degrees, some of which may coincide. We determine the degrees of all such maps in Section \ref{sec: splitting maps} and use them to bound splitting numbers for links.

In a striking example, we can take the $n$-component torus link $T(n,n)$, change $\binom{n}{2}$ crossings between different components from positive to negative and obtain the unlink $O_n$. In this case, we are able to completely determine {\bf all} crossing change maps.

\begin{theorem}
\label{thm: intro permutahedron}
If one chooses either $\psi_0$ or $\psi_{-1}$ for each of $\binom{n}{2}$ crossing changes from $T(n,n)$ to $O_n$, the Alexander degrees of the resulting maps correspond to integer points in the permutahedron $P_n$. Any two maps of the same degree coincide, and any integer point in $P_n$ corresponds to an injective map $\cHFL(T(n,n))\to \cHFL(O_n)$ which can be described explicitly on generators of $\cHFL(T(n,n))$.
\end{theorem}

For example, $P_3$ is a hexagon with 6 vertices and 1 interior point, see Figure \ref{fig:polytopes}. To get from $T(3,3)$ to unlink, one needs to change 3 crossings, so there are $2^3=8$ possible splitting maps. Six of them correspond to the vertices of $P_3$ and two remaining ones coincide and correspond to the interior point of $P_3$. We generalize Theorem \ref{thm: intro permutahedron} to arbitrary L-space links in Section \ref{sec: L space splitting}.

\begin{theorem}
Suppose that $L$ is an L-space link. Then:

a) For any choice of crossing changes and the maps $\psi_k,\phi_k$ at the crossings, the resulting map $F:\cHFL(L)\to \cHFL(\Split(L))$ is completely determined by its Alexander and Maslov degrees. 

b) If, in addition, all crossings between the different components of $L$ are positive, the splitting maps  are in bijection with the integer points in a certain  polytope $P_L$ (see Definition \ref{def: zonotope}).
\end{theorem}

One can also study the compositions of maps $\Psi_{ij}$ from skein exact sequence \eqref{eq: intro skein} for crossings in $T(n,n)$ between $L_i$ and $L_j$. Let $\J$ be the ideal in $\cHFL(O_n)$ generated by determinants 
$$
\Delta_S=\det\left(
\begin{matrix}
U_1^{a_1}V_1^{b_{1}} & \cdots & U_1^{a_n}V_1^{b_{n}} \\
\vdots & & \vdots\\
U_n^{a_1}V_n^{b_{1}} & \cdots & U_n^{a_n}V_n^{b_{n}} \\
\end{matrix}
\right)
$$ for all possible $n$-element subsets 
$S=\{(a_1,b_1),\ldots,(a_n,b_n)\}\subset \Z_{\ge 0}\times \Z_{\ge 0}$. 
We denote by $\bJ$ the completion of $\J$ in $\cbHFL(O_n)$.

\begin{theorem}
\label{thm: Tnn dets intro}
a) Let $\Omega:\cbHFL(T(n,n))\to \cbHFL(O_n)$ be the composition of the maps $\Psi_{ij}$ from Theorem \ref{thm: intro skein} over all $i<j$. Then $\Omega$ is injective and its image is the ideal  $\bJ$ in $\cbHFL(O_n)$.

b) Let $\Omega^0:\cHFL(T(n,n))\to \cHFL(O_n)$ be the composition of the maps $\Psi^0_{ij}$ from Corollary \ref{cor: tau intro} over all $i<j$. Then $\Omega^0$ is injective and its image is the ideal  $\J$ in $\cHFL(O_n)$.
\end{theorem}

\begin{corollary}
We have $\cHFL(T(n,n))\simeq \J$ as modules over $\F[U_1,\ldots,U_n,V_1,\ldots,V_n]$.
\end{corollary}

Theorem \ref{thm: Tnn dets intro} can be compared with the main result of \cite{GHog} where the ``$y$-ified" triply graded Khovanov-Rozansky homology (also known as HOMFLY homology) of $T(n,n)$ was computed using a very similar ideal to $\J$, see Section \ref{sec: J cousins}.
This suggests a spectral sequence from the ``$y$-ified'' HOMFLY homology to $\cHFL$ which we plan to study in a future work. Such a spectral sequence should generalize the spectral sequences for reduced homology studied in \cite{BPRW,Dowlin,Gilmore}

Finally, we can use the above results to compute the Heegaard Floer homology of certain links in $S^1\times S^2$.

\begin{theorem}
\label{thm: intro Zn}
Let $Z_n$ be the link consisting of $n$ parallel copies of $S^1$ inside $S^1\times S^2$. Then $\cbHFL(S^1\times S^2,Z_n)\simeq \bJ/(\gamma)$ where
$$
\gamma=\mu_0\prod_{i<j}(V_i-V_j)+\mu_{n-1}\prod_{i<j}(U_i-U_j)+
\sum_{j=1}^{n-2} \det\left(\begin{matrix}
U_1^{j} & \cdots & U_1 & 1 & V_1 & \cdots & V_1^{n-1-j}\\
\vdots &  &\vdots & \vdots & \vdots & & \vdots \\
U_n^{j} & \cdots & U_n & 1 & V_n & \cdots & V_n^{n-1-j}
\end{matrix}
\right).
$$
and 
$$
\mu_0=\sum_{k=0}^{\infty}(V_1\cdots V_n)^{k}\UU^{\frac{k(k-1)}{2}},\quad \mu_{n-1}=\sum_{k=0}^{\infty}(U_1\cdots U_n)^{k}\UU^{\frac{k(k-1)}{2}}.
$$
\end{theorem}

\section*{Acknowledgments}

We are grateful to Daren Chen, Matthew Hedden, Peter Kronheimer, Tye Lidman, Robert Lipshitz, Ciprian Manolescu, Lisa Piccirillo and Ian Zemke for useful discussions. A. A. and E. G. were partially supported by  the NSF
grant DMS-1928930 while they were in residence at the Simons Laufer Mathematical Sciences
Institute (previously known as MSRI) in Berkeley, California, during the Fall 2022 semester. A. A. was also partially supported by NSF grants DMS-2000506 and DMS- 2238103. E. G. was also partially supported by the NSF grant DMS-1760329. B. L. is partially supported by the NSF grant DMS-2203237. 

\section{Background}

\subsection{Lattices}

We will work with the lattice $\Z^n$ and its translates. We define a partial order on $\Z^n$ by
$$
\mathbf{u}\preceq \mathbf{v}\ \Leftrightarrow\ u_i\le v_i\ \text{for all}\ i.
$$
We will denote the basis vectors by $\ee_i=(0,\ldots0,1,0,\ldots,0)$. Given a vector $\mathbf{k}=(k_1,\ldots,k_n)\in \Z^n$, and a set of variables $U_1,\ldots, U_n$ (resp. $V_1,\ldots V_n$), we write
$$
U^{\mathbf{k}}=U_1^{k_1}\cdots U_n^{k_n},\quad V^{\mathbf{k}}=V_1^{k_1}\cdots V_n^{k_n}.
$$

\subsection{Variables and gradings}
We will be working with links in $S^3$ and the ``full" version of Heegaard Floer complex $\cCFL$ for links, defined in \cite{Zemke}. The coefficients are in $\F=\Z/2\Z$. 
Let $L=L_1\cup\ldots\cup L_n$ be an oriented link with $n$ components. Unless stated otherwise, we will assume that each component $L_i$ has exactly two marked points $z_i$ and $w_i$. The corresponding link homology $\cHFL(L)$ is a module over the polynomial ring $R=\F[U_1, \cdots, U_n, V_1, \cdots, V_n]$. We let $R_{UV}$ denote the ring  in variables $U_1,\ldots,U_n,V_1,\ldots,V_n,\UU$ satisfying the relations  
$$
U_1V_1=\ldots=U_nV_n=\UU.
$$  
The actions of $U_iV_i$ on the complex $\cCFL(L)$ are pairwise homotopic, and the action of $R$ on $\cHFL(L)$ factors through $R_{UV}$. 

Further, define
 \[\cCFL^\infty(L):=\cCFL(L)\otimes_R\F[U_1,U_1^{-1},\cdots, U_n,U_n^{-1}, V_1,V_1^{-1},\cdots,V_n,V_{n}^{-1}]\]
 and $\cHFL^\infty(L):=H_*(\cCFL^{\infty}(L))$.

We denote by $\lk(L_i,L_j)$ the linking number between the components $L_i$ and $L_j$, and write $\ell_i=\sum_{j\neq i}\lk(L_i,L_j)$. Moreover, we let $\ell_{L}=\dfrac{1}{2}(\ell_1, \cdots, \ell_n)$.  

The link Floer homology has an {\bf Alexander grading} $A=(A_1,\ldots,A_n)$ valued in the lattice
$$
\HH_L=\Z^n+\frac{1}{2}(\ell_1,\ldots,\ell_n).
$$
It also has a {\bf homological} (or {\bf Maslov}) grading $\gr_{\w}$ and an additional grading $\gr_{\z}$ satisfying
$$
A_1+\ldots+A_n=\frac{1}{2}(\gr_{\w}-\gr_{\z}).
$$
Thanks to the relation between $A$, $\gr_{\w}$ and $\gr_{\z}$, we can determine $\gr_{\z}$ from the Alexander and Maslov gradings. 
Note that the differential on the chain complex $\cCFL(L)$ preserves Alexander multi-grading and changes the Maslov grading by 1. So, for any $\k\in\HH_{L}$, let $\cCFL(L,\k)$ denote the subcomplex of $\cCFL(L)$ generated by the elements of $A(x)=\k$. The variable $U_i$ decreases $A_i$ by $1$, decreases $\gr_{\w}$ by $2$ and preserves $\gr_{\z}$, while the variable $V_i$ increases $A_{i}$ by $1$, preserves $\gr_{\w}$ and decreases $\gr_{\z}$ by $2$. Therefore, the coefficient ring for the subcomplex $\cCFL(L,\k)$ is the subring $\F[U_1V_1,U_2V_2,\cdots,U_nV_n]$ and so $\cHFL(L,\k)$ is an $\F[\UU]$-module.

For example, the homology of the unlink with $n$ components has one generator in Alexander degree $(0,\ldots,0)$ and Maslov degree 0, and is isomorphic to the ground ring $R_{UV}$. 

Sometimes we will need to work with the completion $\cbHFL(L)$ which is a module over the power series ring $\F[[U_1,\ldots,U_n,V_1,\ldots,V_n]]$.

\subsection{Specializing $V_i$}

We will need to compare the above construction of Heegaard Floer homology with more ``classical" ones \cite{OS1,OS2,MO}. This is done by specializing $V_i$ in various ways.

First, we specialize $V_i=1$ for all $i$ and denote the specialized complex by $\CFL^-$ following \cite{OS2}. The specialized complex still has commuting actions of $U_i$, which are all homotopic to $\UU$. Since $\gr_{\w}(V_i)=0$, the specialized complex has a homological grading given by $\gr_{\w}$. On the other hand, the Alexander grading becomes {\bf Alexander filtration}, as follows:

\begin{proposition}\label{prop:filter}
For all $\k\in \HH_L$,
there is a bijection between the generators of $\CFL^-$ of Alexander grading $\preceq \k$ and the generators of $\cCFL$ of Alexander grading exactly $\k$ i.e. generators of $\cCFL(L,\k)$:
$$
x\leftrightarrow V^{\k-A(x)}x,\ A(x)\preceq \k.
$$ 
The span of such generators, denoted by $\A^-(\k)=\A^-(L;\k)$, is a subcomplex of $\CFL^-$, and such subcomplexes yield a $\Z^n$-filtration on $\CFL^-$.
\end{proposition}

Another specialization is $V_i=0$ for all $i$. Similarly to the above, one immediately verifies that this is equivalent to considering the associated graded complex $\gr \CFL^-$ with respect to the Alexander filtration.

\subsection{Large surgery and L-space links}

We recall the large surgery theorem of Manolescu and Ozsv\'ath:

\begin{theorem}[\cite{MO}]
Let $L=L_1\cup \ldots \cup L_n$ be an $n$-component link in the three-sphere. For $\mathbf{d}=(d_1,\ldots,d_n)\in\Z^n$ denote the $3$-manifold obtained by performing $d_i$-surgery on $L_i$ for all $1\le i\le n$ by 
$$
Y_{\d}=S^3_{\d}(L).
$$
Then, for $d_i\gg 0$ and arbitrary $\k$ we have an isomorphism of graded $\F[\UU]$-modules, up to a grading shift:  
$$
\A^-(L;\k)\simeq \mathit{CF}^-(Y_{\d};\s_{\k})
$$ 
where $\s_{\k}$ is a $\Spin$-structure on $Y_{\d}$ determined by $\k$.
\end{theorem}

\begin{corollary}
As a graded $\F[\UU]$-module,  the homology $\cHFL(L,\k)$ splits as a direct sum of one copy of $\F[\UU]$ and some $\UU$-torsion.\end{corollary}

The link invariant $h(\k)$, known as the $h$-function, is defined as the $-\frac{1}{2}\gr_{\w}$ for the generator of the $\F[\UU]$-summand of $\cHFL(L,\k)$.

An oriented, connected, closed 3-manifold $M$ is an \emph{L-space} if it is a rational homology sphere, and for each $\Spin$-structure $\mathfrak{s}$ on $M$, one has $HF^{-}(M, \mathfrak{s})\cong \F[\UU]$.  A link $L$ is called an \emph{L-space link} if $S^3_{\d}(L)$ is an L-space for $\d\gg 0$.   Since Dehn surgery does not depend on the orientations of the link, being an L-space link is independent of the orientations on the components of the link.

\begin{corollary}
\label{cor: L space tower}
For an L-space link, we have  $\cHFL(L,\k)=\F[\UU][-2h(\k)]$ for all $\k$.
\end{corollary}

This is a useful way to characterize L-space links \cite{Yajing}. That is, a link $L\subset S^{3}$ is an L-space link if and only if the link Floer homology $\cHFL(L)$ is torsion free as an $\F[\UU]$-module. 

\begin{example}
Let $O_n$ be the unlink with $n$ components. Then, in Alexander grading $\k=(k_1,\ldots,k_n)$ we have
$$
\cCFL(L,\k)=\prod_{i=1}^{r}U_i^{-[k_i]_{-}}V_i^{[k_i]_{+}}\cdot \F[\UU],
$$
where $[k]_{+}=\max(k,0)$ and $[k]_{-}=\min(k,0)$. Note that $\prod_{i=1}^{r}U_i^{-[k_i]_{-}}V_i^{[k_i]_{+}}$ has homological degree $\gr_{\w}=2\sum_{i=1}^{k}[k_i]_{-}$, so $h(\k)=-\sum_{i=1}^{k}[k_i]_{-}$.
\end{example}

\begin{example}\label{ex:L-space links}
The Hopf link $T(2, 2)$ with linking number $1$ and the negative Hopf link $-T(2, 2)$ with linking number $-1$ are both L-space links. In \cite{BLZ}, there are explicit computations of $\cHFL$ of these two links using the Heegaard diagrams  in Figure \ref{HD-Hopf}. The link Floer chain complex of the positive Hopf link $T(2, 2)$ is a module over $R$ given as follows:
$$\partial a=\partial b=0, \quad \partial c=U_1 a+V_2 b, \quad \partial d=U_2 a+V_1 b$$
where the gradings of $a, b$ are the following:
$$
A(a)=\left(\frac{1}{2},\frac{1}{2}\right),\ \gr_{\w}(a)=0,\quad A(b)=\left(-\frac{1}{2},-\frac{1}{2}\right),\ \gr_{\w}(b)=-2.
$$
The gradings of $c, d$ are as follows 
$$\gr_{\w}(c)=\gr_{\w}(d)=\gr_{\z}(c)=\gr_{\z}(d)=-1.$$
Hence, the  full  homology of the Hopf link $T(2,2)$ is generated by $a, b$ and can be written as 
$$
\cHFL(T(2,2))=\frac{R\langle a,b\rangle}{U_1a=V_2b,U_2a=V_1b}.
$$
 
The link Floer chain complex of the negative Hopf link is the dual complex of  $\cCFL(T(2, 2))$, i.e., 
$$\partial c'=\partial d'=0, \quad, \partial a'=U_1 c'+U_2 d', \quad \partial b'=V_1 d'+V_2 c' $$
where the gradings of $c', d'$  are
$$\gr_{\w}(c')=\gr_{\z}(c')=\gr_{\w}(d')=\gr_{\z}(d')=1, \quad A(c')=\left(\dfrac{1}{2}, -\dfrac{1}{2}\right), \quad A(d')=\left(-\dfrac{1}{2}, \dfrac{1}{2}\right).$$
Hence, the full homology of $-T(2, 2)$ can be written as 
$$
\cHFL(-T(2,2))=\frac{R\langle c', d'\rangle}{U_1c'=U_2d', V_2c'=V_1d'}.
$$

\begin{figure}[H]
\centering
\begin{tikzpicture}
    \node[anchor=south west,inner sep=0] at (0,0) {\includegraphics[width=3in]{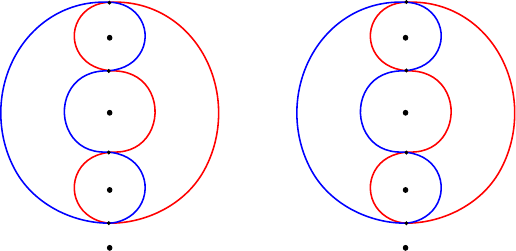}};
    \node[label=below:{\tiny{$\w_2$}}] at (1.7,0.2){};
      \node[label=right:{\tiny{$\w_1$}}] at (1.42,0.9){};
       \node[label=right:{\tiny{$\z_2$}}] at (1.42,2.03){};
       \node[label=right:{\tiny{$\z_1$}}] at (1.42,3.15){};
        \node[label=right:{\tiny{$a$}}] at (1.3,0.28){};
        \node[label=right:{\tiny{$c$}}] at (1.3,1.6){};
        \node[label=right:{\tiny{$b$}}] at (1.3,2.5){};
        \node[label=right:{\tiny{$d$}}] at (1.3,3.85){};

 \node[label=below:{\tiny{$\z_2$}}] at (6.1,0.2){};
      \node[label=right:{\tiny{$\w_1$}}] at (5.82,0.9){};
       \node[label=right:{\tiny{$\w_2$}}] at (5.82,2.03){};
       \node[label=right:{\tiny{$\z_1$}}] at (5.82,3.15){};
       \node[label=right:{\tiny{$c'$}}] at (5.7,0.28){};
        \node[label=right:{\tiny{$a'$}}] at (5.7,1.65){};
        \node[label=right:{\tiny{$d'$}}] at (5.7,2.5){};
        \node[label=right:{\tiny{$b'$}}] at (5.7,3.85){};

\end{tikzpicture}
\caption{Left: genus $0$ Heegaard diagram for $T(2,2)$, Right: genus $0$ Heegaard diagram for $-T(2,2)$} \label{HD-Hopf}
\end{figure}
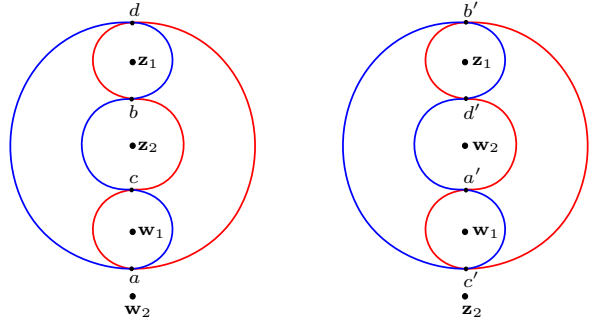

\end{example}

\subsection{Cobordism maps and link TQFT}

We first review \emph{3-manifolds with multi-based links} and \emph{decorated cobordisms} between them. A 3-manifold with a multi-based link consists of an oriented closed 3-manifold $Y$,  an oriented, embedded link $L\subset Y$ together with disjoint collection of basepoints $\w$ and $\z$ on $L$ such that each component $L_i$ of $L$ has at least two basepoints $z_i, w_i$, and the basepoints alternate between those in $\w$ and those in $\z$ when one traverses a component of $L$. The basepoints $w_i$ and $z_i$ correspond to the variables $U_i$ and $V_i$ in a polynomial ring $\F[U_{\w},V_{\z}]=\F[U_1,U_2,\cdots,U_m,V_{1},V_2,\cdots,V_m]$ where $m=|\w|=|\z|$. Then, $\cCFL(L,\w,\z)$ is defined as a curved complex over $\F[U_{\w},V_{\z}]$. 

 In this paper, we mainly consider the case that each component of a link has exactly two basepoints, i.e., the link component  $L_i$ contains $w_i$ and $z_i$ in $\w$ and $\z$, respectively, and $\F[U_{\w},V_{\z}]=R$. Furthermore, for simplicity, we will drop $\w$ and $\z$ from the notation of a multi-based link if the context is clear.

A \emph{coloring} of a multi-based link $(L, \w, \z)$ is a map $\sigma: \w\cup \z\rightarrow \mathrm{P}$, where $\mathrm{P}$ is a finite set, considered as the set of colors. Corresponding to the set of colors $P=\{p_1,p_2,\cdots,p_k\}$, a polynomial ring
\[\mathcal{R}^-_P:=\F[X_{p_1},X_{p_2},\cdots,X_{p_k}]\]
is defined, which clearly is a $\F[U_{\w},V_{\z}]$-module. For a colored multi-based link $(L,\w,\z,\sigma)$
\[\cCFL(L,\w,\z,\sigma)=\cCFL(L,\w,\z)\otimes_{\F[U_{\w},V_{\z}]}\mathcal{R}_{P}^-\]

 \begin{definition}\cite[Definition 1.3]{Zemke2}
 A \emph{decorated link cobordism} from a 3-manifold with a multi-based link $(Y_1, (L_1, \w_1, \z_1))$ to another one $(Y_2, (L_2, \w_2, \z_2))$ consists of a pair $(W, \bF^{\sigma})$ such that 
 \begin{enumerate}
 \item $W$ is a compact 4-manifold with $\partial W=-Y_1\sqcup Y_2$.
 \item $\bF=(\Sigma, A)$ is an oriented, properly embedded surface $\Sigma$ in $W$, along with a properly embedded 1-manifold $A$ in $\Sigma$, called \emph{dividing arcs}. Further, $\Sigma\setminus A$ consists of two disjoint  (possibly disconnected) subsurfaces, $\Sigma_\w$ and $\Sigma_\z$, such that the intersection of the closures of $\Sigma_\w$ and $\Sigma_\z$ is $A$. 
 \item $\partial \Sigma=-L_1\cup L_2$.
 \item Each component of $L_1\setminus A$ (and $L_2\setminus A$) contains exactly one basepoint.
 \item The $\w$ basepoints are all in $\Sigma_\w$ and the $\z$ basepoints are all in $\Sigma_\z$. 
 \item $\bF$ is equipped with a coloring $\sigma$, i.e. a map $\sigma:C(\Sigma\setminus A)\to P$, where $C(\Sigma\setminus A)$ denotes the set of component of $\Sigma\setminus A$.
 \end{enumerate}

  \end{definition}

To a decorated link cobordism $(W, \bF^{\sigma})$ and a $\Spin$ structure $\mathfrak{s}$ on $W$, Zemke\cite[Theorem A]{Zemke} associated a $\Spin$ functorial chain maps 
$$F_{W, \bF^{\sigma}, \mathfrak{s}}: \cCFL(Y_1, L_1,\w_1,\z_1,\sigma_1, \mathfrak{s}\mid_{Y_1})\rightarrow \cCFL(Y_2, L_2,\w_2,\z_2,\sigma_2, \mathfrak{s}\mid_{Y_{2}}).$$
Here, $\sigma_j$ denotes the colorings on $L_j$ obtained by restricting $\sigma$, for $j=1,2$. The maps are $\mathcal{R}_P^-$-equivariant, $\Z^{P}$-filtered, and are invariants up to $\mathcal{R}_P^-$-equivariant, $\Z^{P}$-filtered chain homotopies.

  Another version of functorial maps for decorated cobordism between links have been independently defined by the first author and Eftekhary in \cite{AE}.

\begin{convention}\label{conv:cob}In this paper, we consider the case that every component of a link has exactly two basepoints (unless when we stabilize them), i.e., the link component $L_i$ contains $w_i$ and $z_i$ in $\w$ and $\z$, respectively, and so $\F[U_{\w},V_{\z}]=R$. Moreover, 
 mostly we work with special cobordisms that every connected component of $\Sigma$ is an annulus, decorated with two parallel vertical dividing arcs. More precisely, for $j=1,2$, $L_j=\coprod_{i=1}^n L_{i,j}$ and $\Sigma=\coprod_{i=1}^n\Sigma_{i}$ where each $\Sigma_i$ is an annulus with $\partial\Sigma_i=-L_{i,1}\sqcup L_{i,2}$. Further, each $A_i=A\cap\Sigma_i$ consists of  two parallel, vertical dividing arcs connecting $L_{i,1}$ to $L_{i,2}$ and dividing $\Sigma_i$ into two rectangles, one containing $w_{i,1}, w_{i,2}$ and another containing $z_{i,1},z_{i,2}$ basepoints. Finally, our coloring set $P$, which is the codomain of $\sigma$, contains exactly $2n$ colors, and $\mathcal{R}_P^{-}\cong R$ such that under this identification $X_{\sigma_j(w_{i,j})}$ and $X_{\sigma_j(z_{i,j})}$ are identified with $U_i$ and $V_i$, respectively. Here, $j=1,2$ and $w_{i,j},\ z_{i,j}$ are the basepoints on $L_{i,j}$. Thus, if we do not emphasis on the basepoints, dividing curves and the colorings, we automatically mean this fixed convention. 
\end{convention}
For a decorated cobordism $(W,\bF)$ as above the cobordism maps $F_{W,\bF,\s}$ are $R$-equivariant and $\Z^{2n}$-filtered. The grading changes under the cobordism maps $F_{W, \bF, \mathfrak{s}}$ are as follows:

\begin{theorem}(Special case of \cite[Theorems 1.4 and 2.14]{Zemke2})
\label{thm:grading}
Suppose $(W, \bF)$  is a decorated link cobordism from $(Y_1, L_1)$ to $(Y_2, L_2)$. Then,\begin{enumerate}
\item If $c_1(\s|_{Y_{1}})$ and $c_{1}(\s|_{Y_{2}})$ are torsion, then $F_{W, \bF, \s}$ is graded with respect to $\gr_{\bf{w}}$, and satisfies 
$$\gr_\w(F_{W, \bF, \s}(x))-\gr_{\w}(x)=\dfrac{c_{1}(\s)^{2}-2\chi(W)-3\sigma(W)}{4}. $$
\item If $c_{1}(\s|_{Y_{1}}-\mathit{PD}\left[ L_1\right])$ and $c_{1}(\s|_{Y_{2}}-\mathit{PD}\left[ L_2\right])$ are torsion, then the map $F_{W, \bF, \s}$ is graded with respect to $\gr_{\z}$, and satisfies
$$\gr_{\z}(F_{W, \bF, \s}(x))-\gr_{\z}(x)=\dfrac{(c_{1}(\s)-\mathit{PD}\left[ \Sigma \right])^{2}-2\chi(W)-3\sigma(W)}{4}. $$
\item Suppose $L_1\subset Y_1$ and $L_2\subset Y_2$ are null-homologous links, i.e. $[L_{i,j}]=0$ in $H_1(Y_{j},\Z)$ for $1\le i\le n$ and $j=1,2$. Moreover, assume both $c_1(\s|_{Y_1})$ and $c_1(\s|_{Y_2})$ are torsion. Then, 
$$A_{i}(F_{W, \bF, \s}(x))-A_{i}(x)=\dfrac{\left\langle c_{1}(\s), \left[\widehat{\Sigma}_{i} \right] \right\rangle-\left[\widehat{\Sigma} \right]\cdot \left[\widehat{\Sigma}_{i} \right]}{2},$$
where $\widehat{\Sigma}_{i}$ denotes the closure of $\Sigma_i$ by adding arbitrary Seifert surfaces of $L_{i,1}\subset Y_1$ and $L_{i,2}\subset Y_2$, and $\left[\widehat{\Sigma}\right]=\sum_{i=1}^{n} \left[ \widehat{\Sigma}_i\right].$ 

\end{enumerate}

\end{theorem}

\begin{theorem}
Assume that $(W, \bF): (S^3, L_1)\rightarrow (S^3, L_2)$ is a decorated  link cobordism with $b_2^{+}(W)=0$ as in Convention \ref{conv:cob}. Then for all $\s$ and $\k\in\HH_{L_1}$ the induced map on homology
\[F_{W, \bF, \s}:\cHFL^{\infty}(L_1,\k)\to\cHFL^\infty(L_2,\k+\d)\]
 is an isomorphism, where $\d$ is the Alexander multi-degree of $F_{W, \bF, \s}$. 
\end{theorem}
\begin{proof} Consider the diagram 

\[
\begin{tikzcd}
  \mathit{CF}^\infty(S^3,\w_1) \arrow[r, "F_{W,\s}"] \arrow[d]
    & \mathit{CF}^\infty(S^3,\w_2) \arrow[d] \\
 \cCFL^\infty(L_1,\k) \arrow[r,  "F_{W,\bF,\s}"]
& \cCFL^\infty(L_2,\k+\d) \end{tikzcd}
\]
where the left (resp. right) vertical arrow is defined by sending $x\in\mathit{CF}^\infty(S^3,\w_1)$ (resp. $x\in\mathit{CF}^\infty(S^3,\w_2)$) to $V^{\k-A(x)}x\in\cCFL^\infty(L_1,\k)$ (resp. $V^{\k+\d-A(x)}x\in\cCFL^\infty(L_2,\k+\d)$). Moreover, $F_{W,\s}$ is the cobordism map corresponding to $W$ and $\Sigma_{\w}$ as defined in \cite{OS-4m}. Similar to Proposition \ref{prop:filter}, it is easy to see that the vertical maps are chain maps and define an isomorphism between the chain complexes. Moreover, it follows from the definition of the cobordism maps that this diagram commutes. By the proof of \cite[Theorem 9.6]{OS-dinvariant} the induced map on homology by $F_{W,\s}$ is an isomorphism and thus the induced map on homology by $F_{W,\bF,\s}$ is an isomorphism as well.

\end{proof}
\begin{corollary}
\label{cor: negative definite}
Assume $(W, \bF): (S^3, L_1)\rightarrow (S^3, L_2)$ is a decorated link cobordism with $b_2^{+}(W)=0$, and $L_1$ and $L_2$ are $L$-space links. If $F_{W, \bF, \s}$ has Alexander multi-degree $\d$ and homological degree $d$ then for all $\k\in \HH_{L_1}$ the induced map on homology
$$
F_{W, \bF, \s}:\cHFL(L_1,\k)\to \cHFL(L_2,\k+\d)
$$
is injective and completely determined by its homological degree.
\end{corollary}

\begin{proof}
Let $z_{L_j}(\k)$ denote the generator of $\cHFL(L_j,\k)\cong \F[\UU]$  of homological degree $-2h_{L_j}(\k)$, for $j=1,2$.  Then 
$$
F_{W, \bF, \s} \left(z_{L_1}(\k)\right)=\UU^{m(\k)}z_{L_2}(\k+\d).
$$
where 
$$
m(\k)=-\left(d+2h_{L_2}(\k+\d)-2h_{L_1}(\k)\right)/{2}.
$$
\end{proof}




\section{Surgery maps}
\label{sec: surgery}

\subsection{Crossing changes}
As shown in Figure \ref{blowup}, one can locally change a positive or negative crossing by performing $(-1)$-surgery on the specified red unknot. In this section, we associate a link cobordism with a simple decoration to these crossing change surgeries and then study the properties of the corresponding cobordism maps.

\begin{figure}[H]
\centering
\begin{tikzpicture}
    \node[anchor=south west,inner sep=0] at (0,0) {\includegraphics[width=3.0in]{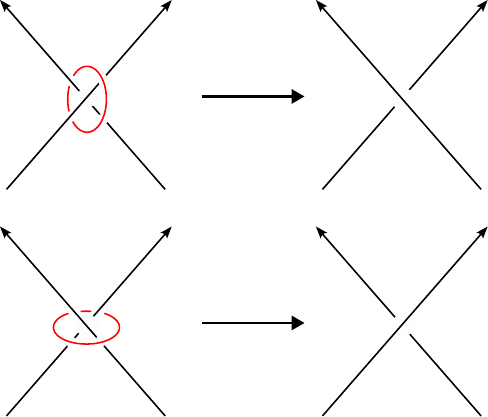}};
    \node[label=above:{$+$}] at (1.4,5.5){};
    \node[label=above:{$-$}] at (1.4,1.8){};
    \node[label=above:{$-$}] at (6.3,5.5){};
    \node[label=above:{$+$}] at (6.3,1.8){};
    \node[label=right:{\color{red}$-1$}] at (1.5,5){};
    \node[label=right:{\color{red}$-1$}] at (1.7,1.4){};
\end{tikzpicture}
\caption{Crossing changes: In top (resp. bottom) figure, $(-1)$-surgery on the red unknot will change the positive (resp. negative) crossing to the negative (resp. positive) crossing.} \label{blowup}
\end{figure}
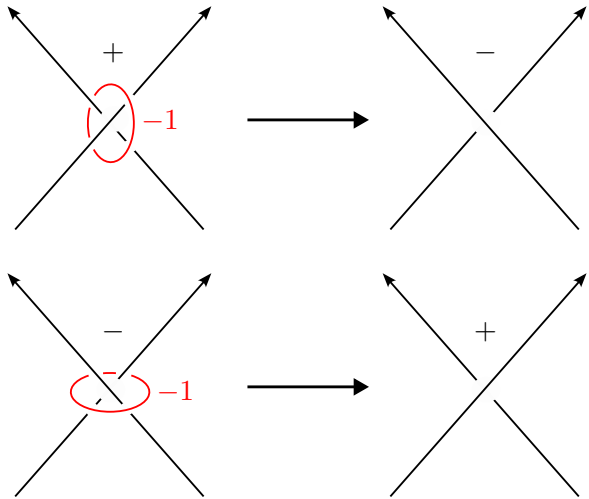


Suppose $L_+=\coprod_{i=1}^nL_{i,+}$ is an $n$-component link in $S^3$, and $L_-$ is the link obtained from $L_+$ by changing a positive crossing between different components to a negative crossing.   So, $L_-$ will have $n$ components as well. Denote the component of $L_-$ corresponding to $L_{i,+}$ by $L_{i,-}$.  Let $W$ be the cobordism from $S^3$ to $S^3$ obtained by attaching a $2$-handle to $S^3\times\{1\}$ in $S^3\times [0,1]$, along the $(-1)$-framed unknot as in the top of Figure \ref{blowup}. Then, the embedded surface $\Sigma=L_+\times[0,1]$ in $W$ gives a cobordism from $L_+$ to $L_-$. The surface $\Sigma$ consists of $n$ connected components, all of them annuli. Denote the component of $\Sigma$ that bounds $-L_{i,+}$ and $L_{i,-}$ by $\Sigma_i$.  Assume each connected component $L_{i,+}$ of $L_+$ contains exactly two basepoints  $w_{i,+}$, $z_{i,+}$, and denote the corresponding basepoints on $L_{i,-}$ by $w_{i,-}$ and $z_{i,-}$, respectively. Decorate each $\Sigma_i$ with two parallel and vertical arcs $A_i$ to divide $\Sigma_i$ into two rectangles, such that one of these rectangles contains the basepoins $z_{i,\pm}$, and the other one contains $w_{i,\pm}$. Then, for $\bF=(\Sigma,A=\coprod_{i=1}^n A_i)$ colored as in Convention \ref{conv:cob}, the pair $(W,\bF)$ gives a decorated cobordism from $(S^3,L_+)$ to $(S^3,L_-)$. Similarly, we define a decorated cobordism from $(S^3,L_-)$ to $(S^3,L_+)$ using the unknot in the bottom of Figure \ref{blowup} as well.

\begin{proposition}
\label{prop:positive}
Let $(W, \bF): (S^{3}, L_{+})\rightarrow (S^{3}, L_{-})$ be the decorated link cobordism induced from attaching a 2-handle along the $(-1)$-framed unknot $K$ so that a positive crossing  becomes a negative crossing, as above. Suppose that the link components $L_{i,+}, L_{j,+}$ are passing through the unknot $K$ with $i<j$. Let $\s_k$ be the $\Spin$ structure on $W$ such that 
$$\langle c_{1}(\s_k), [S^{2}]\rangle =2k+1$$
where $[S^{2}]$ is the generator of $H_{2}(W)$ corresponding to the attached 2-handle. Define $\psi_k:=F_{W,\bF,\s_k}$, the corresponding cobordism map in link Floer homology.
Then
$$\gr_{\w}(\psi_k)=\gr_{\z}(\psi_k)=-k^{2}-k,$$
and 
$$A_{i}(\psi_k)=- A_{j}(\psi_k)=k+1/2.$$
Moreover, $ A_{l}(\psi_k)=0$ for all  $l\neq i, j$. 
In particular,  $A(\psi_0)=\dfrac{1}{2}(\ee_i-\ee_j)$, while $A(\psi_{-1})=-\dfrac{1}{2}(\ee_i-\ee_j)$ and both $\psi_0,\psi_{-1}$ have homological grading zero. 
\end{proposition}

\begin{proof}
By a direct computation, $\chi(W)=1, \sigma(W)=-1$ and 
$$c_{1}(\s_{k})^{2}=(c_{1}(\s_{k})-PD [\Sigma])^{2}=-(2k+1)^{2}. $$
By Theorem \ref{thm:grading}, 
$$\Delta \gr_{\w}=\Delta \gr_{\z}=\dfrac{-(2k+1)^{2}+1}{4}=-k^{2}-k.$$
For the Alexander grading, observe that $[\widehat{\Sigma}]=0\in H_{2}(W)$. Then 
$$\langle c_{1}(\s_{k}), [\widehat{\Sigma}_{i}]\rangle =-(2k+1)\langle [S^{2}], [\widehat{\Sigma}_{i}]\rangle=-(2k+1)\lk(K, L_i).$$
Similarly, $\langle c_{1}(\s_{k}), [\widehat{\Sigma}_{j}]\rangle =-(2k+1) \lk(K, L_j)$. Since, all other components in $L$ do not interact with $K$, by Theorem \ref{thm:grading}, $\Delta A_{l}=0$ for all $l\neq i, j$, and 
$$\Delta A_{i}=-\Delta A_{j}=k+1/2.$$
\end{proof}


\begin{example}\label{ex:Hopf-psi}
By Example \ref{ex:L-space links}, the full  homology of the Hopf link $T(2,2)$ has two generators $a,b$ and is given by
$$
\cHFL(T(2,2))=\frac{R\langle a,b\rangle}{aU_1=bV_2,aU_2=bV_1}.
$$
Cobordism maps $\psi_k:\cHFL(T(2,2))\to \cHFL(O_2)$ are nonzero by Corollary \ref{cor: negative definite}, since $T(2,2)$ and $O_2$ are $L$-space links and $W$ is a nonpositive definite cobordism. Thus, the grading shifts from Proposition \ref{prop:positive} will determine $\psi_k$. Therefore, 
\[\psi_0(a)=V_1,\ \psi_0(b)=U_2;\quad \psi_{-1}(a)=V_2,\ \psi_{-1}(b)=U_1.\]
In general, we have
$$
\psi_k(a)=\begin{cases}
V_1^{k+1}U_2^k\UU^{\frac{k(k-1)}{2}} & \text{if}\ k\ge 0\\
V_2^{-k}U_1^{-1-k}\UU^{\frac{(k+1)(k+2)}{2}}   & \text{if}\ k\le -1,
\end{cases}
\psi_k(b)=\begin{cases}
V_1^{k}U_2^{k+1}\UU^{\frac{k(k-1)}{2}} & \text{if}\ k\ge 0\\
V_2^{-1-k}U_1^{-k}\UU^{\frac{(k+1)(k+2)}{2}}   & \text{if}\ k\le -1.
\end{cases}
$$
\end{example}

\begin{example}
\label{ex:negativeHopf}
One can also regard $L_{+}$ as the 2-component unlink, and $L_{-}$ as the negative Hopf link. By Example \ref{ex:L-space links},  the full homology of $-T(2, 2)$ is generated by $c', d'$ with the relations: 
$$
\cHFL(-T(2,2))=\frac{R\langle c', d'\rangle}{c'U_1=d'U_2, c'V_2=d'V_1}.
$$
By Corollary \ref{cor: negative definite} the cobordism maps $\psi_{k}: \cHFL(O_2)\rightarrow \cHFL(-T(2, 2))$ are nonzero and determined by the grading shift. Therefore,
$$\psi_{0}(1)=c', \quad \psi_{-1}(1)=d'.$$ 
In general, by the grading reasons we have 
$$\psi_k=\begin{cases}
V_1^{k}U_2^k\UU^{\frac{k(k-1)}{2}} \psi_0& \text{if}\ k\ge 0\\
V_2^{-1-k}U_1^{-1-k}\UU^{\frac{(k+1)(k+2)}{2}} \psi_{-1}  & \text{if}\ k\le -1.
\end{cases}$$
\end{example}

\begin{proposition}
\label{prop:psi}
For any link $L=L_+$, the maps $\psi_k$ have the following properties:
\begin{itemize}

\item[(a)] For $k\ge 0$, we have $\psi_k=(V_iU_j)^k\UU^{\frac{k(k-1)}{2}}\psi_0$.

\item[(b)] For $k\le -1$, we have $\psi_k=(V_jU_i)^{-1-k}\UU^{\frac{(k+1)(k+2)}{2}}\psi_{-1}.$

\item[(c)] We have $V_j\psi_0=V_i\psi_{-1}$ and $U_i\psi_0=U_j\psi_{-1}$.
\end{itemize}
\end{proposition}

\begin{proof} 
 The Proposition holds for the special case that $L_+=O_2$ and $L_-=-T(2,2)$ by Example \ref{ex:negativeHopf}. We use the functoriality of link Floer homology and the properties of cobordism maps to show that the general case follows from this special case.

We stabilize $L_{\pm}$ by adding two extra pairs of basepoints $w_{i,\pm}',z_{i,\pm}'$ and $w_{j,\pm}',z_{j,\pm}'$ to $L_{i,\pm}$ and $L_{j,\pm}$, respectively.  Denote the stabilized links by $L'_{\pm}$.  Moreover, we extend the coloring $\sigma$ on $L_{\pm}$ to a coloring $\sigma'$ on $L'_{\pm}$ so that its codomain is ${P}'=P\sqcup \{p'_i,p'_{j}\}$ and 
\[\sigma'(w_{i,\pm}')=p'_i,\quad \sigma'(w'_{j,\pm})=p'_j,\quad \sigma'(z'_{i,\pm})=\sigma(z_i)\quad \text{and}\quad \sigma'(z'_{j,\pm})=\sigma(z_j).\] 
Note that $\sigma'$ restricts to a coloring of $L_{\pm}$ with codomain ${P}'$, and abusing the notation we denote it by $\sigma'$. The isomorphism $\mathcal{R}_{P}^-\cong\F[U_1,\cdots,U_n,V_1,\cdots,V_n]$ extends to an isomorphism 
\[\mathcal{R}_{P'}^-\cong\F[U_1,\cdots,U_n,V_1,\cdots,V_n,U_i',U_j']\]
by sending $X_{p'_i}$ and $X_{p'_j}$ to $U_i'$ and $U_j'$, respectively. Under this isomorphism
\[\cHFL(L_{\pm}^{\sigma'})\cong \cHFL(L_{\pm})\otimes_{\F}\F[U_i',U_j'].\]

By \cite[Section 6]{OS2} (or \cite[Proposition 5.3]{Zemke:basepoint}) we have 
\[\cHFL(L_{\pm}^{'\sigma'})\cong \cHFL(L_{\pm}^{\sigma'})/\langle U_i-U_i',U_j-U_j'\rangle\cong \cHFL(L_{\pm})\]
and under this isomorphism the induced map from $\cHFL(L_{\pm})$ to itself by the quasi-stabilization maps $S_{\pm}$ (see \cite[Section 4.1]{Zemke}) is identity. Note that $S_{\pm}$ corresponds to the quasi-stabilization cobordism $\mathcal{C}_{\pm}$ from $L_{\pm}$ to $L'_{\pm}$ obtained from the product cobordism by adding two dividing arcs on the $i$-th and $j$-th cylinders that split off disks containing $w_{i, \pm}'$ and $w_{j, \pm}'$, as in Figure \ref{fig:quaistab}.

\begin{figure}[ht!]
 \begin{tikzpicture}
    \node[anchor=south west,inner sep=0] at (0,0) {\includegraphics[width=2.0in]{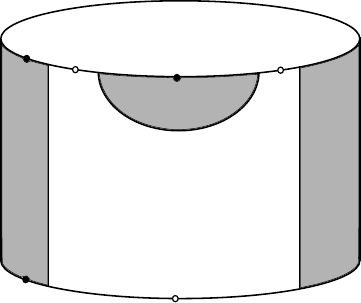}};
    \node[label=right:$w_{i,\pm}$] at (-0.1,0){};
    \node[label=right:$z_{i,\pm}$] at (2,-0.2){};
    \node[label=right:$w_{i,\pm}$] at (-0.1,3.7){};
    \node[label=right:$z_{i,\pm}$] at (0.7,3.5){};
    \node[label=right:$w'_{i,\pm}$] at (2,3.5){};
    \node[label=right:$z'_{i,\pm}$] at (3.3,3.6){};

\end{tikzpicture}
\caption{The decoration on the component $\Sigma_i$ of the quasi-stabilization cobordism $\mathcal{C}_{\pm}$}\label{fig:quaistab}
\end{figure}
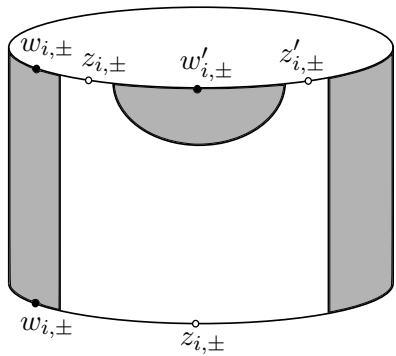 

Next, we construct a decorated cobordism $\mathcal{C}'=(W,\bF')$ from $L'_+$ to $L'_-$ by modifying the decoration on $\mathcal{C}$, as follows. Add two parallel, vertical dividing arcs to $\Sigma_i$ (resp. $\Sigma_j$) such that they divide $\Sigma_i$ (resp. $\Sigma_j$) into four rectangles. Moreover, each one of them contains exactly one of  the pairs $w_{i,\pm}$, $z_{i,\pm}$, $w'_{i,\pm}$ and $z'_{i,\pm}$ (resp.$w_{j,\pm}$, $z_{j,\pm}$, $w'_{j,\pm}$ and $z'_{j,\pm}$) on its boundary. Clearly, $\sigma'$ extends to a coloring on $\bF'$. Define \[\widetilde{\mathcal{C}}=\mathcal{C}'\circ \mathcal{C}_+=\mathcal{C}_-\circ\mathcal{C}.\]  Under the aforementioned isomorphism $\cHFL(L_-^{'\sigma'})\cong \cHFL(L_-)$, the homomorphism induced by the cobordism map $F_{\widetilde{\mathcal{C}},\s_k}$ from $\cHFL(L_+)$ to $\cHFL(L_-)$ is equal to $\psi_k$.

 On the other hand, $\mathcal{C}_+$ can be decomposed as a cobordism $\mathcal{B}$ containing two births from $L_+$ to $L_+\coprod O_2$ followed by two band attachments $\mathcal{C}_{b}$ from $L_+\coprod O_2$ to $L'_+$ as in Figure \ref{fig:bands}.

\begin{figure}[ht!]
 \begin{tikzpicture}
    \node[anchor=south west,inner sep=0] at (0,0) {\includegraphics[width=4.0in]{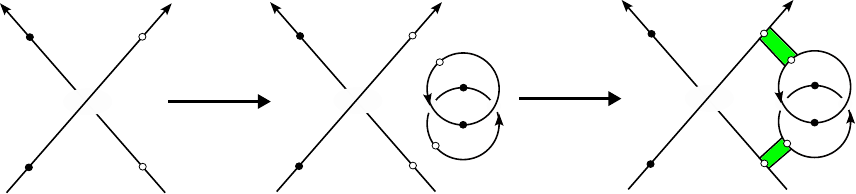}};
    \node[label=right:$\mathcal{B}$] at (2.2,1.4){};
        \node[label=right:$\mathcal{C}_{b}$] at (6.3,1.4){};
\end{tikzpicture}
\caption{Decomposition of the cobordism $\mathcal{C}_+$ as $\mathcal{B}$ followed by $\mathcal{C}_b$.}\label{fig:bands}
\end{figure} 

\begin{figure}[ht!]
\includegraphics{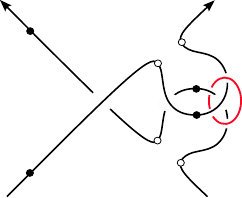}
\caption{}\label{fig:isotopy}
\end{figure} 
 
 \begin{figure}[ht!]
 \begin{tikzpicture}
    \node[anchor=south west,inner sep=0] at (0,0) {\includegraphics[width=6.0in]{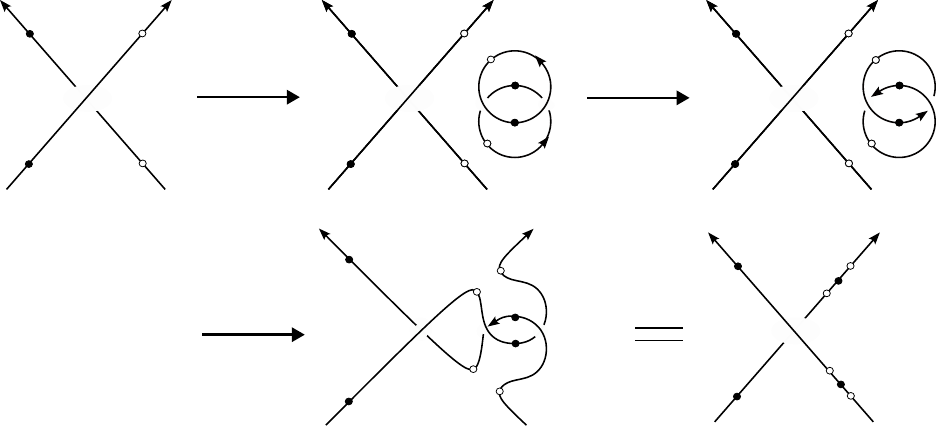}};
    \node[label=right:$\mathcal{B}$] at (3.7,5.7){};
        \node[label=right:$\mathcal{C}_{O}$] at (10,5.7){};
                \node[label=right:$\mathcal{C}_{b}$] at (3.7,1.8){};
\end{tikzpicture}
\caption{}\label{fig:psi-local}
\end{figure}

  On the other hand, one may isotope the attaching circle of the $2$-handle in $\widetilde{\mathcal{C}}$ as in Figure \ref{fig:isotopy}. Then, changing the order of $2$-handle attachment and band attachments as in Figure \ref{fig:psi-local}, we get a decomposition 
 \[\widetilde{\mathcal{C}}=\mathcal{C}_b\circ\mathcal{C}_O\circ\mathcal{B}\] 
 where $\mathcal{C}_O$ denotes the change of crossing cobordism map from $L_+\coprod O_2$ to $L_+\coprod -T(2, 2)$ and $\mathcal{C}_b$ is the band attachment cobordism from $L_+\coprod -T(2, 2)$ to $L'_-$. 
 
By \cite[Theorem B]{Zemke} for any $\Spin$ structure $\s_k$ we have 
\[F_{\widetilde{\mathcal{C}},\s_k}=F_{\mathcal{C}_b}\circ F_{\mathcal{C}_O,\s_k}\circ F_{\mathcal{B}}.\]
Moreover, $\cHFL(L\coprod L')=\cHFL(L)\otimes_{\F}\cHFL(L')$ for any multipointed colored links $L$ and $L'$, and under corresponding identifications 
\[F_{\mathcal{C}_O,s_k}=\mathrm{Id}\otimes\psi_k^O\]
where  $\psi_k^O$ denotes the map $\psi_k$ for the unlink $O_2$. So, the claim holds, because equalities hold for the change of crossing maps for the unlink $O_2$  from Example \ref{ex:negativeHopf}.

\end{proof}

As in the bottom of Figure \ref{blowup}, $(-1)$-surgery on the specified red unknot can change a negative crossing to a positive crossing. Hence, we can also consider the cobordism from $(S^3, L_{-})$ to $(S^{3}, L_+)$ induced by attaching a $2$-handle along this unknot. The embedded surface is a disjoint union of $n$ annuli, and each one of them is equipped with two parallel and vertical dividing arcs. 
\begin{proposition}
\label{prop:negative}
Let $(W, \bF): (S^{3}, L_{-})\rightarrow (S^{3}, L_{+})$ be the decorated link cobordism induced by attaching a 2-handle to the $(-1)$-framed unknot in Figure \ref{blowup} which changes a negative crossing to a positive crossing. Suppose that the link components $L_{i,-}, L_{j,-}$ are passing through the $(-1)$-framed unknot with $i<j$. Let $\s_k$ be the $\Spin$ structure on $W$ satisfying that 
$$\langle c_{1}(\s_k), [S^{2}]\rangle =2k+1$$
where $[S^{2}]$ is the generator of $H_{2}(W)$ corresponding to the attached 2-handle.  Let $\phi_k=F_{W,\bF,\s_k}$ be the corresponding map in link Floer homology.
Then
$$\gr_{\w}(\phi_k)=-k^{2}-k, \quad  \gr_{\z}(\phi_k)=-k^{2}+3k-2$$
and 
$$A_{i}(\phi_k)=A_{j}(\phi_k)=-k+1/2.$$
 Moreover, $ A_{l}(\phi_k)=0$ for all $l\neq i, j$. In particular, $A(\phi_0)=\dfrac{1}{2}(\ee_i+\ee_j)$, $A(\phi_1)=-\dfrac{1}{2}(\ee_i+\ee_j)$ and $\gr_\w(\phi_0)=0$, $\gr_{\w}(\phi_1)=-2$.
\end{proposition}

\begin{proof}
The proof is very similar to the one of Proposition \ref{prop:positive}. 
By the same computation, we get $\gr_{\w}(\phi_k)=-k^{2}-k$. 
For the Alexander gradings, note that 
$$A_i(\phi_k)=\dfrac{\langle c_{1}(\s_{k}), [\widehat{\Sigma}_i]\rangle-[\widehat{\Sigma}]\cdot [\widehat{\Sigma}_{i}]}{2}=\dfrac{-2k-1+2}{2}=-k+1/2.$$
The same computation works for $A_{j}(\phi_k)=-k+1/2$. However, $[\widehat{\Sigma}_l]=0$ for all $l\neq i, j$. Hence, $A_l(\phi_k)=0$ for all such $l$. 
Note that  $\gr_{\w}(\phi_k)-\gr_{\z}(\phi_k)=2(A_1(\phi_k)+\cdots +A_n(\phi_k))=2(A_i(\phi_k)+A_j(\phi_k))=2(-2k+1)$, so $\gr_{\z}(\phi_k)=-k^{2}+3k-2$.

\end{proof}

\begin{example}
\label{ex:Hopf-phi}
By Corollary \ref{cor: negative definite} cobordism maps $\phi_k:\cHFL(O_2)\to\cHFL(T(2,2))$ are non-zero and determined by the grading shift formulas from Proposition \ref{prop:negative}. We compute
\[\phi_0(1)=a\quad\quad\text{and}\quad\quad \phi_1(1)=b.\]
 In general, we have
$$
\phi_k(1)=\begin{cases}
(U_1U_2)^{k-1}\UU^{\frac{(k-1)(k-2)}{2}}b & \text{if}\ k\ge 1\\
(V_1V_2)^{-k}\UU^{\frac{k(k+1)}{2}}a & \text{if}\ k\le 0.
\end{cases}
$$

\end{example} 

\begin{example}
\label{negativeHopf2}
Let $L_{-}=-T(2, 2)$ and $L_{+}=O_2$. Then by Corollary \ref{cor: negative definite} the cobordism maps $\phi_{k}: \cHFL(-T(2, 2))\rightarrow \cHFL(O_2)$ are nonzero and determined by the grading shifts. In particular, 
$$\phi_{0}(c')=V_1, \quad \phi_{1}(c')=U_2, \quad \phi_{0}(d')=V_2, \quad \phi_{1}(d')=U_1.$$
In general, we have $$
\phi_k(c')=\begin{cases}
U_1^{k-1} U_2^{k}\UU^{\frac{(k-1)(k-2)}{2}} & \text{if}\ k\ge 1\\
V_1^{-k+1}V_2^{-k}\UU^{\frac{k(k+1)}{2}}   & \text{if}\ k\le 0,
\end{cases}
\phi_k(d')=\begin{cases}
U_1^{k}U_2^{k-1}\UU^{\frac{(k-1)(k-2)}{2}} & \text{if}\ k\ge 1\\
V_1^{-k}V_2^{-k+1}\UU^{\frac{k(k+1)}{2}}   & \text{if}\ k\le 0.
\end{cases}
$$

\end{example}





\begin{proposition}
\label{prop:phi}
The maps $\phi_k$ satisfy the following properties:
\begin{itemize}
\item[(a)] For $k\ge 1$, we have $\phi_k=(U_iU_j)^{k-1}\UU^{\frac{(k-1)(k-2)}{2}}\phi_1$.
\item[(b)] For $k\le 0$, we have $\phi_k=(V_iV_j)^{-k}\UU^{\frac{k(k+1)}{2}}\phi_{0}.$
\item[(c)] We have $U_i\phi_0=V_j\phi_1$ and $U_j\phi_0=V_i\phi_1$.
\end{itemize}
\end{proposition}

\begin{proof}
The proof is very similar to that of Proposition \ref{prop:psi}. By Example \ref{ex:Hopf-phi}, the claim holds for $L_-=O_2$ and $L_+=T(2,2)$. It remains to show that the general case follows from this special case. Let $\mathcal{C}=(W, \mathcal{F})$ denote the decorated crossing change cobordism from $L_{-}$ to $L_+$. Following the notation in the proof of Proposition \ref{prop:psi}, let $L'_{\pm}$ denote the links obtained from $L_{\pm}$ by adding two extra pairs of base points on $L_{i,\pm}$ and $L_{j,\pm}$, and $\mathcal{C}'$ be the decorated cobordism from $L'_{-}$ to $L'_{+}$ obtained from $\mathcal{C}$ by adding two pairs of parallel and vertical dividing arcs. Further, $\CC_{\pm}$ denotes the decorated quasi-stabilization cobordisms from $L_{\pm}$ to $L'_{\pm}$.  As in the proof of Proposition \ref{prop:psi}, for the decorated cobordism $\widetilde{\CC}=\CC'\circ \CC_{-}$, composing $F_{\widetilde{\CC}, \mathfrak{s}_{k}}$ with a specific isomorphism $\cHFL(L_+^{'\sigma'})\cong \cHFL(L_+)$ is equal to identity.



We now still decompose $\CC_{-}$ as a cobordism $\mathcal{B}$ from $L_{-}$ to $L_-\coprod O_2$ followed by two band attachments $\CC_b$ from $L_-\coprod O_2$ to $L'_-$, as in Figure \ref{fig:bands-phi}.

\begin{figure}[ht!]
 \begin{tikzpicture}
    \node[anchor=south west,inner sep=0] at (0,0) {\includegraphics[width=4.0in]{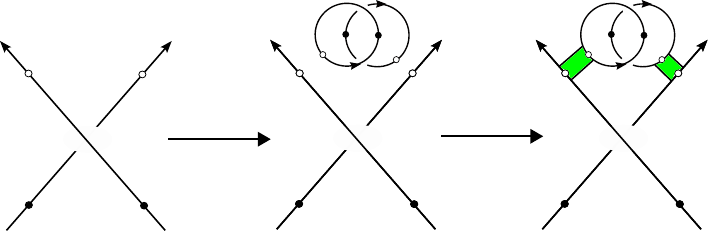}};
    \node[label=right:$\mathcal{B}$] at (2.7,1.6){};
        \node[label=right:$\mathcal{C}_{b}$] at (6.6,1.6){};
\end{tikzpicture}
\caption{Decomposition of the cobordism $\mathcal{C}_-$ as $\mathcal{B}$ followed by $\mathcal{C}_b$.}\label{fig:bands-phi}
\end{figure} 
Changing the order of the $2$-handle attachment in $\mathcal{C}'$ and the band attachments in $\mathcal{C}_b$, we get 
$$\widetilde{\CC}=\CC_b\circ \CC_O\circ \mathcal{B}$$
where $\CC_O$ denotes the cobordism from $L_{-}\coprod O_2$ to $L_{-}\coprod T(2, 2)$ and $\CC_b$ is the band attachment cobordism from $L_{-}\coprod T(2, 2)$ to $L'_{+}$, see Figure \ref{fig:phi-local}. Hence, for any Spin$^{c}$ structure $\mathfrak{s}_{k}$ we have 
$$F_{\widetilde{\CC},\s_k}=F_{\CC_{b}}\circ F_{\CC_{O}, \mathfrak{s}_{k}}\circ F_{\mathcal{B}}.$$
Here $F_{\CC_{O}, \mathfrak{s}_{k}}=\mathrm{Id}\otimes \phi_{k}^{O}$  where $\phi_{k}^{O}$ denotes the map $\phi_k$ for the unlink $O_2$. Hence the claim holds for general links. 

 \begin{figure}[ht!]
 \begin{tikzpicture}
    \node[anchor=south west,inner sep=0] at (0,0) {\includegraphics[width=6.0in]{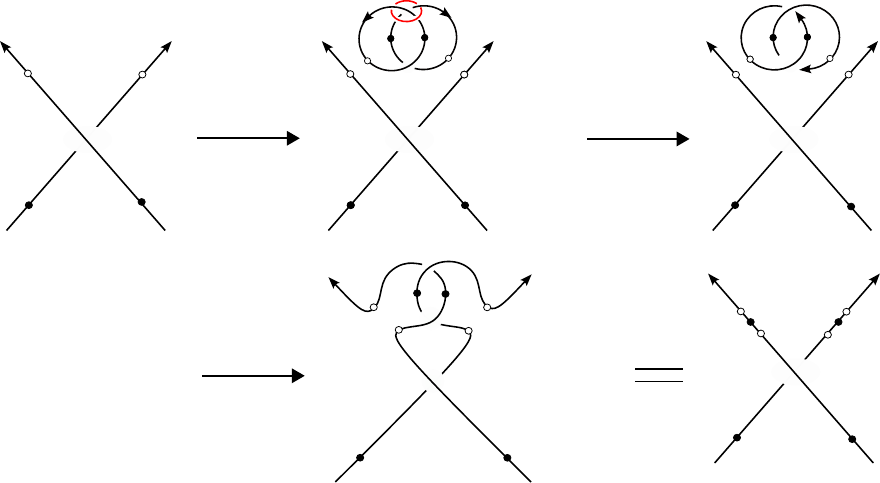}};
    \node[label=right:$\mathcal{B}$] at (3.8,6.2){};
        \node[label=right:$\mathcal{C}_{O}$] at (10.5,6.2){};
                \node[label=right:$\mathcal{C}_{b}$] at (3.8,2.1){};
\end{tikzpicture}
\caption{}\label{fig:phi-local}
\end{figure} 

\end{proof}

\begin{proposition}
\label{prop: psi phi}
The maps $\psi_k$ and $\phi_k$ in Proposition \ref{prop:positive} and Proposition \ref{prop:negative} compose as follows:
$$
\phi_0\psi_0=V_i,\ \phi_0\psi_{-1}=V_j, \ \phi_1\psi_0=U_j, \ \phi_1\psi_{-1}=U_i
$$
$$
\psi_0\phi_0=V_i,\ \psi_{-1}\phi_0=V_j,\ \psi_0\phi_1=U_j,\ \psi_{-1}\phi_1=U_i,
$$
The rest of compositions are determined by these.
\end{proposition}
\begin{proof}
We prove the equalities in the first row, and the proof for the second row is similar.  It is straightforward from Examples \ref{ex:negativeHopf} and \ref{negativeHopf2} that the claim holds for $L_+=O_2$ and $L_-=-T(2, 2)$ the negative Hopf link, because
$$
\phi_0\psi_0(1)=V_1,\ \phi_0\psi_{-1}(1)=V_2,\ \phi_1\psi_0(1)=U_2,\ \phi_1\psi_1(1)=U_1.
$$

The strategy is similar to the proof of Propositions \ref{prop:psi} and \ref{prop:phi} and we fix the same notation. To distinguish the cobordisms defining $\psi_k$ and $\phi_k$ we use subscripts $1$ and $2$, i.e. let $\CC_1=(W_1,\mathcal{F}_1)$ be the decorated crossing change cobordism from $L_+$ to $L_-$ and $\CC_2=(W_2,\mathcal{F}_2)$ be the decorated crossing change cobordism from $L_-$ to $L_+$. As before, we denote the links obtained from $L_{\pm}$ by adding two extra pairs of base points on $L_{i,\pm}$ and $L_{j,\pm}$ by $L'_{\pm}$. Further, we denote the corresponding cobordism from $L_+'$ to $L_-'$ (resp. $L_-'$ to $L_+'$) obtained from $\mathcal{C}_{1}$ (resp. $\mathcal{C}_2$) by adding two pairs of parallel and vertical dividing arcs by $\mathcal{C}_1'$ (resp. $\mathcal{C}_2'$). Moreover, we consider quasi-stabilization cobordisms $\mathcal{C}_{\pm}$ from $L_{\pm}$ to $L'_{\pm}$. 

Let $\widetilde{\CC}=\mathcal{C}'_2\circ\mathcal{C}'_1\circ\mathcal{C}_+=\mathcal{C_+}\circ\mathcal{C}_2\circ\mathcal{C}_1$. For any $k_1,k_2\in\Z$, denote the $\Spin$ structure on $\widetilde{\CC}$ whose restriction to $\mathcal{C}_1$ and $\CC_2$ is equal to $\s_{k_1}$ and $\s_{k_2}$, respectively, by $\s_{k_1,k_2}$. Thus, under the aforementioned isomorphism $\cHFL(L_{+}^{'\sigma'})\cong\cHFL(L_+)$ the cobordism map $F_{\widetilde{C},\s_{k_1,k_2}}$ is equal to $\phi_{k_2}\circ\psi_{k_1}$.

On the other hand, as depicted in Figure \ref{fig:bands}, the cobordism $\CC_+$ can be decomposed as $\CC_+=\CC_b\circ\mathcal{B}$, where $\mathcal{B}$ is the decorated cobordism from $L_+$ to $L_+\coprod O_2$ corresponding to two births, and $\mathcal{C}_b$ is defined by attaching two bands. By Figure \ref{fig:composition}, after an isotopy on the attaching circles of the $2$-handles in $\CC_1$ and $\CC_2$, we may change their order with the band attachments in $\CC_b$  to get another decomposition \[\widetilde{\mathcal{C}}=\mathcal{C}_b\circ\mathcal{C}_2^{O}\circ\mathcal{C}_1^{O}\circ\mathcal{B}\]
Here,  $\mathcal{C}_1^{O}$ denotes the decorated cobordism from  $L_+\coprod O_2$ to $L_+\coprod -T(2, 2)$ corresponding to changing a positive crossing to a negative crossing in $O_2$. Similarly,  $\mathcal{C}_2^{O}$ is the cobordism from $L_+\coprod -T(2, 2)$ to $L_+\coprod O_2$ corresponding to changing a negative crossing to a positive crossing in $-T(2, 2)$. Thus,

\[F_{\widetilde{\mathcal{C}},\s_{k_1,k_2}}=F_{\mathcal{C}_b}\circ F_{\mathcal{C}_2^{O},\s_{k_2}}\circ F_{\mathcal{C}_1^{O},\s_{k_1}}\circ F_{\mathcal{B}},\]
and the claim follows from the special case of $L_+=O_2$ and $L_-=-T(2,2)$.

 \begin{figure}[ht!]
  \includegraphics[width=4in]{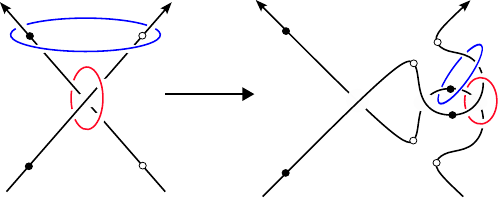}
\caption{Cobordisms $\CC_1$ and $\CC_2$ are define by attaching $2$-handles along the red and the blue unknots, respectively.  }\label{fig:composition}
\end{figure}

\end{proof}

\subsection{Full twists}
\label{subsec: gen crossing change}
In this section, we will apply similar computation as in Proposition \ref{prop:negative} to get the properties of the cobordism map induced by attaching a $2$-handle along a $(-1)$-framed unknot through $n$-strand braid to get a positive full twist. 

 \begin{figure}[ht!]
 \begin{tikzpicture}
    \node[anchor=south west,inner sep=0] at (0,0) {\includegraphics[width=3in]{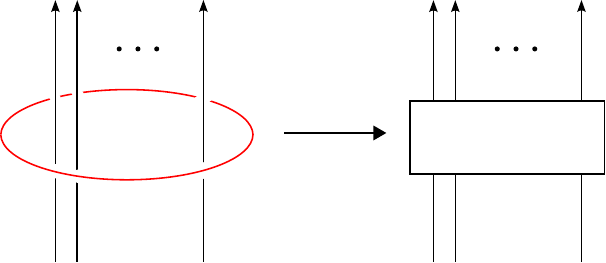}};
    \node[label=right:$n$] at (1.3,3){};
    \node[label=right:{\color{red} $-1$}] at (2.8,2){};
    \node[label=right:$+n$] at (6,1.6){};
   
\end{tikzpicture}
\caption{ $(-1)$-surgery on the red unknot will add a positive full twist.}\label{fig:fulltwist}
\end{figure}
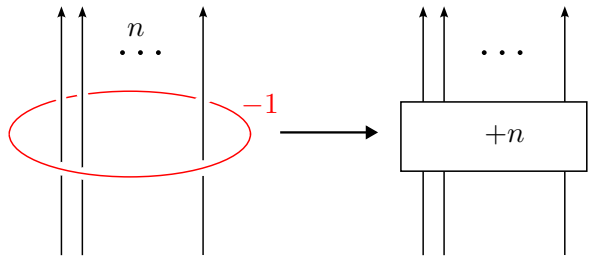 

Using the similar computation as in Proposition \ref{prop:negative}, we have the following:
\begin{proposition}
\label{prop:fulltwist}
Let $(W, \bF): (S^{3}, L)\rightarrow (S^{3}, \bar{L})$ be the decorated link cobordism obtained by attaching a $2$-handle on the $(-1)$-framed unknot which adds a full twist to the $n$ parallel strands as in Figure \ref{fig:fulltwist}. Let $\s_k$ be the $\Spin$ structure on $W$ satisfying that 
$$\langle c_{1}(\s_k), [S^{2}]\rangle =2k+1$$
where $[S^{2}]$ is the generator of $H_{2}(W)$ corresponding to the attached 2-handle.  Let $\phi^n_k=F_{W,\bF,\s_k}$ be the corresponding map in link Floer homology.
Then
$$\gr_{\w}(\phi^n_k)=-k^{2}-k$$
and 
$$A_{i}(\phi^n_k)=-k+(n-1)/2,$$
 for $i=1, 2, \cdots, n$.  
\end{proposition}

\begin{proof}
The computation of $\gr_{\w}$ is exactly the same as the one of Proposition \ref{prop:negative}. Hence, $\gr_{\w}(\phi^n_k)=-k^2-k$. For the computation of the Alexander grading, it is also similar to the one of Proposition \ref{prop:negative}, except now for each $i=1, 2, \cdots, n$, we have 
$$A_i(\phi^n_k)=\dfrac{\langle c_{1}(\s_{k}), [\widehat{\Sigma}_i]\rangle-[\widehat{\Sigma}]\cdot [\widehat{\Sigma}_{i}]}{2}=\dfrac{-2k-1+n}{2}=-k+(n-1)/2.$$

\end{proof}

Now we consider the following example where $L=O_n$ and  $\bar{L}=T(n, n)$. It is known \cite{GH} that $T(n,n)$ is an $L$-space link.  We first recall the link Floer homology $\cHFL(T(n, n))$. For the explicit computation, see \cite{BLZ}. 

\begin{theorem}[\cite{BLZ}]
\label{thm:toruslink}
The Heegaard Floer homology  $\mathcal{HFL}(T(n,n))$ has $n$ generators, which we denote by $a_0,\dots,a_{n-1}$ subject to the following relations:
\begin{equation}
\label{eq: Tnn}
 \left( \prod_{i\in I_k}U_i\right) a_{k-1}=\left(\prod_{j\in \{1, \cdots, n\}\setminus I_k}V_j\right) a_{k},\quad 
  U_iV_i a_k=U_jV_j a_k
\end{equation}
Here, $I_k$ is any subset of the set $\{1,\dots,n\}$ of length $k$ (so the first equation has $\binom{n}{k}$ relations for each $k$), and in the second equation $i,j$, and $k$ range from $1$ to $n$.
\end{theorem}

Now we list the explicit gradings of the generators $a_k$ where $0\leq k \leq n-1$. The Alexander multi-grading of $a_k$ is
$$\left( \dfrac{n-1}{2}-k, \dfrac{n-1}{2}-k, \cdots, \dfrac{n-1}{2}-k \right).$$

The generator $a_k$ has homological grading 
$$(\gr_{\w}(a_k), \gr_{\z}(a_k))=\left(-k(k+1), -k(k+1)-n(n-1)+2kn\right).$$
The Maslov grading $\gr_{\w}$ is obtained from the $H$-function of the torus link $T(n,  n)$, which is computed in \cite{GH}. The computation of $\gr_{\z}$ follows from the relation
$$\dfrac{\gr_{\w}-\gr_{\z}}{2}=A_1+A_2+\cdots+A_n. $$

\begin{example}
\label{ex:toruslink-phi}
By Corollary \ref{cor: negative definite} the cobordism maps $\phi^n_k:\cHFL(O_n)\to\cHFL(T(n,n))$ are non-zero and determined by the grading shift formulas from Proposition \ref{prop:fulltwist}. Recall that $A_i(\phi^n_k)=(-k+\frac{n-1}{2},\cdots, -k+\frac{n-1}{2})$ and $\gr_\w(\phi^n_k)=-k^{2}-k$. Then
\[\phi^n_k(1)=a_{k}\]
for $k=0, 1, \cdots, n-1$. 

 In general, we have
$$
\phi^n_k(1)=\begin{cases}
(U_1\cdots U_n)^{k-(n-1)}\UU^{\frac{(k-(n-1))(k-n)}{2}}a_{n-1} & \text{if}\ k\ge n-1\\
(V_1\cdots V_n)^{-k}\UU^{\frac{k(k+1)}{2}}a_0 & \text{if}\ k\le 0.
\end{cases}
$$

\end{example} 

Similar to Proposition \ref{prop:phi}, the maps $\phi_k^{n}$ satisfy the following properties:
\begin{proposition}
\label{prop:toruslink-ends}
The maps $\phi^{n}_k: \cHFL(L)\rightarrow \cHFL(\bar{L})$ satisfy the following properties: 

a) For $k\ge n-1$, we have $\phi^n_k=(U_1\cdots U_n)^{k-(n-1)}\UU^{\frac{(k-(n-1))(k-n)}{2}}\phi^n_{n-1}$.

b) For $k\le 0$, we have $\phi^n_k=(V_1\cdots V_n)^{-k}\UU^{\frac{k(k+1)}{2}}\phi^n_{0}.$
\end{proposition}

\begin{proof}
The proof is very similar to the one of Proposition \ref{prop:phi}. As before, we denote $L'$ (resp. $\bar{L}'$) as the link obtained from $L$ (resp. $\bar{L}$) by adding an extra pair of basepoints $w'_i, z'_i$ for each component $L_i$. Let $\CC'$ be the induced decorated cobordism from $L'$ to $\bar{L}'$ induced from the decorated cobordism $\CC=(W, \mathcal{F})$ from $L$ to $\bar{L}$, and $\sigma'$ be the induced coloring on $L'$ and $\bar{L}'$ as in the proof of Proposition \ref{prop:psi}. We still get the isomorphism 
$$\cHFL(L'^{\sigma'})\cong \left(\cHFL(L)\otimes \F[U_1, \cdots, U_n, V_1, \cdots, V_n, U'_1, \cdots, U'_n]\right)/ \langle U_1-U'_1, \cdots, U_n-U'_n\rangle\cong \cHFL(L).$$
Similarly, we also have $\cHFL(\bar{L}'^{\sigma'})\cong \cHFL(\bar{L})$.

As before, we let $\CC_+$ be the decorated cobordism from $L$ to $L'$, which can be decomposed as a cobordism $\mathcal{B}$ from $L$ to $L\coprod O_n$ followed by $n$ band attachments $\CC_b$ from $L\coprod O_n$ to $L'$. Hence, $F_{\widetilde{\CC}, \mathfrak{s}_k}=\phi^n_k$ where $\widetilde{\CC}=\CC'\circ \CC_+$. Now we use the same trick as before to isotope attaching circle of the 2-handle as in Figure \ref{fig:isotopy} so that it encircles the unlink $O_n$ and change the order of the 2-handle attachment and band attachments. Then 
$$\widetilde{\CC}=\CC_{b}\circ \CC_{O_{n}}\circ \mathcal{B}$$
where $\CC_{O_{n}}$ denotes the cobordism obtained by attaching a $2$-handle along $(-1)$-framed unknot from $L\coprod O_n$ to $L\coprod T(n, n)$ and $\CC_{b}$ denotes the band attachment cobordism from $L\coprod T(n, n)$ to $\bar{L}$. Hence, 
$$F_{\CC_{O_{n}}, \mathfrak{s}_k}=\mathrm{Id}\otimes \phi_{k}$$
where $\phi_k$ denotes the map $\phi_{k}^{n}$ for the unlink $O_n$. Since the proposition holds for unlink $O_n$ by Example \ref{ex:toruslink-phi}, the general case follows. 

\end{proof}

\section{Skein exact sequence}
\label{sec: surgery skein}

\subsection{Surgery exact triangle}

Suppose $L$ is a link in an integer homology sphere $Y$, and $K\subset Y\setminus L$ is a knot. Let $(W_1,\bF_1)$ be the decorated link cobordism from $(Y,L)$ to $(Y_{-1}(K),L)$ obtained by attaching a two-handle along $K$ with framing $-1$, and decorated as in Convention \ref{conv:cob}. Similarly, $(W_2,\bF_2)$ and $(W_3,\bF_3)$ denote the cobordisms from $(Y_{-1}(K),L)$ to $(Y_{0}(K),L)$ and $(Y_{0}(K),L)$ to $(Y,L)$, respectively. 

\begin{proposition}
The link cobordism maps $F_{i}=\sum_{\mathfrak{s}\in\Spin(W_i)}F_{W_i,\bF_i,\mathfrak{s}}$ form an exact triangle as follows.
\[
 \begin{tikzcd}
\cbHFL(Y, L) \arrow[rr,"F_{1}"]& &\cbHFL(Y_{-1}(K),L)\arrow[dl, "F_{2}"] \\
  &\cbHFL(Y_{0}(K),L)\arrow[ul, "F_{3}"]&
\end{tikzcd}
\]


\end{proposition}

\begin{proof}
This is a straightforward generalization of the surgery exact triangle for $\mathit{HF}^+$ in \cite[Section 9]{OS5}. We will outline the proof and highlight the differences here.  
Consider a multi-pointed Heegaard diagram
\[\mathcal{H}=\left(\Sigma,\alphas,\betas=\{\beta_1,\cdots,\beta_k\},\gammas=\{\gamma_1,\cdots,\gamma_k\},\deltas=\{\delta_1,\cdots,\delta_k\},\z,\w \right)\]
where $k=g+n-1$, such that 
\begin{itemize}
\item $\mathcal{H}_{\alpha\beta}=(\Sigma,\alphas,\betas,\z,\w)$, $\mathcal{H}_{\alpha\gamma}=(\Sigma,\alphas,\gammas,\z,\w)$ and $\mathcal{H}_{\alpha\delta}=(\Sigma,\alphas,\deltas,\z,\w)$ are Heegaard diagrams for the link $L$ in $3$-manifolds $Y$, $Y_{-1}(K)$ and $Y_0(K)$, respectively. So, $\z$ and $\w$ consist of $n$ basepoints, where $n$ is the number of connected components of $L$. 
\item For any $1\le i\le k-1$, $\gamma_i$ and $\delta_i$ are small isotopic translations of $\beta_i$ such that they intersect $\beta_i$ transversely in two points and are disjoint from $\beta_j$ for $j\neq i$ . Moreover, $\delta_i$ intersects $\gamma_i$ in two transverse points as well. 
\item Pairwise intersections of $\beta_k$, $\gamma_k$ and $\delta_k$ are single points with signs  $\#(\beta_k\cap\gamma_k)=\#(\gamma_k\cap\delta_k)=\#(\delta_k\cap\beta_k)=-1$. Moreover, $\gamma_k$ is obtained from the juxtaposition of $\beta_k$ and $\delta_k$.
\item Strongly admissible in the sense of \cite[Definition 4.15]{Zemke:graph} which is a multipointed version of \cite[Section 8.4.2]{OS-3m1}.

\end{itemize}

Let $F_{\alpha\beta\gamma}$ be the chain map defined by counting holomorphic triangles as: 
\[\begin{split}
&F_{\alpha\beta\gamma}:\cbCFL(\Sigma,\alphas,\betas,\z,\w)\otimes \cbCFL(\Sigma,\betas,\gammas,\z,\w)\to\cbCFL(\Sigma,\alphas,\gammas,\z,\w)\\
&F_{\alpha\beta\gamma}(\x\otimes\x')=\sum_{\y\in\mathbb{T}_{\alpha}\cap\mathbb{T}_{\gamma}}\sum_{\{\Psi\in\pi_2(\x,\x',\y)|\mu(\Psi)=0\}}\prod_{i=1}^nU_i^{n_{w_i}(\Psi)}V_i^{n_{z_i}(\Psi)}\cdot\y
\end{split}\]
Analogously, we define chain maps $F_{\alpha\gamma\delta}$ and $F_{\alpha\delta\beta}$.

The Heegaard diagram $\mathcal{H}_{\beta\gamma}=(\Sigma,\betas,\gammas,\z,\w)$ represents an $n$ component unlink in $\#^{g-1}(S^1\times S^2)$, denoted by $O_n$. It is straightforward that $\cbHFL(\#^{g-1}(S^1\times S^2), O_n)$ is a free $\mathbf{R}$-module of rank $2^{g-1}$. Moreover, the summand with largest $\gr_{\w}$ has rank one and so it has a unique generator. The Heegaard diagram $\mathcal{H}_{\beta\gamma}$ has an intersection point denoted by $\Theta_{\beta\gamma}$ that generates this top degree homology class, called \emph{top generator}. Specifically, $\Theta_{\beta\gamma}$ is the intersection point that every element of $\pi_2(\x,\Theta_{\beta\gamma})$ has nonzero coefficient in at least one $\z$ or $\w$ basepoint, for all other intersection points $\x$. 
Top generators $\Theta_{\gamma\delta}$ and $\Theta_{\delta\beta}$ for $\cbCFL(\mathcal{H}_{\gamma\delta})$ and $\cbCFL(\mathcal{H}_{\delta\beta})$, respectively, are defined analogously. 

Let 
\[f_1(\cdot)=F_{\alpha\beta\gamma}(\cdot\otimes\Theta_{\beta\gamma}),\quad f_2(\cdot)=F_{\alpha\gamma\delta}(\cdot\otimes\Theta_{\gamma\delta})\quad\text{and}\quad f_3(\cdot)=F_{\alpha\delta\beta}(\cdot\otimes \Theta_{\delta\beta}).\]
By definition of cobordism maps in \cite{Zemke}, for any $1\le i\le 3$ we have
\[F_i=(f_i)_*=\sum_{\mathfrak{s}\in\mathrm{Spin}^c(W_i)}F_{W_i,\bF_i,\mathfrak{s}}.\]

By \cite[Lemma 4.4]{OS-doublecover} to show that they form an exact triangle, we need to check that 
\begin{enumerate}
\item $f_{i+1}\circ f_i$ is chain homotopically trivial by a chain homotopy $h_i$,
\item $f_{i+2}\circ h_i+h_{i+1}\circ f_i$ is a homotopy equivalence,
\end{enumerate} 
where indices are cyclic modulo three. Note that we need a version of \cite[Lemma 4.4]{OS-doublecover} for chain complexes over $\mathbf{R}$, which for example follows from \cite[Lemma 3.3]{AE-minus}. 

First, we check condition $(1)$. Suppose $i=1$. The proof for $i=2$ and $3$ is similar.  Then, 
\[f_2\circ f_1(\cdot)=F_{\alpha\gamma\delta}(F_{\alpha\beta\gamma}(\cdot\otimes\Theta_{\beta\gamma})\otimes\Theta_{\gamma\delta})\simeq F_{\alpha\beta\delta}(\cdot\otimes F_{\beta\gamma\delta}(\Theta_{\beta\gamma}\otimes\Theta_{\gamma\delta})),\]
where the chain homotopy is $h_i(\cdot)=F_{\alpha\beta\gamma\delta}(\cdot\otimes\Theta_{\beta\gamma}\otimes \Theta_{\gamma\delta})$ and

\[\begin{split}
&F_{\alpha\beta\gamma\delta}:\cbCFL(\mathcal{H}_{\alpha\beta})\otimes \cbCFL(\mathcal{H}_{\beta\gamma})\otimes\cbCFL(\mathcal{H}_{\gamma\delta})\to\cbCFL(\mathcal{H}_{\alpha\delta})\\
&F_{\alpha\beta\gamma\delta}(\x\otimes\x'\otimes \x'')=\sum_{\y\in\mathbb{T}_{\alpha}\cap\mathbb{T}_{\delta}}\sum_{\{\phi\in\pi_2(\x,\x',\x'',\y)|\mu(\phi)=-1\}}\prod_{i=1}^nU_i^{n_{w_i}(\phi)}V_i^{n_{z_i}(\phi)}\cdot\y
\end{split}\]
An argument similar to the proof of \cite[Proposition 9.5]{OS5} implies that $F_{\beta\gamma\delta}(\Theta_{\beta\gamma}\otimes\Theta_{\gamma\delta})=0$ and so $f_2\circ f_1\simeq 0$. 

For condition $(2)$, let $\betas'$ be a generic small Hamiltonian isotopic translate of $\betas$, and $F_{\alpha\beta\gamma\delta\beta'}$ be the chain map defined by counting pentagons of Maslov index $-2$, analogous to $F_{\alpha\beta\gamma}$ and $F_{\alpha\beta\gamma\delta}$. Then, $F_{\alpha\beta\gamma\delta\beta'}(\cdot\otimes\Theta_{\beta\gamma}\otimes\Theta_{\gamma\delta}\otimes\Theta_{\delta\beta'})$ gives a chain homotopy between $f_3\circ h_1+h_2\circ f_1$ and $F_{\alpha\beta\beta'}(\cdot\otimes F_{\beta\gamma\delta\beta'}(\Theta_{\beta\gamma}\otimes \Theta_{\gamma\delta}\otimes\Theta_{\delta\beta'}))$. By a standard ``stretching the neck argument" and following the strategy in \cite[Section 2]{OS-lectures} and \cite[Section 4.2]{OS-doublecover} we have 
\[F_{\beta\gamma\delta\beta'}(\Theta_{\beta\gamma}\otimes \Theta_{\gamma\delta}\otimes\Theta_{\delta\beta'})=\sum_{k=0}^\infty \UU^{\frac{k(k+1)}{2}}\Theta_{\beta\beta'}\]
and
$$
F_{\alpha\beta\beta'}(\cdot\otimes F_{\beta\gamma\delta\beta'}(\Theta_{\beta\gamma}\otimes \Theta_{\gamma\delta}\otimes\Theta_{\delta\beta'}))=\left(\sum_{k=0}^\infty \UU^{\frac{k(k+1)}{2}}\right)F_{\alpha\beta\beta'}(\cdot \otimes \Theta_{\beta\beta'})
$$
Since $F_{\alpha\beta\beta'}(\cdot \otimes \Theta_{\beta\beta'})$ is a homotopy equivalence (see the proof of \cite[Theorem 8.6]{AE-minus}) and $\sum_{k=0}^\infty \UU^{\frac{k(k+1)}{2}}$ is invertible, $f_3\circ h_1+h_2\circ f_1$ is a homotopy equivalence.

\end{proof}

Now let us relate the surgery exact triangle with resolutions. Suppose $Y=S^3$, $L=L_+$ is a link in $S^3$ with a fixed positive crossing, and $K\subset S^3\setminus L$ is an unknot as in the top of Figure \ref{blowup}. Then, $(S^3_{-1}(K), L)$ will be identified with $(S^{3}, L_{-})$.
Next we relate $(S^{3}, L_0)$ with $(S^3_{0}(K), L)$, where $L_0$ denotes the oriented resolution at the fixed crossing. Observe that $S^3_{0}(K)=S^{3}\# (S^{2}\times S^{1})$, and $L$ in $S^3_{0}(K)$ still has $n$ components, while $L_{0}$ is an $(n-1)$-component link. Note that we can replace the 2-handle attaching to $K$ with framing $0$ by a 1-handle as in Figure \ref{kirby1}.

 \begin{figure}[h]
\includegraphics{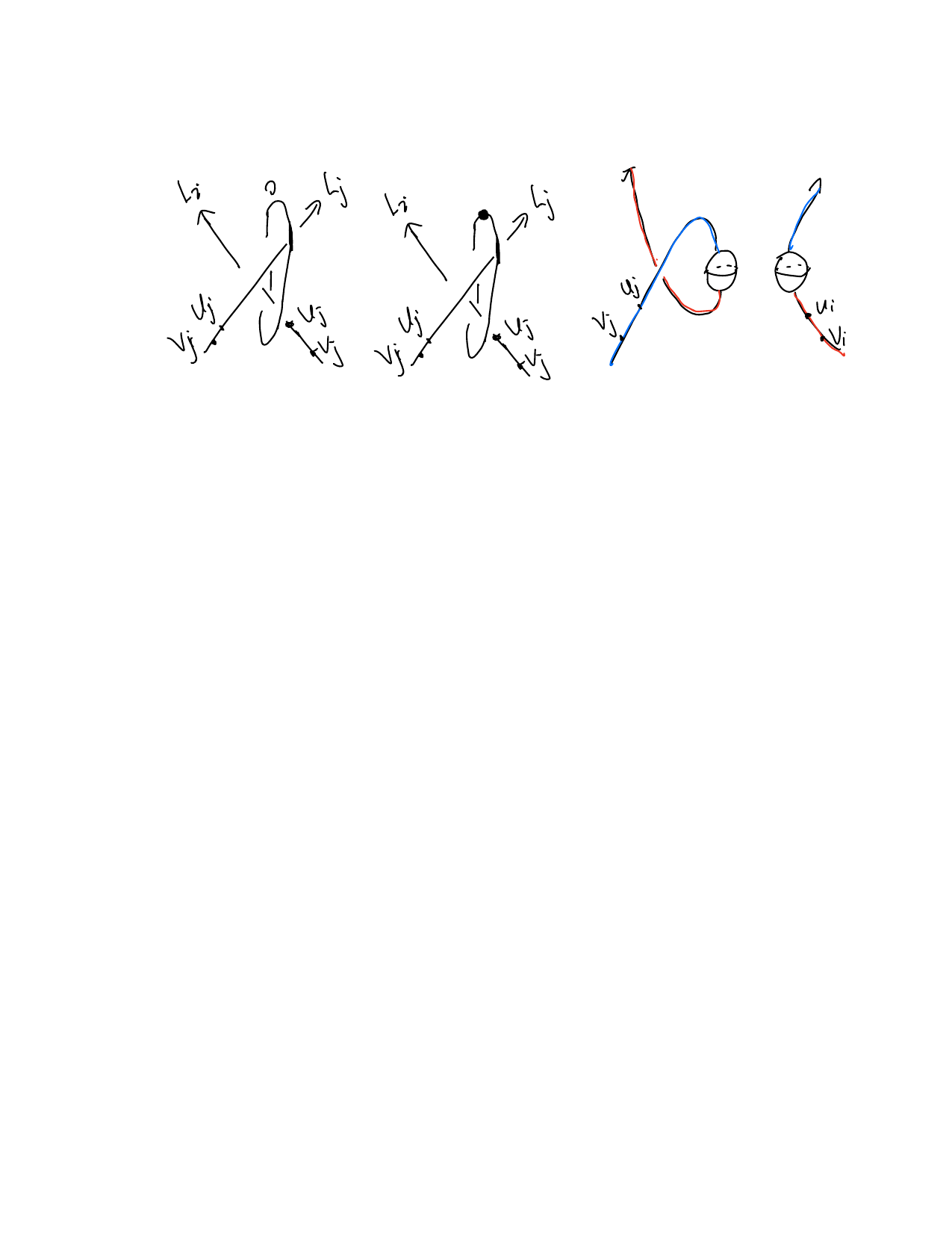}
\caption{}\label{kirby1}
\end{figure}

\begin{lemma}
\label{lemma:unlinkcone}
The link $(S^{3}_{0}(K), L)$ can be identified with $(S^{3}\#(S^{2}\times S^{1}), L_0\# Z_2)$ where $Z_2$ is the 2-component unlink in $S^{2}\times S^{1}$ consisting of two parallel circles representing the homology class of $S^{1}$, and the $\#$ between $L_0$ and $Z_2$ is identified as in Figure \ref{kirby2}.

\end{lemma}

\begin{figure}[ht!]
\includegraphics{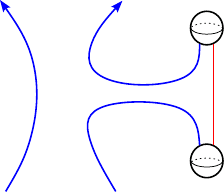}
\caption{Special (local) connected sum between $L_0\subset S^3$ (in blue) and $Z_2\subset S^2\times S^1$ (one component in blue and another component in red) }\label{kirby2}
\end{figure}

\begin{proof}
The proof is depicted in Figure  \ref{kirby1}. Specifically, we regard the $2$-handle for the $0$-surgery on $K$ as a $1$-handle and then we move the feet of 1-handle along the link $L$. At the end, we get the connected sum of $L_0$ with one component of $Z_2$, colored blue, along with the other component of $Z_2$, colored red, in Figure 
\ref{kirby2}.

\end{proof}

Therefore, we have the following theorem:

\begin{theorem}
Given a local positive crossing of the link components $L_i$ and $L_j$ of a link $L_+$ in the integer homology sphere $Y$,  there is a skein exact sequence 
\begin{equation}
\label{eq: skein}
\rightarrow \cHFL(Y, L_{+})\xrightarrow{\Psi} \cHFL(Y, L_{-})\xrightarrow{\alpha} H_*(\cCFL(Y, L_0)\otimes_{R}R_0) \xrightarrow{\beta} \cHFL(Y, L_{+})\rightarrow
\end{equation}
where the map from $\cHFL(Y, L_{+})$ to  $\cHFL(Y, L_{-})$ is given by $\Psi=\sum_{k\in \Z}(-1)^{k}\psi_k$, and 
$$
R_0=\frac{\F[U_1,\ldots,U_n,V_1,\ldots,V_n]}{(U_i-U_j,V_i-V_j)}.
$$

\end{theorem}

\begin{proof}
By Lemma \ref{lemma:unlinkcone}, the cone of $\Psi$ is the homology of the tensor of $\cCFL(Y, L_0)$ with some complex $Z$ corresponding to the unlink $Z_2$ in $S^2\times S^1$. Moreover, $Z_2$ is independent of the pair $(Y, L_+)$. We use the special case that $Y=S^{3}$ and $L_+=T(2, 2)$ to compute $\cCFL(S^2\times S^1, Z_2)$ which gives the module $\cK$. We put the detailed computation of Hopf link in Section \ref{skein:Hopf} (equivalently, see \eqref{coneK}).

The complex $\cK$ is given as follows:
\begin{equation}
\label{def:K}
\cK=
\begin{tikzcd}
R \arrow{r}{U_i-U_j} \arrow[swap]{d}{V_i-V_j}  & R  \arrow{d}{V_i-V_j}   \\
       R \arrow[swap]{r}{U_j-U_i} &  R  \\ 
\end{tikzcd}
\end{equation}
Note that it is a free resolution of $R_0$ over $R$. Since $\cCFL(Y,L_0)$ is a complex of free $R$-modules, we get 
$$
H_*(\cCFL(Y,L_0)\otimes \cK)\simeq H_*(\cCFL(Y,L_0)\otimes R_0). 
$$
\end{proof}

\begin{remark}
In $\HFL^-$ version of Heegaard Floer homology one sets $V_i=V_j=0$, and the complex $\cK$ breaks into a direct sum of two copies of $\F[U_1,\ldots,U_n]\xrightarrow{U_i-U_j} \F[U_1,\ldots,U_n].$ This explains the appearance of a two-dimensional vector space in \cite{OS4}.

Similarly, for $\widehat{\HFL}$ one sets $U_i=U_j=V_i=V_j=0$, and the complex $\cK$ breaks into four copies of $\F$.
\end{remark}

Without loss of generality, we assume that $i=1, j=2$ for the rest of the section.  Recall that $L_+$ and $L_-$ have $n$ components while $L_0$ has $(n-1)$ components. For all $k\in \Z$ we have chain maps $\psi_k: \cHFL(Y, L_{+})\to \cHFL(Y, L_{-})$, and one can consider the formal sum
$$
\Psi=\sum_{k\in \Z}(-1)^{k} \psi_k: \cbHFL(Y, L_{+})\to \cbHFL(Y, L_{-})
$$ 
Note that $\Psi$ is a non-homogeneous map containing terms of various non-positive homological degrees.

As the non-homogeneous map $\Psi$ is hard to deal with, we would like to reduce it to the degree zero piece $\Psi^0=\psi_0-\psi_{-1}$. 

 \begin{lemma}
\label{lem: tau}
Let $L=L_+$ be an arbitrary link in the three-sphere with a fixed positive crossing between its first and second components. For $k\in\Z$, suppose $\psi_k:\cbHFL(L_{+})\to \cbHFL(L_{-})$ is the corresponding crossing change map and $\Psi=\sum_{k\in \Z}(-1)^k\psi_k$. Then in homology we have 
$\Psi=\tau(\psi_{0}-\psi_{-1})$ where 
\begin{multline}
\label{def tau}
\tau=\sum_{k\ge 0}(-1)^{k}\left[(V_1U_2)^{k}+(V_1U_2)^{k-1}(U_1 V_2)\ldots+(U_1V_2)^{k}\right]\UU^{\frac{k(k-1)}{2}}+\\
\sum_{k\ge 1}(-1)^{k}\left[(V_1U_2)^{k-1}+(V_1U_2)^{k-2}(U_1 V_2)\ldots+(U_1V_2)^{k-1}\right]\UU^{\frac{k(k-1)}{2}+1}=1+\ldots.
\end{multline}
In particular, $\tau$ is an invertible power series.
\end{lemma}

\begin{proof}
First, we introduce notations
$$
A_k=(V_1U_2)^k\UU^{\frac{k(k-1)}{2}},\ B_k=(U_1V_2)^k\UU^{\frac{k(k-1)}{2}},
$$
$$
C_k=(V_1U_2)^{k-1}+\ldots+(U_1V_2)^{k-1}=\sum_{i=0}^{k-1}(V_1U_2)^{i}(U_1V_2)^{k-1-i}.
$$
Clearly,
$$
A_k-B_k=(V_1U_2-U_1V_2)C_k \UU^{\frac{k(k-1)}{2}}$$
and so
$$A_k+U_1V_2C_k \UU^{\frac{k(k-1)}{2}}=B_k+V_1U_2C_k \UU^{\frac{k(k-1)}{2}}=C_{k+1} \UU^{\frac{k(k-1)}{2}},
$$
and
$$
\tau=\sum_{k\ge 0}(-1)^{k}(C_{k+1}+C_k\UU)\UU^{\frac{k(k-1)}{2}}=\sum_{k\ge 0}(-1)^{k}\left(B_k+(V_1U_2+V_1U_1)C_k\UU^{\frac{k(k-1)}{2}}\right).
$$
By Lemma \ref{prop:psi} parts (a) and (b) we have
 $
\psi_k=A_k\psi_0,\ \psi_{-1-k}=B_k\psi_{-1}
$ for $k\ge 0$, and 
therefore
$$
\Psi=\sum_{k\ge 0}(-1)^{k}A_k\psi_0-\sum_{k\ge 0}(-1)^{k}B_k\psi_{-1}=
$$
$$
\sum_{k\ge 0}(-1)^{k}B_k(\psi_0-\psi_{-1})+\sum_{k\ge 0}(-1)^{k}(A_k-B_k)\psi_{0}=
$$
$$
\sum_{k\ge 0}(-1)^{k}B_k(\psi_0-\psi_{-1})+\sum_{k\ge 0}(-1)^{k}\UU^{\frac{k(k-1)}{2}}C_k(V_1U_2-U_1V_2)\psi_{0}.
$$
By Lemma \ref{prop:psi} part (c) we have
$$
V_2\psi_0=V_1\psi_{-1},\ U_1\psi_0=U_2\psi_{-1},
$$
so 
$$
(V_1U_2+V_1U_1)(\psi_0-\psi_{-1})=V_1U_2\psi_0-V_1U_2\psi_{-1}+V_1U_1\psi_0-V_1U_1\psi_{-1}=$$
$$V_1U_2\psi_0-V_1U_1\psi_{0}+V_1U_1\psi_0-V_2U_1\psi_{0}=(V_1U_2-V_2U_1)\psi_0.
$$
Therefore 
$$
\Psi=\sum_{k\ge 0}(-1)^{k}B_k(\psi_0-\psi_{-1})+\sum_{k\ge 0}(-1)^{k}\UU^{\frac{k(k-1)}{2}}C_k(V_1U_2+V_1U_1)(\psi_0-\psi_{-1})=\tau(\psi_0-\psi_{-1}).
$$
\end{proof}

\begin{corollary}
\label{cor: short psi}
The cones of $\Psi$ and of $\psi_0-\psi_{-1}$ are homotopy equivalent.
\end{corollary}
\begin{remark}
Note that at $V_1=V_2=1$ and $U_1=U_2=\UU$, we get that the invertible factor $\tau$ is
$$
\tau=\sum_{k\ge 0}(-1)^{k}\left((k+1)\UU^{k}\cdot \UU^{\frac{k(k-1)}{2}}+k\UU^{k-1}\cdot \UU^{\frac{k(k-1)}{2}+1}\right)=
\sum_{k\ge 0}(-1)^{k}(2k+1)\UU^{\frac{k(k+1)}{2}}
$$
which agrees with \cite[Theorem 3.7, Blow-up formula]{OS-4m} modulo $2$. That is because in this case $\psi_k$ is equal to the Ozsv\'ath-Szab\'o's cobordism map associated to the blow-up of the product cobordism $\left(S^3\times[0,1]\right)\#\ovl{\mathbb{CP}^2}$ from $S^3$ to $S^3$.

\end{remark}

\subsection{Skein exact sequence for Hopf link}
\label{skein:Hopf}

Next, we compute the skein exact sequence on the chain complex level for the Hopf link. Recall that the full link Floer complex for $H=T(2,2)$ has generators $a,b,c,d$ and the differential 
$$
\partial(c)=U_1a-V_2b,\ \partial(d)=U_2a-V_1b.
$$
The homology is generated by $a,b$ modulo relations $U_1a=V_2b,U_2a=V_1b$ as above.

On the other hand, the link Floer complex for the unknot has generators $1,\xi$ and the differential $\partial(\xi)=(U_1V_1-U_2V_2)$. 
By Example \ref{ex:Hopf-psi}
we have
$$
\psi_0(a)=V_1,\ \psi_0(b)=U_2,\ \psi_{-1}(a)=V_2,\ \psi_{-1}(b)=U_1.
$$
and the maps $\psi_0$ and $\psi_{-1}$ can be uniquely lifted to chain complex level by setting
$$
\psi_0(c)=\xi,\ \psi_0(d)=0, \psi_{-1}(c)=0,\ \psi_{-1}(d)=-\xi.
$$
Furthermore,  for all $k\ge 0$ we have 
$
\psi_k=A_k\psi_0,\ \psi_{-1-k}=B_k\psi_{-1},
$
where we follow the notations in Lemma \ref{lem: tau} and its proof. For concreteness, we can lift the statement of 
 Lemma \ref{lem: tau} to the level of chain complexes.
\begin{lemma}
\label{lem: tau hopf}
The map $\Psi:=\sum_{k\in \Z}(-1)^{k}\psi_k$ is homotopic to $\tau(\psi_0-\psi_{-1})$, where $\tau$ is defined by \eqref{def tau}.
\end{lemma}
\begin{proof}
We have
$$
\Psi=\sum_{k\in \Z}(-1)^{k}\psi_k=\sum_{k\ge 0}(-1)^{k}(A_k\psi_0-B_k\psi_{-1}).
$$
Define 
$$
h_a=V_1\sum_{k\ge 0}(-1)^{k}C_k\UU^{\frac{k(k-1)}{2}},\ h_b=-U_1\sum_{k\ge 0}(-1)^{k}C_k\UU^{\frac{k(k-1)}{2}}
$$
then 
$$
\tau=\sum_{k\ge 0}(-1)^{k}A_k+U_1h_a-V_2h_b=\sum_{k\ge 0}(-1)^{k}B_k+U_2h_a-V_1h_b.
$$
We define the homotopy $h$ by $h(a)=h_a\xi, h(b)=h_b\xi$ and $h(c)=h(d)=0$, and let $\widetilde{\Psi}=\Psi+\partial h+h\partial.$ Then,
$$
\widetilde{\Psi}(c)=\sum_{k\ge 0}(-1)^{k}A_k\xi+h(U_1a-V_2b)=\tau\xi
$$
$$
\widetilde{\Psi}(d)=\sum_{k\ge 0}(-1)^{k}B_k\xi+h(U_2a-V_1b)=\tau\xi
$$
$$
\widetilde{\Psi}(a)=V_1\sum_{k\ge 0}(-1)^kA_k-V_2\sum_{k\ge 0}(-1)^{k}B_k+\partial(h_a\xi)=
$$
$$
V_1\sum_{k\ge 0}(-1)^kA_k-V_2\sum_{k\ge 0}(-1)^{k}B_k+(U_1V_1-U_2V_2)h_a=(V_1-V_2)\tau 
$$
Similarly, $\widetilde{\Psi}(b)=\tau(U_2-U_1)$ and we conclude that $\widetilde{\Psi}=\tau(\psi_0-\psi_{-1})$.
\end{proof}

By Lemma \ref{lem: tau hopf} we can replace the cone of $\Psi$ by the cone of $\psi_0-\psi_{-1}$ which is isomorphic to the following complex:
\begin{center}
\begin{tikzcd}
 & \xi \arrow[bend left=50]{ddr}{U_1V_1-U_2V_2}&  \\
c \arrow{r}{U_1} \arrow[near end]{ddr}{-V_2} \arrow{ur}{1}&  a \arrow[near start]{dr}{V_1-V_2}&  \\
      &  &  1  \\
d \arrow{r}{-V_1} \arrow[near start]{uur}{U_2} \arrow[bend left=100]{uuur}{1}& b\arrow[swap]{ur}{U_2-U_1} &  \\
\end{tikzcd}
\end{center}

We can define $\xi'=c-d$ and change the basis from $(c,d)$ to $(\xi',d)$:
\begin{center}
\begin{equation}\label{coneK}
\begin{tikzcd}
 & \xi \arrow[bend left=50]{ddr}{U_1V_1-U_2V_2}&  \\
\xi' \arrow{r}{U_1-U_2} \arrow[near end]{ddr}{V_1-V_2} &  a \arrow[near start]{dr}{V_1-V_2}&  \\
      &  &  1  \\
d \arrow{r}{-V_1} \arrow[near start]{uur}{U_2} \arrow[bend left=100]{uuur}{1}& b\arrow[swap]{ur}{U_2-U_1} &  \\
\end{tikzcd}
\end{equation}
$\simeq$
\begin{tikzcd}
\xi' \arrow{r}{U_1-U_2} \arrow[swap]{dr}{V_1-V_2}  &  a \arrow[near start]{dr}{V_1-V_2}&  \\
      & b\arrow[swap]{r}{U_2-U_1} &  1  \\ 
\end{tikzcd}
\end{center}
Here we use the fact that the quotient complex $d\xrightarrow{1} \xi$ is contractible.

\begin{lemma}
Let $\Phi=\sum_{k} (-1)^{k}\phi_k: \cbCFL(O_2)\to \cbCFL(T(2,2))$, then the cone of $\Phi$ is quasi-isomorphic to the cone of $\Psi$ up to relabeling the variables.
\end{lemma}

\begin{proof}
By Example \ref{ex:Hopf-phi} we have $\phi_{-k}(1)=(V_1V_2)^{k}\UU^{\frac{k(k-1)}{2}}a$ and $\phi_{1+k}(1)=(U_1U_2)^{k}\UU^{\frac {k(k-1)}{2}}b$ for $k\ge 0$, so
$$
\Phi(1)=a\sum_{k\ge 0}(-1)^{k}(V_1V_2)^{k}\UU^{\frac{k(k-1)}{2}}-b\sum_{k\ge 0}(-1)^{k}(U_1U_2)^{k}\UU^{\frac{k(k-1)}{2}}.
$$
The homology of the cone of $\Phi$ is generated by $a$ and $b$ modulo relations $U_1a=V_2b,U_2a=V_1b$ and $\Phi(1)=0$. Since the coefficients at $a$ and $b$ in $\Phi(1)$ are invertible, the result is generated by $a$ modulo relations
\begin{equation}\label{eq: cone Phi}
\begin{aligned}
aV_2\sum_{k\ge 0}(-1)^{k}(V_1V_2)^{k}\UU^{\frac{k(k-1)}{2}}&=aU_1\sum_{k\ge 0}(-1)^{k}(U_1U_2)^{k}\UU^{\frac{k(k-1)}{2}},\\
aV_1\sum_{k\ge 0}(-1)^{k}(V_1V_2)^{k}\UU^{\frac{k(k-1)}{2}}&=aU_2\sum_{k\ge 0}(-1)^{k}(U_1U_2)^{k}\UU^{\frac{k(k-1)}{2}}
\end{aligned}
\end{equation}
We claim that 
\begin{equation}\label{eq: tau prime}
V_2\sum_{k\ge 0}(-1)^{k}(V_1V_2)^{k}\UU^{\frac{k(k-1)}{2}}-U_1\sum_{k\ge 0}(-1)^{k}(U_1U_2)^{k}\UU^{\frac{k(k-1)}{2}}=(V_2-U_1)\tau'
\end{equation}
where
\begin{multline*}
\tau'=\sum_{k\ge 0}(-1)^{k}\left((U_1U_2)^{k}+(U_1U_2)^{k-1}(V_1V_2)+\ldots+(V_1V_2)^{k}\right)\UU^{\frac{k(k-1)}{2}}+\\
\sum_{k\ge 0}(-1)^{k}\left((U_1U_2)^{k-1}+(U_1U_2)^{k-2}(V_1V_2)\ldots+(V_1V_2)^{k-1}\right)\UU^{\frac{k(k-1)}{2}+1}.
\end{multline*}
Indeed, let $C'_{k}=\sum_{i=0}^{k-1}(V_1V_2)^i(U_1U_2)^{k-1-i}$, then
$$
V_2(V_1V_2)^{k}-U_1(U_1U_2)^{k}=(V_2-U_1)(V_1V_2)^k-U_1(U_1U_2-V_1V_2)C'_k
$$
and $U_1U_2-V_1V_2=(U_1-V_2)(U_2+V_1)$, hence
$$
(V_2-U_1)\left[(V_1V_2)^k+U_1(U_2+V_1)C'_k\right]=(V_2-U_1)\left[C'_{k+1}+\UU C'_k\right].
$$
Now by \eqref{eq: tau prime} we can rewrite the equations \eqref{eq: cone Phi} as
$$
a(V_2-U_1)\tau'=a(V_1-U_2)\tau'=0
$$
which is equivalent to 
$$
a(V_2-U_1)=a(V_1-U_2)=0.
$$
\end{proof}

\begin{remark}
Note that the above computation is very similar to the one in Lemma \ref{lem: tau}, in particular, $\tau'$ is related to $\tau$ (resp. $C'_k$ is related to $C_k$) by exchanging the variables $V_2\leftrightarrow U_2$ which corresponds to changing the orientation on a link component.
\end{remark}

\section{Link splitting maps}
\label{sec: splitting maps}

\subsection{Splitting maps}

Let $L=L_1\cup\ldots\cup L_n$ be an arbitrary link in the three-sphere. We can change the crossings between different components in $L$ arbitrarily and consider the corresponding maps in Heegaard Floer homology: if we change a  positive crossing to a negative one we can use either $\psi_{-1}$ or $\psi_0$, and  if we change a  negative crossing to a positive one we can use either $\phi_0$ or $\phi_1$.
This does not change the topological type of the components and we will denote the components for all such links by $L_i$. In particular, by such crossing changes we can transform $L$ to the split link $\Split(L)$ obtained by the split union of all link components $L_i$. 

More precisely, we consider two links $L,L'$ related by such crossing changes, and a chain map $F:\cHFL(L)\to \cHFL(L')$
obtained as a composition of: 
\begin{itemize}
\item $P_{ij}^{-1}$ of maps $\psi_{-1}$ associated to positive crossings between $L_i$ and $L_j$;
\item $P_{ij}^{0}$ of maps $\psi_{0}$ associated to positive crossings between $L_i$ and $L_j$;
\item $N_{ij}^{0}$ of maps $\phi_{0}$ associated to negative crossings between $L_i$ and $L_j$;
\item $N_{ij}^{1}$ of maps $\phi_{1}$ associated to negative crossings between $L_i$ and $L_j$.
\end{itemize}
Note that, in principle, $F$ may depend on the order of the crossings and the choice of these crossings (for example, for the Hopf link we can change either one of two crossings to transform it to unlink). Also note that the linking number between $L_i$ and $L_j$ changes by 
$$
\lk_{L'}(L_i,L_j)-\lk_L(L_i,L_j)=N_{ij}^{1}+N_{ij}^{0}-P_{ij}^{0}-P_{ij}^{-1}.
$$
Nevertheless, we have the following general result:
\begin{lemma}
\label{lem: many crossing changes}
Let $L,L'$ be two links related by crossing changes as above. Then:
\begin{itemize}
\item[(a)] The chain map $F:\cHFL(L)\to \cHFL(L')$ has Alexander degree
$$
A(F)=\sum_{i<j}\left[ \frac{1}{2}(P_{ij}^{0}-P_{ij}^{-1})(\ee_i-\ee_j)+\frac{1}{2}(N_{ij}^{0}-N_{ij}^{1})(\ee_i+\ee_j)\right]
$$
and homological degree 
$$
\gr_{\w}(F)=-2\sum_{i<j}N_{ij}^{1}.
$$

\item[(b)] For any $\k\in\HH_L$, the map $F:\cHFL(L,\k)\to \cHFL(L',\k+A(F))$, restricted to the $\F[\UU]$-free summand, is injective and determined by its homological degree. 

\item[(c)] The $h$-functions of $L$ and $L'$ are related by the inequality
$$
h_L(\k)+\sum_{i<j}N_{ij}^{1}\ge h_{L'}(\k+A(F)).
$$ 
\end{itemize}
\end{lemma}

\begin{proof}
Part (a) is clear from the degrees of maps $\psi_{-1},\psi_0,\phi_0,\phi_1$ computed in Propositions \ref{prop:positive} and \ref{prop:negative}.

Part (b) follows from Proposition \ref{prop: psi phi}. Note that  $F$ is a composition of crossing change maps. For each crossing change map $\phi_0, \phi_1, \psi_0, \psi_{-1}$ associated to the link components $L_i, L_j$, by Proposition \ref{prop: psi phi}, one can choose a corresponding crossing change map such that the composition of these two crossing change maps is one of the monomials $U_i, U_j, V_i, V_j$. Hence, 
 we can compose the map $F$ with another map $F':\cHFL(L')\to \cHFL(L)$ such that the composition $F'\circ F$ is given by a monomial in variables $U_1,\cdots, U_n, V_1, \cdots, V_n$. Therefore,  the map  $F: \cHFL(L, \k)\rightarrow \cHFL(\L', \k+A(F))$ is injective when restricted to the $\F[\UU]$-free part of  $\cHFL(L, \k)$.

Now part (c) follows from (b): indeed, $F$ sends $\F[\UU][-2h_L(\k)]$ to  $$\F[\UU]\left[-2h_L(\k)-2\sum_{i<j}N_{ij}^{1}\right]$$ which should inject into $\F[\UU][-2h_{L'}(\k+A(F))].$ Hence,
$$
-2h_L(\k)-2\sum_{i<j}N_{ij}^{1}\le -2h_{L'}(\k+A(F))
$$
which yields the desired inequality. 
\end{proof}

If $L$ and $L'$ differ by a single positive crossing change, that is, $L'$ is obtained by changing a negative crossing of $L$ into a positive one, we can recover the comparison between $h_L$ and $h_L'$ in item (b) of Theorem 6.20 in \cite{BG}. 

\begin{corollary}
\label{coro:onecrossing}
Let $L$ be an $n$-component link in $S^3$. Suppose $L'$ is obtained by changing a negative crossing between the components $L_i$ and $L_j$ of $L$ into a positive one. Then 
$$\max \{h_{L'}(\k+\dfrac{1}{2}(\ee_i+\ee_j)), h_{L'}(\k-\dfrac{1}{2}(\ee_i+\ee_j))-1\}\leq h_{L}(\k)\leq \min\{h_{L'}(\k\pm \dfrac{1}{2}(\ee_i-\ee_j))\}$$
for all $\k\in \HH_L$. 

\end{corollary}

\begin{proof}
Note that $L'$ is obtained from $L$ by changing a negative crossing to a positive crossing, which induces a cobordism map $\phi_{0}$ or $\phi_{1}$. So we can make  $N^{0}_{ij}=1$ or $N^{1}_{ij}=1$. By  Lemma \ref{lem: many crossing changes}, if $N^{0}_{ij}=1$ then 
$$h_{L}(\k)\geq h_{L'}(\k+\dfrac{1}{2}(\ee_i+\ee_j)).$$
If $N^{1}_{ij}=1$ then 
$$h_{L}(\k)+1\geq h_{L'}(\k-\dfrac{1}{2}(\ee_i+\ee_j)).$$
Conversely, $L$ can be obtained from $L'$ by changing a positive crossing to a negative crossing. Then we can make $P^{0}_{ij}=1$ or $P^{-1}_{ij}=1$. By Lemma  \ref{lem: many crossing changes} again, if $P^{0}_{ij}=1$, then 
$$h_{L'}(\k)\geq h_{L}(\k+\dfrac{1}{2}(\ee_i-\ee_j)).$$
If $P^{-1}_{ij}=1$, then 
$$h_{L'}(\k)\geq h_{L}(\k-\dfrac{1}{2}(\ee_i-\ee_j)). $$

\end{proof}

Note that \cite[Theorem 6.20]{BG} uses the $J$-function, which is some normalizations of the $h$-function in the following formula:
$$J_{L}(\m)=h_{L}(\m+\frac{1}{2}(\ell_1, \cdots, \ell_n))$$
for $\m\in \mathbb{Z}^{n}$. Let $\mathbf{\ell}_{L}=\frac{1}{2}(\ell_1, \cdots, \ell_n)$. Since $L'$ is obtained from $L$ by changing a negative crossing to a positive crossing between $L_i$ and $L_j$. Then $\mathbf{\ell}_{L'}=\ell_L+\dfrac{1}{2}(\ee_i+\ee_j)$. Then the inequalities in Corollary \ref{coro:onecrossing} become
$$\max \{J_{L'}(\m), J_{L'}(\m-\ee_i-\ee_j)-1\}\leq J_{L}(\m)\leq \min \{J_{L'}(\m-\ee_i), J_{L'}(\m-\ee_j)\}.$$
So we recover Theorem 6.20 in \cite{BG}. Moreover, we have an extra inequality that $J_{L'}(\m-\ee_i-\ee_j)-1\leq J_{L}(\m)$. Note that $h_{L}(\k)=h_{L}(\k+\ee_i)$ or $h_{L}(\k)=h_{L}(\k+\ee_i)+1$ for all links $L$ and all $i$. So the above inequalities imply that $J_{L}(\m)=J_{L'}(\m)$ or $J_{L}(\m)=J_{L'}(\m)+1$.

\subsection{Positive links}

Suppose that all crossings between different components of $L$ are positive (in particular, this holds if $L$ is a positive link). Then there are exactly $2\lk(L_i,L_j)$ crossings between $L_i$ and $L_j$, and we need to change  $\lk(L_i,L_j)$ of them to split the components $L_i$ and $L_j$. 
We can encode crossing change maps as above, with $N_{ij}^{0}=N_{ij}^{1}=0$ and 
$$
P_{ij}^{-1}+P_{ij}^{0}=\lk(L_i,L_j)=\ell_{ij}. 
$$
Then Lemma \ref{lem: many crossing changes} simplifies dramatically, and we get the following
\begin{corollary}
Suppose that all crossings between different components of $L$ are positive.
Let $F:\cHFL(L)\to \cHFL(\Split(L))$ be a composition of $P_{ij}^{-1}$ maps of type $\psi_{-1}$ and $P_{ij}^{0}$ maps of type $\psi_0$ between the components $L_i$ and $L_j$ for all $i<j$. Define $\varepsilon_{ij}=P_{ij}^{0}-P_{ij}^{-1}$, then $F$ has Alexander degree
\begin{equation}
\label{eq: deg alex positive}
A(F)=\sum_{i<j}\frac{1}{2}\varepsilon_{ij}(\ee_i-\ee_j)
\end{equation}
and homological degree zero.
\end{corollary}

We can visualize the degrees of such maps as follows. 

\begin{definition}
\label{def: zonotope}
Suppose that $L$ is a link such that all crossings between different components are positive, and $\ell_{ij}=\lk(L_i,L_j)\ge 0$.
We define the {\bf link zonotope} $P_L$ as the Minkowski sum of the intervals $[\ell_{ij}\ee_i,\ell_{ij}\ee_j]$ for all $i<j$.
\end{definition}

Note that $P_L$ is an $(n-1)$-dimensional polytope contained in the hyperplane $$\left\{\sum_{i=1}^{n} x_i=\sum_{i<j}\ell_{ij}\right\}\subset \R^{r}.$$ It is centrally symmetric around the point $\frac{1}{2}(\ell_1,\ldots,\ell_n)$ where $\ell_i=\sum_{j\neq i}\ell_{ij}$.

\begin{example}
For $n=2$ the polytope $P_L$ is a segment $[\ell_{12}\ee_1,\ell_{12}\ee_2]$ with $\ell_{12}+1$ integer points on it. 
\end{example} 

\begin{example}
For $n=3$, the polytope $P_L$ is a hexagon where the opposite sides are parallel to each other and both contain $\ell_{ij}+1$ integer points. The vertices of $P_L$ are:
$$
(0,\ell_{12},\ell_{13}+\ell_{23}),(\ell_{12},0,\ell_{13}+\ell_{23}),(\ell_{12}+\ell_{13},0,\ell_{23}),
$$
$$
(\ell_{12}+\ell_{13},\ell_{23},0),(\ell_{13},\ell_{12}+\ell_{23},0),(0,\ell_{12}+\ell_{23},\ell_{13}).
$$
\end{example}

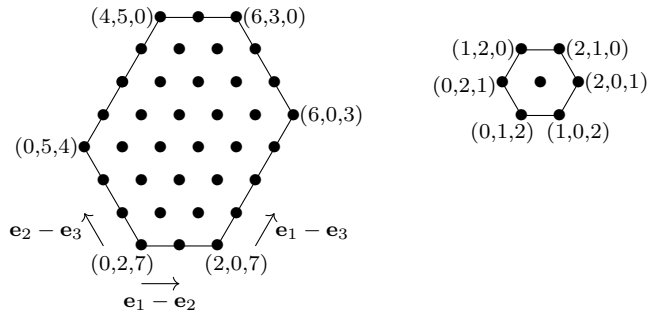
\begin{figure}[ht!]
 \begin{tikzpicture}[scale=0.5]
\draw (0,0)--(2,0)--(4,3.46)--(2.5,6.06)--(0.5,6.06)--(-1.5,2.6)--(0,0);
\draw (0,0) node {$\bullet$};
\draw (1,0) node {$\bullet$};
\draw (2,0) node {$\bullet$};
\draw (-0.5,0.87) node {$\bullet$};
\draw (0.5,0.87) node {$\bullet$};
\draw (1.5,0.87) node {$\bullet$};
\draw (2.5,0.87) node {$\bullet$};
\draw (-1,1.73) node {$\bullet$};
\draw (0,1.73) node {$\bullet$};
\draw (1,1.73) node {$\bullet$};
\draw (2,1.73) node {$\bullet$};
\draw (3,1.73) node {$\bullet$};
\draw (-1.5,2.6) node {$\bullet$};
\draw (-0.5,2.6) node {$\bullet$};
\draw (0.5,2.6) node {$\bullet$};
\draw (1.5,2.6) node {$\bullet$};
\draw (2.5,2.6) node {$\bullet$};
\draw (3.5,2.6) node {$\bullet$};
\draw (-1,3.46) node {$\bullet$};
\draw (0,3.46) node {$\bullet$};
\draw (1,3.46) node {$\bullet$};
\draw (2,3.46) node {$\bullet$};
\draw (3,3.46) node {$\bullet$};
\draw (4,3.46) node {$\bullet$};
\draw (-0.5,4.33) node {$\bullet$};
\draw (0.5,4.33) node {$\bullet$};
\draw (1.5,4.33) node {$\bullet$};
\draw (2.5,4.33) node {$\bullet$};
\draw (3.5,4.33) node {$\bullet$};
\draw (-0.5,4.33) node {$\bullet$};
\draw (0,5.20) node {$\bullet$};
\draw (1,5.20) node {$\bullet$};
\draw (2,5.20) node {$\bullet$};
\draw (3,5.20) node {$\bullet$};
\draw (0.5,6.06) node {$\bullet$};
\draw (1.5,6.06) node {$\bullet$};
\draw (2.5,6.06) node {$\bullet$};

\draw [->] (0,-1)--(1,-1);
\draw (0.5,-1.5) node {\scriptsize $\ee_1-\ee_2$};
\draw  [->] (-1,0)--(-1.5,0.87);
\draw (-2.5,0.4) node {\scriptsize $\ee_2-\ee_3$};

\draw  [->] (3,0)--(3.5,0.87);
\draw (4.5,0.4) node {\scriptsize $\ee_1-\ee_3$};

\draw (-0.5,-0.5) node {\scriptsize (0,2,7)};
\draw (2.5,-0.5) node {\scriptsize (2,0,7)};
\draw (5,3.46) node {\scriptsize (6,0,3)};
\draw (3.5,6.06) node {\scriptsize (6,3,0)};
\draw (-0.5,6.06) node {\scriptsize (4,5,0)};
\draw (-2.5,2.6) node {\scriptsize (0,5,4)};

\draw (10,3.46)--(11,3.46)--(11.5,4.33)--(11,5.20)--(10,5.20)--(9.5,4.33)--(10,3.46);

\draw (10,3.46) node {$\bullet$};
\draw (11,3.46) node {$\bullet$};
\draw (9.5,4.33) node {$\bullet$};
\draw (10.5,4.33) node {$\bullet$};
\draw (11.5,4.33) node {$\bullet$};
\draw (10,5.20) node {$\bullet$};
\draw (11,5.20) node {$\bullet$};

\draw (9.5,3) node {\scriptsize (0,1,2)};
\draw (11.5,3) node {\scriptsize (1,0,2)};
\draw (12.5,4.33) node {\scriptsize (2,0,1)};
\draw (12,5.20) node {\scriptsize (2,1,0)};
\draw (9,5.20) node {\scriptsize (1,2,0)};
\draw (8.5,4.22) node {\scriptsize (0,2,1)};

\end{tikzpicture}
\caption{Polytopes for 3-component positive links with $(\ell_{12},\ell_{23},\ell_{13})=(2,3,4)$ (left) and $(1,1,1)$ (right).
These are Minkowski sums of intervals $[(2,0,0),(0,2,0)]+[(0,3,0),(0,0,3)]+[(4,0,0),(0,0,4)]$ and 
$[(1,0,0),(0,1,0)]+[(0,1,0),(0,0,1)]+[(1,0,0),(0,0,1)]$, respectively.}
\label{fig:polytopes}
\end{figure}

\begin{theorem}
\label{thm: polytope}
After a shift by the vector $\frac{1}{2}(\ell_1,\ldots,\ell_n)$, the Alexander degrees of splitting maps \eqref{eq: deg alex positive} of a positive link
correspond to the integer points in the polytope $P_L$. Conversely, any such integer point corresponds to at least one nontrivial splitting map.
\end{theorem}

\begin{proof}
By Lemma \ref{lem: many crossing changes} all compositions of splitting maps are nontrivial, and for positive links all such maps have homological degree zero, but we need to understand their Alexander degrees. 

 By varying $P_{ij}^{0}$ and $P_{ij}^{-1}$, the terms
$\frac{1}{2}\varepsilon_{ij}(\ee_i-\ee_j)$ can have the values
$$
\frac{1}{2}\ell_{ij}(\ee_i-\ee_j),\frac{1}{2}(\ell_{ij}-2)(\ee_i-\ee_j),\ldots,-\frac{1}{2}\ell_{ij}(\ee_i-\ee_j).
$$
By shifting these by $\frac{1}{2}\ell_{ij}(\ee_i+\ee_j)$, we get the points
$$
\ell_{ij}\ee_i,\ (\ell_{ij}-1)\ee_i+\ee_j,\ldots,\ell_{ij}\ee_j
$$
which coincide with the set of integer points on the interval $[\ell_{ij}\ee_i,\ell_{ij}\ee_j]$. By adding these degrees over all $i<j$, we obtain an integer point in $P_L$, and the overall shift equals
$$
\sum_{i<j}\frac{1}{2}\ell_{ij}(\ee_i+\ee_j)=\frac{1}{2}(\ell_1,\ldots,\ell_n).
$$  

It remains to prove that any integer point in $P_L$ can be obtained as a sum of integer points in the intervals $[\ell_{ij}\ee_i,\ell_{ij}\ee_j]$. This follows from the combinatorial results in \cite[Section 9]{BR}. Indeed, by \cite[Lemma 9.1]{BR} the polytope $P_L$ can be decomposed into a disjoint union of parallelepipeds of various dimensions labeled by the linearly independent subsets of the set $\{\ee_i-\ee_j\ :\ i<j\}$ (the edges of each parallelepiped have integer length $\ell_{ij}$), see Figure \ref{fig: decompose P}. These parallelepipeds can be themselves decomposed into smaller parallelepipeds with edges of integer length 1. It is easy to see \cite[Lemma 9.6]{BR} that the linearly independent subsets correspond to forests on $n$ vertices, and by \cite[Theorem 9.5]{BR} the relative volume of each small parallelepiped equals 1. Therefore every vector connecting an integer point inside each large parallelepiped with a vertex can be written as a linear combination of vectors along the edges and the result follows. 
\end{proof}

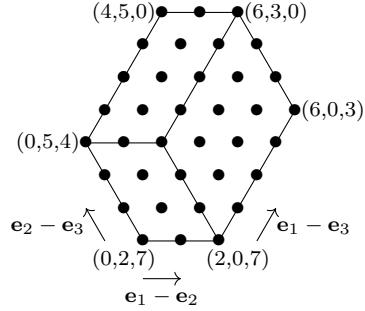
\begin{figure}[ht!]
 \begin{tikzpicture}[scale=0.5]
\draw (0,0)--(2,0)--(4,3.46)--(2.5,6.06)--(0.5,6.06)--(-1.5,2.6)--(0,0);
\draw (0,0) node {$\bullet$};
\draw (1,0) node {$\bullet$};
\draw (2,0) node {$\bullet$};
\draw (-0.5,0.87) node {$\bullet$};
\draw (0.5,0.87) node {$\bullet$};
\draw (1.5,0.87) node {$\bullet$};
\draw (2.5,0.87) node {$\bullet$};
\draw (-1,1.73) node {$\bullet$};
\draw (0,1.73) node {$\bullet$};
\draw (1,1.73) node {$\bullet$};
\draw (2,1.73) node {$\bullet$};
\draw (3,1.73) node {$\bullet$};
\draw (-1.5,2.6) node {$\bullet$};
\draw (-0.5,2.6) node {$\bullet$};
\draw (0.5,2.6) node {$\bullet$};
\draw (1.5,2.6) node {$\bullet$};
\draw (2.5,2.6) node {$\bullet$};
\draw (3.5,2.6) node {$\bullet$};
\draw (-1,3.46) node {$\bullet$};
\draw (0,3.46) node {$\bullet$};
\draw (1,3.46) node {$\bullet$};
\draw (2,3.46) node {$\bullet$};
\draw (3,3.46) node {$\bullet$};
\draw (4,3.46) node {$\bullet$};
\draw (-0.5,4.33) node {$\bullet$};
\draw (0.5,4.33) node {$\bullet$};
\draw (1.5,4.33) node {$\bullet$};
\draw (2.5,4.33) node {$\bullet$};
\draw (3.5,4.33) node {$\bullet$};
\draw (-0.5,4.33) node {$\bullet$};
\draw (0,5.20) node {$\bullet$};
\draw (1,5.20) node {$\bullet$};
\draw (2,5.20) node {$\bullet$};
\draw (3,5.20) node {$\bullet$};
\draw (0.5,6.06) node {$\bullet$};
\draw (1.5,6.06) node {$\bullet$};
\draw (2.5,6.06) node {$\bullet$};

\draw [->] (0,-1)--(1,-1);
\draw (0.5,-1.5) node {\scriptsize $\ee_1-\ee_2$};
\draw  [->] (-1,0)--(-1.5,0.87);
\draw (-2.5,0.4) node {\scriptsize $\ee_2-\ee_3$};

\draw  [->] (3,0)--(3.5,0.87);
\draw (4.5,0.4) node {\scriptsize $\ee_1-\ee_3$};

\draw (-0.5,-0.5) node {\scriptsize (0,2,7)};
\draw (2.5,-0.5) node {\scriptsize (2,0,7)};
\draw (5,3.46) node {\scriptsize (6,0,3)};
\draw (3.5,6.06) node {\scriptsize (6,3,0)};
\draw (-0.5,6.06) node {\scriptsize (4,5,0)};
\draw (-2.5,2.6) node {\scriptsize (0,5,4)};

\draw (-1.5,2.6)--(0.5,2.6)--(2,0);
\draw (0.5,2.6)--(2.5,6.06);
\end{tikzpicture}
\caption{A decomposition of $P_L$ into parallelepipeds.}
\label{fig: decompose P}
\end{figure}

\begin{remark}
In principle, there could be several splitting maps of the same degree. For example, for $(\ell_{12},\ell_{23},\ell_{13})=(1,1,1)$ 
we can write the point $(1,1,1)$ at the center as a sum of integer points on intervals in two ways:
$$
(1,1,1)=(1,0,0)+(0,1,0)+(0,0,1)=(0,1,0)+(0,0,1)+(1,0,0).
$$
\end{remark}

\subsection{L-space links}
\label{sec: L space splitting}

If $L$ is an L-space link, then some results simplify significantly.

\begin{theorem}
\label{thm: L space splitting}
Suppose $L$ is an L-space link. Then:
\begin{itemize}
\item[(a)] For any choice of crossing changes and the maps $\psi_k,\phi_k$ at the crossings, the resulting map $F:\cHFL(L)\to \cHFL(\Split(L))$ is completely determined by its Alexander and Maslov degrees. 

\item[(b)] If, in addition, all crossings between the different components of $L$ are positive, the splitting maps all have homological degree zero and are in bijection with the integer points in the polytope $P_L$. Two splitting maps of the same Alexander degree coincide. 
\end{itemize}
\end{theorem}

\begin{proof}
If $L$ is an L-space link then by \cite{Yajing} all its components $L_i$ are L-space knots. In particular, the corresponding split link 
$\Split(L)$ is an $L$-space link as well.

By Corollary \ref{cor: L space tower} this means that $\cHFL(L,\k)$ and $\cHFL(\Split(L),\k+A(F))$ are isomorphic to $\F[\UU]$, and there is a unique nonzero map between two copies of $\F[\UU]$ of a given degree. The splitting map is nonzero by Lemma \ref{lem: many crossing changes}, so this implies (a).

Part (b) follows from (a) and Theorem \ref{thm: polytope}.
\end{proof}

\subsection{Torsion estimates}
Recall that  the \emph{splitting number} $sp(L)$ is the minimal number of crossings between different components of a link $L$ that should be changed to turn $L$ into the split link.   Let $n_{ij}$ be the number of crossing changes between the $i$-th and $j$-th components, so that the resulting link is the split link. 

Consider $(2n)$-dimensional lattice $\Z^{2n}$ with basis $\ee_1,\ff_1,\ee_2,\ff_2,\ldots,\ee_n,\ff_n$. A point in this lattice parametrizes a monomial in $U_1,V_1,U_2,V_2,\ldots,U_n,V_n$.
For each pair $i<j$ we consider the 3-dimensional tetrahedron:
$$
T_{ij}:=\left\{a\ee_i+b\ff_i+c\ee_j+d\ff_j\ :\ a+b+c+d=n_{ij},\ a,b,c,d\ge 0\right\}.
$$
Next, we consider the Minkowski sum
$$
T=\sum_{i<j} T_{ij}\subset \Z^{2n}.
$$
\begin{remark}
Note that the intersection of $T$ with the $n$-dimensional sublattice $\mathrm{span}\{\ee_1,\ee_2,\cdots,\ee_n\}$ is the Minkowski sum of segments $[n_{ij}\ee_i,n_{ij}\ee_j]$ which is similar to the generalized permutahedra $P_L$ above.
\end{remark}

\begin{theorem}
\label{thm:torsionkill}
Suppose that $L$ is an $n$-component  link where the components are L-space knots. Suppose that  $n_{ij}$ is the number of crossing changes between the $i$-th and $j$-th components so that the resulting link is a split link. Then:
\begin{itemize}

\item[(a)] For any integer point in the polytope $T \subset \Z^{2n}$ the corresponding monomial in $U_i,V_i$ annihilates any torsion element in $\cHFL(L)$.

\item[(b)] The monomial $\UU^{\sum_{i<j} \left\lceil \frac{n_{ij}}{2} \right\rceil}$ annihilates any torsion element in $\cHFL(L)$.

\end{itemize}


\end{theorem}

\begin{proof}
(a) Given an integer point in $T$, we construct two maps  
$$F: \cHFL(L)\rightarrow \cHFL(\Split(L)),\ F': \cHFL(\Split(L))\rightarrow \cHFL(L)$$
as follows. Each time one changes a positive crossing to a negative crossing between $L_i$ and $L_j$, we choose either $\psi_{0}$ or $\psi_{-1}$ for $F$ and  either $\phi_{0}$ or $\phi_{1}$ for $F'$; if one changes a negative crossing to a positive, the roles of $F$ and $F'$ are switched. Specifically, given nonnegative integers $a_{ij},b_{ij},c_{ij},d_{ij}$ such that $a_{ij}+b_{ij}+c_{ij}+d_{ij}=n_{ij}$ (or an integer point in the tetrahedron $T_{ij}$) we can arrange the maps  so that 
\begin{itemize}
\item To $a_{ij}$ crossings between $L_i,L_j$ we associate $\psi_{-1},\phi_1$
\item To $b_{ij}$ crossings between $L_i,L_j$ we associate $\psi_{0},\phi_0$
\item To $c_{ij}$ crossings between $L_i,L_j$ we associate $\psi_{0},\phi_1$
\item To $d_{ij}$ crossings between $L_i,L_j$ we associate $\psi_{-1},\phi_0$
\end{itemize}
By composing those crossing change maps, we obtain $F$ and $F'$. 
By Proposition \ref{prop: psi phi} we get
$$
F'\circ F=\prod_{i<j}\left(U_i^{a_{ij}}V_i^{b_{ij}}U_j^{c_{ij}}V_{j}^{d_{ij}}\right).
$$
The right hand side defines a monomial corresponding to a point in $\sum_{i<j} T_{ij}=T$.
Note that the composition $F'\circ F$ does not depend on the order of the crossings, just the number of crossings of each type.

 If $x$ is a torsion element in $\cHFL(L)$, then $F(x)=0$ since there are no torsion elements in $\cHFL(\Split(L))$, therefore
$$
\prod_{i<j}U_i^{a_{ij}}V_i^{b_{ij}}U_j^{c_{ij}}V_{j}^{d_{ij}}\cdot x=0.
$$

(b) We can choose  $a_{ij}=b_{ij}=\left\lceil \frac{n_{ij}}{2} \right\rceil, c_{ij}=d_{ij}=0$, then $a_{ij}+b_{ij}\ge n_{ij}$ and 
$$
\left(U_i^{a_{ij}}V_i^{b_{ij}}U_j^{c_{ij}}V_{j}^{d_{ij}}\right)=\UU^{\left\lceil \frac{n_{ij}}{2} \right\rceil}.
$$ By part (a),
the monomial 
$$
\prod_{i<j}\left(U_i^{a_{ij}}V_i^{b_{ij}}U_j^{c_{ij}}V_{j}^{d_{ij}}\right)=\UU^{\sum_{i<j} \left\lceil \frac{n_{ij}}{2} \right\rceil}
$$ 
annihilates any torsion element in $\cHFL(L)$.

\end{proof}

\begin{corollary}
\label{coro:2com}
Suppose that $L=L_1\cup L_2$ is a 2-component link with splitting number $sp(L)=n$, and $L_1,L_2$ are L-space knots. Then for any $a,b,c,d\ge 0$ such that $a+b+c+d=n$ the monomial $U_1^{a}V_1^{b}U_2^{c}V_2^{d}$ annihilates any torsion element in $\cHFL(L)$. 
\end{corollary}

\begin{example}
We consider the following boundary links. Given any knot $K$ as in Figure \ref{bing},  let $B(K, n)$ be the new 2-component link obtained by applying an $n$-twisted Bing doubling to $K$. Observe that $B(K, n)$ is a boundary link with unknotted components. The linking number of $B(K, n)$ is $0$, consisting of a positive crossing and a negative crossing between the link components. So one has to change both of the crossings to split the link, and the splitting number is 2 automatically. Hence, by Corollary \ref{coro:2com} any monomial $U_1^a V_1^b U_2^c V_2^d$ where $a+b+c+d=2$ annihilates any torsion element of $\cHFL(B(K, n))$  for all $K$ and $n$. In particular, the $U_i$-torsion and $V_i$-torsion are bounded by $2$ for all $K$ and $n$ where $i=1, 2$. 
\end{example}
\begin{figure}[h]
\centering
\begin{tikzpicture}
    \node[anchor=south west,inner sep=0] at (0,0) {\includegraphics[width=4in]{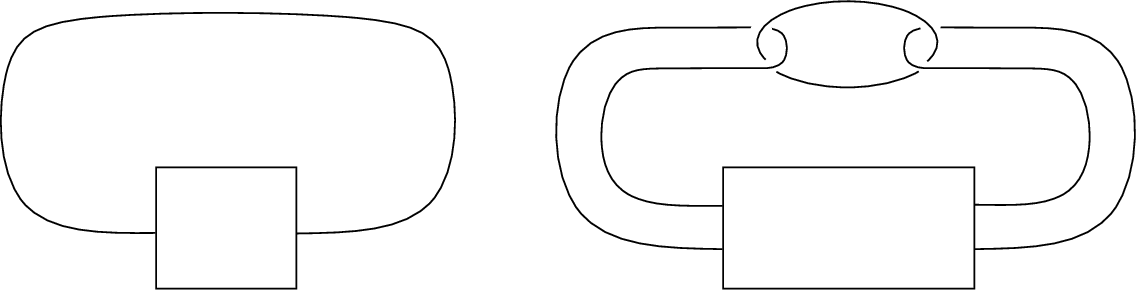}};
    \node[label=right:{$K$}] at (1.6,0.5){};
    \node[label=right:{$K,n$}] at (7,0.5){};
\end{tikzpicture}
\caption{The left figure is $K$ and right figure is the $n$-twisted Bing doubling of $K$, $B(K,n)$.} \label{bing}
\end{figure}

\section{Example: $T(n, n)$}

We illustrate all of the above constructions in detail for the torus link $T(n,n)$ with $n$ components. By \cite{GH} it is an L-space link and its Heegaard Floer homology is given by Theorem \ref{thm:toruslink}.

\subsection{Integer points in permutahedra and splitting maps}

 The {\em permutahedron} $P_n$ is the convex hull of points $(\sigma(n)-1,\ldots,\sigma(1)-1)$ for all permutations $\sigma\in S_n$. It is a convex polytope of dimension $n-1$, and it is easy to see that the center of $P_n$ is at the point $\left(\frac{n-1}{2},\ldots,\frac{n-1}{2}\right)$. By \cite[Theorem 9.4]{BR} the permutahedron $P_n$ is the Minkowski sum of segments $[\ee_i,\ee_j]$ and hence agrees with the link zonotope $P_{T(n,n)}$.   

Following Theorem \ref{thm: polytope} we will be interested in the integer points in $P_n$. For example,
$P_2$ is a segment connecting $(1,0)$ and $(0,1)$. Next, $P_3$ is a hexagon with vertices obtained by permutations of $(2,1,0)$
which contains one more integer point $(1,1,1)$ at its center, see Figure \ref{fig:polytopes} (right). 

For $n=4$ we have a 3-dimensional polytope with 24 vertices corresponding to permutations of $(3,2,1,0)$. Additionally, it contains 6 permutations of 
$$
(2,2,1,1)=\frac{1}{2}(3,1,2,0)+\frac{1}{2}(1,3,0,2),
$$
and 4 permutations of both
$$
(2,2,2,0)=\frac{1}{2}(3,2,1,0)+\frac{1}{2}(1,2,3,0),\quad (3,1,1,1)=\frac{1}{2}(3,2,1,0)+\frac{1}{2}(3,0,1,2).
$$
In total we get $38=24+6+4+4$ integer points in $P_4$. In general, it is known that the integer points in $P_n$ correspond to {\bf forests} on $n$ labeled vertices, but we will not need this. We refer to \cite[Chapter 9]{BR} for more information on  permutahedra and integer points in them.

For all $i<j$ we define the vector\footnote{The reader might recognize the root system of type $A_{n-1}$.}:
$$
\alpha_{ij}=\left(0,\ldots,0,\frac{1}{2},0\ldots,0,-\frac{1}{2},0,\ldots,0\right)=\frac{1}{2}(\ee_i-\ee_j).
$$
\begin{lemma}\label{lem: points} For any integer $n>0$,
\begin{itemize}
\item[(a)] We have
$$
\left(\frac{n-1}{2},\ldots,\frac{n-1}{2}\right)+\sum_{i<j}\alpha_{ij}=(n-1,n-2,\ldots,1,0).
$$
\item[(b)] For any choice of signs $\varepsilon=(\varepsilon_{ij})_{i<j}\in \{\pm 1\}^{\binom{n}{2}}$ define
$$
p_{\varepsilon}=\left(\frac{n-1}{2},\ldots,\frac{n-1}{2}\right)+\sum_{i<j}\varepsilon_{ij}\alpha_{ij}
$$
Then $p_{\varepsilon}$ is an integer point in the permutahedron $P_n$ and all integer points in $P_n$ can be obtained this way.
\end{itemize}

\end{lemma}

\begin{proof}
Part (a) is clear. Part (b) follows from the description of $P_n$ as a zonotope, that is, Minkowski sum of segments 
$$
[\ee_i,\ee_j]=\frac{1}{2}(\ee_i+\ee_j)+[\alpha_{ij},-\alpha_{ij}],
$$
see \cite[Theorems 9.4, 9.5]{BR}. 

\end{proof}

 Recall that by Theorem \ref{thm:toruslink} $\cHFL(T(n,n))$ has $n$ generators $a_0,\ldots,a_{n-1}$.

\begin{theorem}
\label{thm: psi monomial}
Let $\varepsilon=(\varepsilon_{ij})_{i<j}$ be a choice of signs as above. Choose a minimal sequence of crossings changes that splits $T(n,n)$. For any $1\le i<j\le n$, this sequence contains exactly one crossing change between $L_i$ and $L_j$. Consider the local crossing change map
$$
\Psi_{\varepsilon_{ij}}=\begin{cases}
\psi_{0}\  & \text{if}\ \varepsilon_{ij}=1\\
\psi_{-1}\  & \text{if}\ \varepsilon_{ij}=-1.
\end{cases}
$$
and compose them to define the splitting map $\Omega_{\varepsilon}:\cHFL(T(n,n))\to \cHFL(O_n)$. Then $\Omega_{\varepsilon}$ is independent of the crossing change sequence and satisfies the following properties:
\begin{itemize}
\item[(a)] The Alexander multi-degree of $\Omega_{\varepsilon}$ equals $\sum_{i<j}\varepsilon_{ij}\alpha_{ij}$ and the homological degree  $\gr_{\w}(\Omega_{\varepsilon})$ is zero. 

\item[(b)] $\Omega_{\varepsilon}$ sends the generator $a_0\in \cHFL(T(n,n))$ to $V^{p_{\varepsilon}}$ and every other generator $a_k$ to some other monomials in $R_{UV}$.

\item[(c)] If $p_{\varepsilon}=p_{\varepsilon'}$ then $\Omega_{\varepsilon}=\Omega_{\varepsilon'}$. In other words, the splitting maps for $T(n,n)$ are parametrized by the integer points in the permutahedron $P_n$.

\item[(d)] For $\varepsilon=(+1,\ldots,+1)$  the map $\Omega_{\bf 1}=\Omega_{+1,\ldots,+1}$ is defined on generators by the equation:
$$
\Omega_{\bf 1}(a_k)=V_1^{n-1-k}V_2^{n-2-k}\cdots V_{n-1-k}U_{n+1-k}\cdots U_n^{k}.
$$

\item[(e)] For any permutation $\sigma\in S_n$  there is a map $\Omega_{\sigma}$ corresponding to a vertex of $P_n$. It is obtained from $\Omega_{\bf 1}$  by permuting the indices of $U_i$ and $V_i$ by $\sigma$:
$$
\Omega_{\sigma}(a_k)=V_{\sigma(1)}^{n-1-k}V_{\sigma(2)}^{n-2-k}\cdots V_{\sigma(n-1-k)}U_{\sigma(n+1-k)}\cdots U_{\sigma(n)}^{k}.
$$
\end{itemize}
\end{theorem}
 
\begin{proof}
Part (a) is clear and (c) follows from Theorem \ref{thm: L space splitting}. Part (b) is immediate from (a) since $a_0$ has Alexander degree $\left(\frac{n-1}{2},\ldots,\frac{n-1}{2}\right)$ and $\gr_{\w}(a_0)=0$.  This agrees with the normalization of the generator of $\cHFL(O_n)$ and neither $\Omega_{\varepsilon}$ nor $V_i$ change $\gr_{\w}$.

Part (d) also follows from Theorem \ref{thm: L space splitting} since we can compare the Alexander and Maslov degrees on both sides. Indeed, 
\[
\begin{split}
A(\Omega_{\bf 1}(a_k))&=\sum_{i<j}\alpha_{ij}+A(a_k)=\sum_{i<j}\alpha_{ij}+\left(\frac{n-1}{2},\ldots,\frac{n-1}{2}\right)-(k,\ldots,k)\\
&=(n-1-k,n-2-k,\ldots,1-k,-k).
\end{split}
\]
Here the last equation follows from Lemma \ref{lem: points}(a). Furthermore,
$$
\gr_{\w}(\Omega_{\bf 1}(a_k))=\gr_{\w}(a_k)=-k(k+1),
$$
while
$$
\gr_{\w}\left(V_1^{n-1-k}V_2^{n-2-k}\cdots V_{n-1-k}U_{n+1-k}\cdots U_n^{k}\right)=-2(1+\ldots+k)=-k(k+1).
$$
Part (e) follows from (d) by permuting the variables.
\end{proof}

\begin{example}
For $n=2$ we have $\Omega_{\bf 1}(a_0)=V_1,\Omega_{\bf 1}(a_1)=U_2$ as in Example \ref{ex:Hopf-psi}.
\end{example}

\begin{example}
For $n=3$, the signs $\varepsilon=(+1,+1,+1)$ correspond to the point $(2,1,0)\in P_3$ and  the map 
$\Omega_{\bf 1}(a_0)=V_1^2V_2,\ \Omega_{\bf 1}(a_1)=V_1U_3,\ \Omega_{\bf 1}(a_2)=U_2U_3^2$. 
The maps for other vertices of $P_3$ can be obtained from it by permuting the variables.

The maps  for $(\varepsilon_{12}=+1,\varepsilon_{13}=-1,\varepsilon_{23}=+1)$ and $(\varepsilon_{12}=-1,\varepsilon_{13}=+1,\varepsilon_{23}=-1)$ both correspond to the central point $(1,1,1)\in P_3$ and the corresponding splitting map is given by following:
$$
\Omega_{\varepsilon}(a_0)=V_1V_2V_3,\ \Omega_{\varepsilon}(a_1)=\UU,\ \Omega_{\varepsilon}(a_2)=U_1U_2U_3.
$$
\end{example}

\subsection{Surgery map}

In this section, we study the map $\Omega:\cbHFL(T(n,n))\to \cbHFL(O_n)$ obtained by composing the surgery maps $\Psi$ from Section \ref{sec: surgery skein} for any sequence of crossing changes. Specifically, we choose a minimal sequence of crossing changes that splits $T(n,n)$ and we compose the local surgery maps $\Psi$ associated to each crossing change to define $\Omega$. Below we will show that (at least on the level of homology) the choice of crossing change sequence does not matter and resulting maps agree. 

To specify the link components involved in a crossing change, we will use subscripts for $\Psi$ i.e. for a positive crossing between $L_i$ and $L_j$ we denote the local crossing change map by $\Psi_{ij}$. 
Topologically, each  map $\Psi_{ij}$ corresponds to a cobordism $W_{ij}$ obtained attaching a 2-handle, and their composition $\Omega$ corresponds to the composition $W$ of  cobordisms $W_{ij}$. We have $H_2(W_{ij})=\Z$ and $H_2(W)=\Z^{\binom{n}{2}}$.
A choice of a $\Spin$-structure on $W_{ij}$ corresponds to a choice of an integer $m_{ij}$ and a map $\psi_{m_{ij}}$ defined as in Section \ref{sec: surgery}, so that $\Psi_{ij}=\sum_{m_{ij}\in \Z} (-1)^{m_{ij}}\psi_{m_{ij}}$. Similarly, a choice of a $\Spin$-structure on $W$ corresponds to a choice of a vector $(m_{ij})_{i<j}$ in the  $\binom{n}{2}$-dimensional lattice, and
$$
\Omega=\sum_{(m_{ij})\in \Z^{\binom{n}{2}}}\prod_{i<j}  (-1)^{m_{ij}}\psi_{m_{ij}}.
$$
The choices of $m_{ij}=0$ and $m_{ij}=-1$ correspond, respectively, to choices of binary sequences $\varepsilon_{ij}=1$ and $\varepsilon_{ij}=-1$ in previous section. In Theorem \ref{thm: image Psi} below we  prove that $\Omega$ is injective on homology. To describe its image explicitly, we need to introduce some algebraic notations.

\begin{definition}
Let $S\subset \Z_{\ge 0}^{2}$ be a subset of cardinality $n$, order its elements as $s_1=(a_1,b_1),\ldots,s_n=(a_n,b_n)$.  Then we can define the  polynomial
$$
\Delta_S=\sum_{\sigma\in S_n}(-1)^{\sigma}U_1^{a_{\sigma(1)}}V_1^{b_{\sigma(1)}}\cdots U_n^{a_{\sigma(n)}}V_n^{b_{\sigma(n)}}=\det \left(
\begin{matrix}
U_1^{a_1}V_1^{b_1} & \cdots & U_1^{a_n}V_1^{b_n}\\
\vdots & \ddots  & \vdots \\
U_n^{a_1}V_n^{b_1} & \cdots & U_n^{a_n}V_n^{b_n}\\
\end{matrix}
\right).
$$
Reordering the elements of $S$ changes the sign of $\Delta_S$. Sometimes we will use notation $\Delta_S$ for $n$-tuples $S$ where some elements are repeated, in this case $\Delta_S=0$.
\end{definition}

\begin{definition}
We define $\J\subset \cHFL(O_n)=R_{UV}$ as the ideal generated by $\Delta_S$ for all possible $n$-element subsets $S$. 
\end{definition}

\begin{remark}
The polynomials $\Delta_S$ and the ideal in $\C[U_1,\ldots,U_n,V_1,\ldots,V_n]$ generated by $\Delta_S$ were first introduced by Haiman in his work on Hilbert scheme of points on the plane \cite{Haiman}.
\end{remark}

The following lemma gives a useful characterization of the ideal $\J$, we postpone its proof until Section \ref{sec: proof lemma minors}. It can be used as an alternative definition of $\J$.

\begin{lemma}
\label{lem: minors}
The ideal $\J$ is generated by the $n$ maximal minors corresponding to $n$-tuples of consecutive columns in the matrix
$$
\left(
\begin{matrix}
U_1^{n-1} & U_1^{n-2} & \cdots & U_1 & 1 & V_1 & \cdots & V_1^{n-2} & V_1^{n-1}\\
\vdots & \vdots & & \vdots & \vdots & \vdots  & & \vdots & \vdots \\
U_n^{n-1} & U_n^{n-2} & \cdots & U_n & 1 & V_n & \cdots & V_n^{n-2} & V_n^{n-1}\\
\end{matrix}
\right)
$$ 
\end{lemma}

It is clear that all the minors in Lemma \ref{lem: minors} are of the form $\Delta_S$ for some choices of subsets $S$.

\begin{example}
For $n=2$ we have two determinants
$$
\det\left(\begin{matrix}
U_1 & 1\\
U_2 & 1
\end{matrix}
\right)=U_1-U_2,
\det\left(\begin{matrix}
1 & V_1\\
1 & V_2
\end{matrix}
\right)=V_2-V_1.
$$
\end{example}

\begin{example}
For $n=3$ we have three determinants
$$
\det\left(\begin{matrix}
U_1^2 & U_1 & 1\\
U_2^2 & U_2 & 1\\
U^2_3 & U_3 & 1
\end{matrix}
\right),
\det\left(\begin{matrix}
U_1 & 1 & V_1\\
U_2 & 1 & V_2\\
U_3 & 1 & V_3
\end{matrix}
\right),\
\det\left(\begin{matrix}
1 & V_1 & V_1^2\\
1 & V_2 & V_2^2\\
1 & V_3 & V_3^2
\end{matrix}
\right).
$$
\end{example}

Now we are ready to state the main theorem of this section.

\begin{theorem}
\label{thm: image Psi}
The surgery map $\Omega:\cbHFL(T(n,n))\to \cbHFL(O_n)$ is injective  and its image coincides with the (completed) ideal $\bJ$. The map does not depend on the order and choices of crossing changes. It particular, 
$$
\cHFL(T(n,n))\simeq \J
$$  
as modules over $\F[U_1,\ldots,U_n,V_1,\ldots,V_n]$.
\end{theorem} 

\begin{proof}
For the reader's convenience, we break the proof into several steps.

{\bf Step 1:} By Lemma \ref{lem: tau}, each map $\Psi_{ij}=\sum_{m_{ij}\in \Z} (-1)^{m_{ij}}\psi_{m_{ij}}$ is proportional, up to an explicit invertible factor $\tau_{ij}$, to $\Psi_{ij}^0=\psi_{0}-\psi_{-1}.$ The factors $\tau_{ij}$ do not depend on the order of crossing changes, and do not affect the injectivity or the image (which is an $\F[U_1,\ldots,U_n,V_1,\ldots,V_n]$-submodule of $\cHFL(O_n)$), so we can ignore them from now on and focus on $m_{ij}\in \{0,-1\}$.

{\bf Step 2:} By following the notations of Theorem \ref{thm: psi monomial}, we can then rewrite the composition of $\Psi_{ij}^0=\psi_{0}-\psi_{-1}$ as
\begin{equation}
\label{eq: psi short}
\Omega^{0}=\sum_{\varepsilon\in \{\pm 1\}^{\binom{n}{2}}}\sgn(\varepsilon)\Omega_{\varepsilon},\quad \sgn(\varepsilon)=\prod_{i<j}\varepsilon_{ij}.
\end{equation}
By Theorem \ref{thm: psi monomial}, for any given $\varepsilon$ the order of composition does not matter.

{\bf Step 3:} The terms in \eqref{eq: psi short} are parametrized by the integer points in the permutahedron $P_n$. However, some points will appear several times for different choices of $\varepsilon$, and by Theorem \ref{thm: psi monomial}(c) the corresponding terms in \eqref{eq: psi short} might cancel as long as they have the same Alexander degree. We claim that in fact the terms for all points will cancel except for  the vertices of $P_n$. To show this, consider the generating function
$$
\sum_{\varepsilon}\sgn(\varepsilon)V^{p_{\varepsilon}}=\sum_{\varepsilon}\sgn(\varepsilon)V_1^{\frac{n-1}{2}}\cdots V_n^{\frac{n-1}{2}}V^{\sum_{i<j}\varepsilon_{ij}\alpha_{ij}}=\sum_{\varepsilon}\prod_{i<j} \varepsilon_{ij}\sqrt{V_iV_j}V^{\varepsilon_{ij}\alpha_{ij}}=\prod_{i<j}(V_i-V_j).
$$
Here, we used that $V_i=\sqrt{V_iV_j}V^{\alpha_{ij}}$ and $-V_j=-\sqrt{V_iV_j}V^{-\alpha_{ij}}$.
On the other hand, we have the Vandermonde determinant
$$
\prod_{i<j}(V_i-V_j)=\pm\det\left(
\begin{matrix}
1 & V_1 & \cdots & V_1^{n-1}\\
\vdots & \vdots & & \vdots\\
1 & V_n & \cdots & V_n^{n-1}
\end{matrix}
\right)=
\pm\sum_{\sigma\in S_n}(-1)^{\sigma}V_1^{\sigma(1)-1}\cdots V_n^{\sigma(n)-1}.
$$

As a conclusion of this step, we can write
$$
\Omega^{0}=\sum_{\varepsilon}\sgn(\varepsilon)\Omega_{\varepsilon}=\pm\sum_{\sigma\in S_n}(-1)^{\sgn(\sigma)}\Omega_{\sigma}.
$$

{\bf Step 4:} We are ready to compute the image of $\Omega$ or, equivalently, of $\Omega^{0}$. Indeed, by Theorem \ref{thm: psi monomial}(d),(e) we get 
\[
\begin{split}
\Omega^{0}(a_j)&=\pm\sum_{\sigma\in S_n}(-1)^{\sgn(\sigma)}\Omega_{\sigma}(a_j)=
\pm\sum_{\sigma\in S_n}(-1)^{\sgn(\sigma)}V_{\sigma(1)}^{n-1-j}V_{\sigma(2)}^{n-2-j}\cdots V_{\sigma(n-1-j)}U_{\sigma(n+1-j)}\cdots U_{\sigma(n)}^{j}\\
&=\pm \det\left(\begin{matrix}
U_1^{j} & \cdots & U_1 & 1 & V_1 & \cdots & V_1^{n-1-j}\\
\vdots &  &\vdots & \vdots & \vdots & & \vdots \\
U_n^{j} & \cdots & U_n & 1 & V_n & \cdots & V_n^{n-1-j}
\end{matrix}
\right),\quad j=0,\ldots,n-1.
\end{split}
\] 
By Lemma \ref{lem: minors} these determinants generate the ideal $\J$.

{\bf Step 5:} It remains to prove that $\Omega^{0}$ is injective. Indeed, $T(n,n)$ is an $L$-space link, so in each Alexander multi-degree $\k$ we have $\cHFL(T(n,n),\k)\cong\F[\UU]$. By the above, the image of any element of this tower under $\Omega^{0}$ is a sum of   elements in  $n!$ towers in $\cHFL(O_n)$ located at the vertices of a permutahedron centered at $\k$. Since all these elements appear with nonzero coefficients, their sum is also nonzero.
\end{proof}

\subsection{Proof of Lemma \ref{lem: minors}}
\label{sec: proof lemma minors}

We start with several results which allow us to simplify the determinants $\Delta_S$. 
Given a subset $S=\{(a_1,b_1),\ldots,(a_n,b_n)\}$, we define $m_i=\min(a_i,b_i)$ and
$$
\widetilde{S}=\left\{(a_1-m_1,b_1-m_1),\ldots,(a_n-m_n,b_n-m_n)\right\}.
$$ 
Note that some elements of $\widetilde{S}$ may coincide even if all elements of $S$ are distinct. The subset $\widetilde{S}$ is contained in the union of the horizontal strip $\{b=0\}$ and the vertical strip $\{a=0\}$, dashed in Figure \ref{fig: S tilde}.

\begin{lemma}
\label{lem: shift down}
We have $\Delta_S=\UU^{N}\Delta_{\widetilde{S}}$ for $N=m_1+\ldots+m_n$.
\end{lemma}

\begin{proof}
For all $i,j$ we have
$$
U_j^{a_i}V_j^{b_i}=U_j^{a_i-m_i}V_j^{b_i-m_i}(U_jV_j)^{m_i}=U_j^{a_i-m_i}V_j^{b_i-m_i}\UU^{m_i}
$$
and
$$
\Delta_S=
\det \left(
\begin{matrix}
U_1^{a_1}V_1^{b_1} & \cdots & U_1^{a_n}V_1^{b_n}\\
\vdots & \ddots  & \vdots \\
U_n^{a_1}V_n^{b_1} & \cdots & U_n^{a_n}V_n^{b_n}\\
\end{matrix}
\right)=
\det \left(
\begin{matrix}
U_1^{a_1-m_1}V_1^{b_1-m_1}\UU^{m_1} & \cdots & U_1^{a_n-m_n}V_1^{b_n-m_n}\UU^{m_n}\\
\vdots & \ddots  & \vdots \\
U_n^{a_1-m_1}V_n^{b_1-m_1}\UU^{m_1} & \cdots & U_n^{a_n-m_n}V_n^{b_n-m_n}\UU^{m_n}\\
\end{matrix}
\right)=
$$
$$
\UU^{m_1+\ldots+m_n}\det \left(
\begin{matrix}
U_1^{a_1-m_1}V_1^{b_1-m_1} & \cdots & U_1^{a_n-m_n}V_1^{b_n-m_n}\\
\vdots & \ddots  & \vdots \\
U_n^{a_1-m_1}V_n^{b_1-m_1} & \cdots & U_n^{a_n-m_n}V_n^{b_n-m_n}\\
\end{matrix}
\right)=\UU^{m_1+\ldots+m_n}\Delta_{\widetilde{S}}.
$$
\end{proof}

\begin{example}
For $S=\{(0,0),(1,2),(3,5),(6,4)\}$ we have $\widetilde{S}=\{(0,0),(0,1),(0,2),(2,0)\}$, see Figure \ref{fig: S tilde}.
In this case $N=8$ and we have
$$
\Delta_S=\det\left(\begin{matrix}
1 & U_1V_1^{2} & U_1^3V_1^5 & U_1^6V_1^4\\
1 & U_2V_2^{2} & U_2^3V_2^5 & U_2^6V_2^4\\
1 & U_3V_3^{2} & U_3^3V_3^5 & U_3^6V_3^4\\
1 & U_4V_4^{2} & U_4^3V_4^5 & U_4^6V_4^4\\
\end{matrix}\right)=\UU^8\det\left(\begin{matrix}
1 & V_1 & V_1^2 & U_1^2\\
1 & V_2 & V_2^2 & U_2^2\\
1 & V_3 & V_3^2 & U_3^2\\
1 & V_4 & V_4^2 & U_4^2\\
\end{matrix}\right)=\UU^8\Delta_{\widetilde{S}}.
$$

\end{example}

\begin{figure}
\begin{tikzpicture}
\filldraw[lightgray] (0,0)--(0,6)--(1,6)--(1,1)--(7,1)--(7,0)--(0,0);
\draw (0,0)--(7,0);
\draw (0,1)--(7,1);
\draw (0,2)--(7,2);
\draw (0,3)--(7,3);
\draw (0,4)--(7,4);
\draw (0,5)--(7,5);
\draw (0,6)--(7,6);
\draw (0,0)--(0,6);
\draw (1,0)--(1,6);
\draw (2,0)--(2,6);
\draw (3,0)--(3,6);
\draw (4,0)--(4,6);
\draw (5,0)--(5,6);
\draw (6,0)--(6,6);
\draw (7,0)--(7,6);
\draw (0.5,-0.3) node {$0$};
\draw (1.5,-0.3) node {$1$};
\draw (2.5,-0.3) node {$2$};
\draw (3.5,-0.3) node {$3$};
\draw (4.5,-0.3) node {$4$};
\draw (5.5,-0.3) node {$5$};
\draw (6.5,-0.3) node {$6$};
\draw (-0.3,0.5) node {$0$};
\draw (-0.3,1.5) node {$1$};
\draw (-0.3,2.5) node {$2$};
\draw (-0.3,3.5) node {$3$};
\draw (-0.3,4.5) node {$4$};
\draw (-0.3,5.5) node {$5$};

\draw (0.5,0.5) node {$\bullet$};
\draw (1.5,2.5) node {$\bullet$};
\draw (3.5,5.5) node {$\bullet$};
\draw (6.5,4.5) node {$\bullet$};

\draw [red] (0.5,0.5) node {$\bullet$};
\draw  [red] (0.5,1.5) node {$\bullet$};
\draw  [red] (0.5,2.5) node {$\bullet$};
\draw  [red] (2.5,0.5) node {$\bullet$};

\draw [dotted,->] (1.5,2.5)--(0.6,1.6);
\draw [dotted,->] (3.5,5.5)--(0.6,2.6);
\draw [dotted,->] (6.5,4.5)--(2.6,0.6);
\end{tikzpicture}
\caption{The sets $S$ and $\widetilde{S}$}
\label{fig: S tilde}
\end{figure}
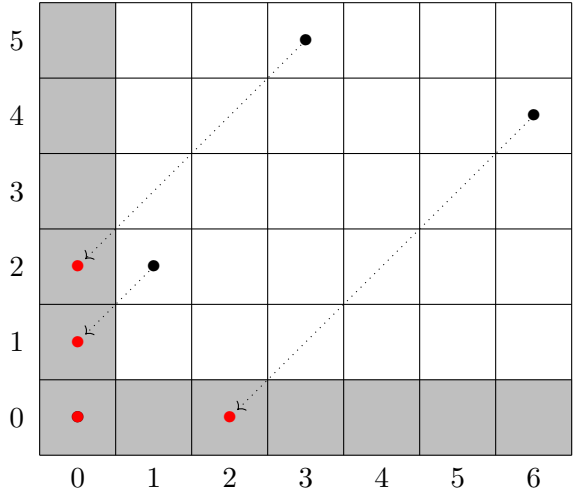

Let $e_k(U)=\sum_{i_1<\ldots<i_k}U_{i_1}\cdots U_{i_k}$ be the $k$-th elementary symmetric function.
We have the following analogue of the Pieri rule for Schur functions \cite{Macd}.

\begin{lemma}
\label{lem: pieri}
We have 
$$
e_k(U)\Delta_S=\sum_{S'}\Delta_{S'}
$$
where $S'$ is obtained   by adding $(1,0)$ to $k$ distinct elements of $S$ and leaving other elements unchanged.
\end{lemma}

\begin{proof}
Given a polynomial $f(U_1,\ldots,U_n,V_1,\ldots,V_n)$, we define
$$
\Alt(f)=\sum_{\sigma\in S_n}(-1)^{\sgn(\sigma)}f(U_{\sigma(1)},\ldots,U_{\sigma(n)},V_{\sigma(1)},\ldots,V_{\sigma(n)}). 
$$
For $S=\{(a_1,b_1),\cdots (a_n,b_n)\}$ we get
$
\Delta_S=\Alt\left(U_1^{a_1}V_1^{b_1}\cdots U_n^{a_n}V_n^{b_n}\right).
$
Clearly, $\Alt(f)$ is linear in $f$ and $\Alt(fg)=\Alt(f)g$ for any symmetric polynomial $g$. Since $e_k(U)$ is symmetric, we get
$$
e_k(U)\Delta_S=e_k(U)\Alt\left(U_1^{a_1}V_1^{b_1}\cdots U_n^{a_n}V_n^{b_n}\right)=\Alt\left(e_k(U)U_1^{a_1}V_1^{b_1}\cdots U_n^{a_n}V_n^{b_n}\right)=
$$
$$
\sum_{i_1<\ldots<i_k}\Alt\left(U_{i_1}\cdots U_{i_k}\cdot U_1^{a_1}V_1^{b_1}\cdots U_n^{a_n}V_n^{b_n}\right)=\sum_{S'}\Delta_{S'}
$$
where 
$S'=\{(a'_1,b'_1),\cdots (a'_n,b'_n)\},$
$$
\quad (a'_{i_s},b'_{i_s})=(a_{i_s},b_{i_s})+(1,0),\quad
(a'_j,b'_j)=(a_j,b_j)\quad j\notin\{i_1,\ldots,i_k\}.
$$
\end{proof}

\begin{example}
\label{ex: pieri}
We have 
$$
e_1(U)\Delta_{\{(0,0),(0,1),(0,2),(1,0)\}}=\Delta_{\{(1,0),(0,1),(0,2),(1,0)\}}+\Delta_{\{(0,0),(1,1),(0,2),(1,0)\}}+$$
$$\Delta_{\{(0,0),(0,1),(1,2),(1,0)\}}+\Delta_{\{(0,0),(0,1),(0,2),(2,0)\}}.
$$
The first term vanishes since $(1,0)$ is repeated twice. By Lemma \ref{lem: shift down} the second term equals 
$$
\Delta_{\{(0,0),(1,1),(0,2),(1,0)\}}=\UU\Delta_{\{(0,0),(0,0),(0,2),(1,0)\}}=0
$$
and the third term equals
$$
\Delta_{\{(0,0),(0,1),(1,2),(1,0)\}}=\UU\Delta_{\{(0,0),(0,1),(0,1),(1,0)\}}=0.
$$
Therefore 
$$
e_1(U)\Delta_{\{(0,0),(0,1),(0,2),(1,0)\}}=\Delta_{\{(0,0),(0,1),(0,2),(2,0)\}}.
$$
\end{example}

\begin{example}
\label{ex: pieri 2}
Similarly, we have 
$$
e_2(U)\Delta_{\{(0,0),(0,1),(0,2),(2,0),(3,0))\}}=\Delta_{\{(1,0),(1,1),(0,2),(2,0),(3,0))\}}+ $$
$$
+\Delta_{\{(1,0),(0,1),(0,2),(2,0),(4,0))\}}+\Delta_{\{(0,0),(0,1),(0,2),(3,0),(4,0))\}}
$$
and all other terms vanish. The first term can be simplified as
$$
\Delta_{\{(1,0),(1,1),(0,2),(2,0),(3,0))\}}=\UU\Delta_{\{(1,0),(0,0),(0,2),(2,0),(3,0))\}}.
$$
\end{example}

\begin{proof}[Proof of Lemma \ref{lem: minors}]
Recall that $\J$ is generated by the determinants $\Delta_S$ for arbitrary subsets $S$.  We need to prove that it is generated by $\Delta_{S_k}$ where
$$
S_k=\left\{(0,0),(1,0),\ldots,(k-1,0),(0,1),\ldots,(0,n-k) \right\},\quad\quad 1\le k\le n.
$$

By Lemma \ref{lem: shift down} we have 
$\Delta_{S}$ proportional to $\Delta_{\widetilde{S}}$ where all elements of $\widetilde{S}$ have the form $(a,0)$ or $(0,b)$ (that is, $\widetilde{S}$ is contained in the dashed region in Figure \ref{fig: S tilde}), so after reordering its elements we can write
$$
\widetilde{S}=\{(a_1,0),\ldots,(a_k,0),(0,b_{k+1})\ldots, (0,b_{n})\},\quad 0\le a_1<\ldots<a_k,\ \ \ 0<b_{k+1}<\ldots<b_n.
$$
It remains to prove that, in fact, it is sufficient to only consider the ``dense" subsets where $a_i=i-1$ and $b_i=i-k$. There is a natural partial order on such $\widetilde{S}$ where $\widetilde{S}'\prec \widetilde{S}''$ if $\widetilde{S}'$ is obtained  
by ``sliding" some elements of $\widetilde{S}''$ left or down. Clearly, ``dense" subsets are minimal in this order.

Assume that $0<a_{j}-1$, $(a_j-1,0)\notin \widetilde{S}$ and $j$ is maximal with this property, that is, $a_j$ bounds the rightmost gap in $\widetilde{S}$.
 This implies that $a_{i}=a_{j}+i-j$ for $j+1\le i\le k$. Let 
$$
\widetilde{S_1}=\{(a_1,0),\ldots,(a_j-1,0),\ldots,(a_k-1,0),(0,b_{k+1})\ldots, (0,b_{n})\}.
$$

In Example \ref{ex: pieri} we have 
$$
\widetilde{S}=\{(0,0),(0,1),(0,2),(2,0)\},\ a_j=2\ \mathrm{and}\ \widetilde{S_1}=\{(0,0),(0,1),(0,2),(1,0)\},
$$
 in Example \ref{ex: pieri 2} we have 
$$
\widetilde{S}=\{(0,0),(0,1),(0,2),(3,0),(4,0)\},\ a_j=3\ \mathrm{and}\ \widetilde{S_1}=\{(0,0),(0,1),(0,2),(2,0),(3,0)\}.
$$
Then by Lemmas \ref{lem: pieri} and   \ref{lem: shift down} (see also Examples \ref{ex: pieri} and \ref{ex: pieri 2}) we have  
$$
e_{k-j+1}(U)\Delta_{\widetilde{S}_1}=\Delta_{\widetilde{S}}+\sum_{S'}\Delta_{S'},\quad \Delta_{S'}=\UU^N\Delta_{\widetilde{S'}}
$$
where $\widetilde{S'}\prec \widetilde{S}$. Therefore $\Delta_{\widetilde{S}}$ belongs to the ideal generated by $\Delta_{\widetilde{S}_1}$ and $\Delta_{\widetilde{S'}}$. We can proceed by induction in the above partial order until $a_j$ are dense. Then repeat the same argument swapping $U_i$ and $V_i$ and using a version of Lemma \ref{lem: pieri} for elementary symmetric functions in $V_i$, this will ensure that $b_j$ are dense as well. 
\end{proof}

\subsection{More on ideal $\J$ and its cousins}
\label{sec: J cousins}

In this section we collect some further facts on the ideal $\J$ and discuss some analogues of Theorem \ref{thm: image Psi}  for other homology theories. 

\begin{definition}
A polynomial $f(U_1,\ldots,U_n,V_1,\ldots,V_n)$ is called antisymmetric if
$$
f(U_{\sigma(1)},\ldots,U_{\sigma(n)},V_{\sigma(1)},\ldots,V_{\sigma(n)})=(-1)^{\sigma}f(U_1,\ldots,U_n,V_1,\ldots,V_n).
$$
\end{definition}

Note that over our ground field $\F$ of characteristic 2, any antisymmetric polynomial is also symmetric. However, for other ground fields there is an important distinction.

\begin{lemma}
Let $\ch\ \mathbf{k}\neq 2$. Then any antisymmetric polynomial with coefficients in $\mathbf{k}$  is a linear combination of minors $\Delta_S$.
\end{lemma} 

\begin{proof}
Let $f$ be an antisymmetric polynomial, and $U_1^{a_1}V_1^{b_1}\cdots U_n^{a_n}V_n^{b_n}$ be a monomial in $f$ with nonzero coefficient. If $a_i=a_j$ and $b_i=b_j$ then the transposition $(i\ j)$ fixes this monomial, but since $f$ is antisymmetric it must change its sign, contradiction. Therefore all pairs $(a_1,b_1),\ldots,(a_n,b_n)$ are distinct and the $S_n$-orbit of this monomial adds up to $\Delta_S$ for $S=\{ (a_1,b_1),\ldots,(a_n,b_n)\}$.
\end{proof}

Informally, we can think of $\J$ as a characteristic 2 reduction of the ideal generated by antisymmetric polynomials in $\cHFL(O_n)$.
Next, we check directly that Theorem \ref{thm: image Psi} is compatible with Theorem  \ref{thm:toruslink}. By Lemma \ref{lem: minors} the ideal $\J$ is generated by $n$ minors $\Omega^0(a_k)$ and it is sufficient to check the relations between them.

\begin{lemma}
The determinants  
$$
\Omega^0(a_j)=\pm \det\left(\begin{matrix}
U_1^{j} & \cdots & U_1 & 1 & V_1 & \cdots & V_1^{n-1-j}\\
\vdots &  &\vdots & \vdots & \vdots & & \vdots \\
U_n^{j} & \cdots & U_n & 1 & V_n & \cdots & V_n^{n-1-j}
\end{matrix}
\right)
$$ 
satisfy the relations \eqref{eq: Tnn} (up to signs).
\end{lemma}

\begin{proof}
Recall that the determinants $\Omega^0(a_k)$ are antisymmetrizations of monomials
$$
\Omega_{\bf 1}(a_k)=V_1^{n-1-k}\cdots V_{n-1-k}U_{n+1-k}\cdots U_n^{k},\ \Omega_{\bf 1}(a_{k+1})=V_1^{n-2-k}\cdots V_{n-2-k}U_{n-k}\cdots U_n^{k+1},
$$
let us check the relations \eqref{eq: Tnn} between them for all possible $I$. Let us first pick $I=\{n-k,\ldots,n\}$, then $\overline{I}=\{1,\ldots,n-k-1\}$ and
$$
\Omega_{\bf 1}(a_k)U_I=\left(V_1^{n-1-k}\cdots V_{n-1-k}U_{n+1-k}\cdots U_n^{k}\right)\times U_{n-k}\cdots U_n=
$$
$$
\left(V_1^{n-2-k}\cdots V_{n-2-k}U_{n-k}\cdots U_n^{k+1}\right)\times V_1\cdots V_{n-k-1}=\Omega_{\bf 1}(a_{k+1})V_{\overline{I}}.
$$  
More generally, for an arbitrary $I$ define 
$$
I_1=I\cap \{1,\ldots,n-1-k\},\ I_2=I\cap \{n-k,\ldots,n\},
$$ 
$$
\overline{I}_1=\{1,\ldots,n-1-k\}\setminus I,\ \overline{I}_2=\{n-k,\ldots,n\}\setminus I.
$$
Let $X_k=V_1^{n-2-k}\cdots V_{n-2-k}U_{n+1-k}\cdots U_n^{k}$, then
$$
\Omega_{\bf 1}(a_k)U_I=X_kV_{\{1,\ldots,n-1-k\}}U_I=X_k\left(V_{I_1}V_{\overline{I}_1}\right)\left(U_{I_1}U_{I_2}\right)=X_kV_{\overline{I}_1}U_{I_2}\left(V_{I_1}U_{I_1}\right)=$$
$$X_kV_{\overline{I}_1}U_{I_2}\left(V_{\overline{I}_2}U_{\overline{I}_2}\right)=X_k\left(V_{\overline{I}_1}V_{\overline{I}_2}\right)\left(U_{I_2}U_{\overline{I}_2}\right)=X_kV_{\overline{I}}U_{\{n-k,\ldots,n\}}=\Omega_{\bf 1}(a_{k+1})V_{\overline{I}}.
$$ 
Here we used the relation $V_{I_1}U_{I_1}=V_{\overline{I}_2}U_{\overline{I}_2}$.

Since the relations \eqref{eq: Tnn} are $S_n$-equivariant, the relations for $\Omega_{\bf 1}(a_k)$ imply the same relations for $\Omega^0(a_k)$.
\end{proof}

Next, we study the relations between the ideals corresponding to the link $T(n,n)$ and its sublinks.

\begin{lemma}
For any subset $I\subset \{1,\ldots,n\}$ with $|I|\ge 2$ let $\J_I$ be the ideal generated by the monomial minors in variables $U_i,V_i,i\in I$. Then $\J\subset \J_I$.
\end{lemma}

\begin{proof}
Let $S=\{ (a_1,b_1),\ldots,(a_n,b_n)\}$, then one can write the minor $\Delta_S$ as follows:
$$
\Delta_S=\sum_{L}\pm \det(U_i^{a_j}V_i^{b_j})_{i\in I,j\in L}\det(U_i^{a_j}V_i^{b_j})_{i\notin I,j\notin L}
$$
where the sum runs over all $|I|$-element subsets $L\subset \{1,\ldots,n\}$. Since  $\det(U_i^{a_j}V_i^{b_j})_{i\in I,j\in L}\in \J_I$, we get $\Delta_S\in \J_I$ and $\J\subset \J_I$.
\end{proof}

\begin{remark}
Topologically, the ideal $\J_I$ corresponds to the union of the sublink $\L_I$ formed by the components $L_i,i\in I$ of $T(n,n)$ with $n-|I|$ unknotted disjoint circles. Since the splitting map for $T(n,n)$ does not depend on the order of crossing changes, we can first split off the components with indices not in $I$, and the splitting map $\Omega$ (respectively, $\Omega^0$) will factor through the splitting map $\Omega_I$ (resp. $\Omega_I^0$) for the resulting link.  Therefore the image of $\Omega$ is contained in the image of $\Omega_I$. 
\end{remark}

\begin{corollary}
We have $$\J\subset \bigcap_{i<j}(U_i-U_j,V_i-V_j)\subset \cHFL(O_n)=\frac{\F[U_1,\ldots,U_n,V_1,\ldots,V_n]}{(U_iV_i=U_jV_j,\ i\neq j)}.$$
\end{corollary}

\begin{proof}
By Lemma \ref{lem: minors}, for a two-component sublink $L_i\cup L_j$ we have $\J_{ij}=(U_i-U_j,V_i-V_j)$, and $\J\subset \J_{ij}$ for all $i<j$. 
\end{proof}

The analogues of the ideal $\J$ and of Theorem \ref{thm: image Psi} appeared in triply graded Khovanov-Rozansky homology. Recall that the triply graded homology of the unknot is $\HHH(O_1)=\C[x,\theta]$ where $x$ is an even and $\theta$ is an odd variable. There is also a skein exact triangle 
$$
\begin{tikzcd}
\HHH(L_+) \arrow{rr}{\Psi_{ij}^{\HHH}}& & \HHH(L_-)\arrow{dl}\\
 & \left(\HHH(L_0)\to  \HHH(L_0)\right) \arrow{ul}& 
\end{tikzcd}
$$
analogous to the skein exact triangle \eqref{eq: skein}. The map $\Psi_{ij}^{\HHH}$ can be used to define a splitting map $\Omega^{\HHH}:\HHH(T(n,n))\to \HHH(O_n)=\C[x_1,\ldots,x_n,\theta_1,\ldots,\theta_n]$ which is, unfortunately, not injective.

To resolve this problem, second author and Hogancamp introduced in \cite{GHog} a deformation, or $y$-ification of Khovanov-Rozansky homology $\HY(L)$. The skein exact triangle can be defined in this deformed theory, and there is a (unique up to homotopy) splitting map $\Psi_{ij}^{\HY}:\HY(L_+)\to \HY(L_-)$. By composing these, one obtains a 
splitting map 
$$
\Omega^{\HY}:\HY(T(n,n))\to \HY(O_n)=\C[x_1,\ldots,x_n,y_1,\ldots,y_n,\theta_1,\ldots,\theta_n].
$$
\begin{theorem}[\cite{GHog}]
The map $\Omega_{\HY}$ is injective, and its image coincides with the ideal $\J^{\HY}$ generated by antisymmetric polynomials in $\HY(O_n)$. Furthermore, the image of $\Omega_{\HY}$ coincides with the ideal
$$
\bigcap_{i<j}\J^{\HY}_{ij}=\bigcap_{i<j}(x_i-x_j,y_i-y_j,\theta_i-\theta_j)\subset \C[x_1,\ldots,x_n,y_1,\ldots,y_n,\theta_1,\ldots,\theta_n].
$$
\end{theorem}

\begin{remark}
Recently Hogancamp, Rose and Wedrich \cite{HRW} studied $y$-ification and the splitting maps for colored triply graded homology, and obtained similar ideals. See \cite[Conjecture 8.12,Theorem 9.33]{HRW} for more details.
\end{remark}
 
In \cite{BS} Batson and Seed defined a deformation of Khovanov homology, which was generalized in \cite{CK} by Cautis and Kamnitzer to $\mathfrak{sl}(N)$ Khovanov-Rozansky homology. Their constructions predate and motivate the construction of $\HY$, and we commonly refer to them as $y$-ified link homology.

\begin{problem}
Define the analogues of the splitting map $\Omega_{Kh},\Omega_{\mathfrak{sl}(N)}:T(n,n)\to O_n$ for $y$-ified Khovanov and $\mathfrak{sl}(N)$ homologies. Is it possible to describe their images as some determinantal ideals in the homology of unlink?
\end{problem}

\begin{problem}
The main result of \cite{BPRW} (following the earlier work in \cite{Dowlin,Gilmore}) defines a spectral sequence from the reduced version of $\HHH$ to $\widehat{\HFK}$ for knots. Is it possible to extend this to a spectral sequence from $\HY$ to $\cHFL$ for arbitrary links?
\end{problem}

\subsection{Unlinks in $S^1\times S^2$}

As an application of the above results, we can compute the Heegaard Floer homology of certain links in $S^1\times S^2$ and prove Theorem \ref{thm: intro Zn}. 
 Let $Z_n$ be the union of the $n$ parallel copies of $S^1$ inside $S^1\times S^2$. Recall from Section \ref{subsec: gen crossing change} that we have  generalized crossing change maps
$$
\phi_k^n: \cHFL(O_n)\to \cHFL(T(n,n))
$$
defined by attaching a 2-handle along the $(-1)$-framed meridian. The index $k$ corresponds to the choice of a $\Spin$ structure on the corresponding cobordism, see Proposition \ref{prop:fulltwist}. The surgery exact sequence immediately implies the following:

\begin{lemma}
\label{lem: exact seq Zn}
There is a long exact sequence 
$$
\rightarrow \cbHFL(S^3,O_n)\xrightarrow{\Phi} \cbHFL(S^3,T(n,n))\rightarrow \cbHFL(S^1\times S^2,Z_n)\rightarrow \cbHFL(S^3,O_n)\rightarrow
$$
where $\Phi=\sum_{k\in \Z}\phi_k^n$ is the  sum of $\phi_k^n$ over all $\Spin$ structures.
\end{lemma}

\begin{remark}
As above, if we work with integer coefficients  then $\phi_n^k$ would acquire some signs in the sum. Instead of guessing the signs, we simply work over $\F$.
\end{remark}

Define two series
$$
\mu_0=\sum_{k=0}^{\infty}(V_1\cdots V_n)^{k}\UU^{\frac{k(k-1)}{2}},\quad \mu_{n-1}=\sum_{k=0}^{\infty}(U_1\cdots U_n)^{k}\UU^{\frac{k(k-1)}{2}}.
$$

\begin{theorem}
The homology $\cbHFL(S^1\times S^2,Z_n)$ admits two equivalent descriptions:

a) $$\cbHFL(S^1\times S^2,Z_n)\simeq \frac{\cbHFL(T(n,n))}{(\mu_0a_0+a_1+ \ldots +a_{n-2}+\mu_{n-1}a_{n-1})}$$ where $a_i$ are the generators of $\cbHFL(T(n,n))$ from Theorem \ref{thm:toruslink}.

b) $\cbHFL(S^1\times S^2,Z_n)\simeq \bJ/(\gamma)$ where $\bJ$ is the (completed) determinantal ideal from Theorem \ref{thm: image Psi} and
$$
\gamma=\mu_0\prod_{i<j}(V_i-V_j)+\mu_{n-1}\prod_{i<j}(U_i-U_j)+
\sum_{j=1}^{n-2} \det\left(\begin{matrix}
U_1^{j} & \cdots & U_1 & 1 & V_1 & \cdots & V_1^{n-1-j}\\
\vdots &  &\vdots & \vdots & \vdots & & \vdots \\
U_n^{j} & \cdots & U_n & 1 & V_n & \cdots & V_n^{n-1-j}
\end{matrix}
\right).
$$
\end{theorem}

The second part of the theorem implies Theorem \ref{thm: intro Zn}. 

\begin{proof}
a) By Lemma \ref{lem: exact seq Zn} we can write 
$$
\cbHFL(S^1\times S^2,Z_n)\simeq \Cone\left[\cbHFL(S^3,O_n)\xrightarrow{\Phi} \cbHFL(S^3,T(n,n))\right]
$$
Since all $\phi_k^n$ are injective on homology and have pairwise different Alexander degrees, $\Phi=\sum_{k\in \Z}\phi_k^n$  is injective as well, and we can write
$$
\cbHFL(S^1\times S^2,Z_n)\simeq \cbHFL(S^3,T(n,n))/\Imm(\Phi).
$$
Now part (a) follows from Example \ref{ex:toruslink-phi} and Proposition \ref{prop:toruslink-ends}.

b) By the proof of Theorem \ref{thm: image Psi} the  map $\Omega^{0}$ provides an isomorphism $\cHFL(S^3,T(n,n))\simeq \J$ and $$\Omega^{0}(a_i)=\pm\det\left(\begin{matrix}
U_1^{j} & \cdots & U_1 & 1 & V_1 & \cdots & V_1^{n-1-j}\\
\vdots &  &\vdots & \vdots & \vdots & & \vdots \\
U_n^{j} & \cdots & U_n & 1 & V_n & \cdots & V_n^{n-1-j}
\end{matrix}
\right),\quad 0\le i\le n-1, 
$$
in particular
$$
\Omega^{0}(a_0)=\pm\det\left(\begin{matrix}
1 & V_1 & \cdots & V_1^{n-1}\\
 \vdots & \vdots & & \vdots \\
 1 & V_n & \cdots & V_n^{n-1}
\end{matrix}
\right)=\pm \prod_{i<j}(V_i-V_j)
$$
and 
$$
\Omega^{0}(a_{n-1})=\pm\det\left(\begin{matrix}
U_1^{n-1} & \cdots & U_1 & 1  \\
\vdots &  &\vdots & \vdots  \\
U_n^{n-1} & \cdots & U_n & 1 
\end{matrix}
\right)=\pm \prod_{i<j}(U_i-U_j).
$$
Therefore 
$$
\Omega^{0}(\mu_0a_0+a_1+ \ldots +a_{n-2}+\mu_{n-1}a_{n-1})=\gamma
$$
and the result follows.
\end{proof}

\begin{remark}
It would be interesting to compare this result to the recent work of Kronheimer and Mrowka \cite{KM}. Their main result expresses the (deformed) instanton homology $I(Z_n,\Gamma)$ of the $n$-component unlink $Z_n$ in $S^1\times S^2$ (with local coefficients) as a quotient of the polynomial ring by a certain determinantal ideal $\J_{inst}$.  
\end{remark}

\end{document}